\documentclass[a4paper,11pt]{amsart}

\usepackage{hyperref}
\usepackage{amssymb, amsthm, amsmath, thmtools, mathtools}
\usepackage{tikz}

\usepackage{bbm}

\usepackage[hmargin=3.5cm,foot=0.7cm]{geometry}
\usepackage[utf8x]{inputenc}

\usepackage[english]{babel}
\usepackage[abbrev]{amsrefs}

\usepackage{tikz}
\usepackage{tikz-cd}
\usetikzlibrary{matrix,arrows}
\usepackage{pgfplots}
\pgfplotsset{width=7cm,compat=1.8}
\usepgfplotslibrary{polar}
\allowdisplaybreaks[1]

\usepackage{enumitem}
\usepackage{verbatim}

\usepackage{comment}

\excludecomment{extracalc}

\declaretheorem[name=Theorem,numberwithin=section]{theorem}
\declaretheorem[name=Remark,style=remark,sibling=theorem]{remark}
\declaretheorem[name=Lemma,sibling=theorem]{lemma}
\declaretheorem[name=Proposition,sibling=theorem]{proposition}
\declaretheorem[name=Definition,style=definition,sibling=theorem]{definition}
\declaretheorem[name=Corollary,sibling=theorem]{corollary}

\declaretheorem[name=Theorem,numbered=no]{theorem*}
\declaretheorem[name=Corollary,numbered=no]{corollary*}
\declaretheorem[name=Question,numbered=no]{question*}
\declaretheorem[style=remark,name=Remark,numbered=no]{remark*}
\declaretheorem[style=definition,name=Definition,numbered=no]{definition*}
\declaretheorem[style=definition,name=Notation,numbered=no]{notation*}
\declaretheorem[name=Lemma,numbered=no]{lemma*}
\newcounter{mt}

\declaretheorem[name=Theorem,sibling=mt]{MainTheorem}
\declaretheorem[name=Corollary,sibling=mt]{MainCorollary}

\makeatletter
\newcommand{\tpitchfork}{%
	\vbox{
		\baselineskip\z@skip
		\lineskip-.52ex
		\lineskiplimit\maxdimen
		\m@th
		\ialign{##\crcr\hidewidth\smash{$-$}\hidewidth\crcr$\pitchfork$\crcr}
	}%
}
\makeatother

\BibSpec{article}{%
	+{}{\PrintAuthors}  		{author}
	+{,}{ \textit}     		{title}
	+{,}{ }             		{journal}
	+{}{ \textbf}       		{volume}
	+{}{ \parenthesize} 		{date}
	+{}{, no. } 		{number}
	+{,}{ }      	      		{conference}
	+{,}{ }      	      		{book}
	+{,}{ }            		{pages}
	+{,}{ }            	 	{note}
	+{,}{ }            	 	{status}
	+{,}{  \texttt } {eprint}
	+{.}{}              {transition}
}

\BibSpec{book}{%
	+{}{\PrintAuthors}  {author}
	+{,}{ \textit}      {title}
	+{,}{ }             {publisher}
	+{,}{ }             {place}
	+{,}{ }             {date}
	+{.}{}              {transition}
}

\BibSpec{incollection}{%
	+{}  {\PrintAuthors}                {author}
	+{,} { \textit}                     {title}
	+{.} { }                            {part}
	+{:} { \textit}                     {subtitle}
	+{,} { \PrintContributions}         {contribution}
	+{,} { \PrintConference}            {conference}
	+{,} { }                            {booktitle}
	+{,} { pp.~}                        {pages}
	+{,}{ }       		{series}
	+{}{ \textbf}       		{volume}
	+{,} { }                            {publisher}
	+{,} { }                 {date}
	+{,} { }                            {status}
	+{,} { \PrintDOI}                   {doi}
	+{,} { available at \eprint}        {eprint}
	+{}  { \parenthesize}               {language}
	+{}  { \PrintTranslation}           {translation}
	+{;} { \PrintReprint}               {reprint}
	+{.} { }                            {note}
	+{.} {}                             {transition}
}

\usepackage{xcolor}
\newcommand{\CC}{\mathbb{C}}

\newcommand{\NN}{\mathbb{N}}

\newcommand{\RR}{\mathbb{R}}

\newcommand{\ZZ}{\mathbb{Z}}

\newcommand{\Iind}{\mathcal{I}}
\newcommand{\I}{\sqrt{-1}}

\newcommand{\pair}[2]{\left\langle #1, #2 \right\rangle}

\newcommand{\dt}{\left.\frac{d}{dt}\right|_{t=0}}

\DeclareMathOperator{\GL}{GL}
\DeclareMathOperator{\SL}{SL}
\DeclareMathOperator{\GW}{GW}

\newcommand{\gl}{\mathfrak{gl}}
\newcommand{\sln}{\mathfrak{sl}}
\newcommand{\so}{\mathfrak{so}}

\DeclareMathOperator{\Val}{Val}
\DeclareMathOperator{\PVal}{PVal}
\DeclareMathOperator{\SO}{SO}
\DeclareMathOperator{\OO}{O}

\DeclareMathOperator{\vol}{vol}
\DeclareMathOperator{\id}{id}
\DeclareMathOperator{\supp}{supp}

\DeclareMathOperator{\Sym}{Sym}
\DeclareMathOperator{\sign}{sign}

\DeclareMathOperator{\im}{im}

\DeclareMathOperator{\tr}{\mathrm{tr}}

\DeclareMathOperator{\Ad}{Ad}

\DeclareMathOperator{\nc}{nc}

\DeclareMathOperator{\vsupp}{v\text{-}supp}

\newcommand{\K}{\mathcal{K}}
\renewcommand{\S}{\mathbb{S}}

\newcommand{\LieDer}{\mathcal{L}}

\newcommand{\Conv}{\mathrm{Conv}}

\newcommand{\VConv}{\mathrm{VConv}}

\title[Localization of valuations and Alesker's irreducibility theorem]{Localization of valuations and\\ Alesker's irreducibility theorem}
\author{Georg C. Hofst\"atter}
\address{Institute f. Discrete Mathematics and Geometry, TU Wien, 1040 Vienna, Austria}
\email{georg.hofstaetter@tuwien.ac.at}
\author{Jonas Knoerr}
\address{Institute f. Discrete Mathematics and Geometry, TU Wien, 1040 Vienna, Austria}
\email{jonas.knoerr@tuwien.ac.at}

\thanks{
	\\\indent
	The research of the first-named author was funded in part by the Austrian Science Fund (FWF), 10.55776/ESP9378724. The research of the second-named author was funded in part by the Austrian Science Fund (FWF), 10.55776/PAT4205224}
\begin{document}

\begin{abstract}
	We provide a new proof of Alesker's Irreducibility Theorem. We first introduce a new localization technique for polynomial valuations on convex bodies, which we use to independently prove that smooth and translation invariant valuations are representable by integration with respect to the normal cycle. This allows us to reduce the statement to a corresponding result for the representation of $\mathfrak{sl}(n)$ on the space of these differential forms.
\end{abstract}

\maketitle

\tableofcontents

\section{Introduction}
Let $\K(\RR^n)$ denote the space of convex bodies, that is, the set of all nonempty, convex, and compact subsets of $\RR^n$ equipped with the Hausdorff metric. A map $\varphi: \K(\RR^n) \to \CC$  is called a \emph{valuation} if
\begin{align*}
	\varphi(K \cup L) + \varphi(K \cap L) = \varphi(K) + \varphi(L)
\end{align*}
whenever $K, L, K \cup L \in \K(\RR^n)$. This notion goes back to Dehn's solution to Hilbert's third problem on the non-equidecomposability of convex polytopes, but has since become a very active area of research with a variety of applications to geometric problems, in particular in integral geometry, see, e.g., \cites{Bernig2011, Bernig2014,Bernig2021,Alesker2003, Alesker2009, Bernig2022, Faifman2023,Kotrbaty2020,Kotrbaty2022,Kotrbaty2022c, Bernig2023} for an overview. This relation was already established by Hadwiger \cite{Hadwiger1957} using his famous characterization of continuous and rigid motion invariant on convex bodies, but has seen significant progress over the last 25 years.\\

Most of these advances are based on the foundational work by Alesker \cite{Alesker2001} on the structure of the space of continuous translation invariant valuations on $\mathcal{K}(\RR^n)$, which we denote by $\Val(\RR^n)$. This space admits a natural decomposition, as shown by McMullen~\cite{McMullen1977}: If we denote by $\Val_r(\RR^n)$ the subspace of all $r$-homogeneous valuations, that is, all $\varphi\in\Val(\RR^n)$ such that $\varphi(tK)=t^r\varphi(K)$ for $t\ge 0$, $K\in\mathcal{K}(\RR^n)$, then
\begin{align*}
	\Val(\RR^n)=\bigoplus_{r=0}^n\Val_r(\RR^n).
\end{align*}
We may further decompose these spaces as $\Val_r(\RR^n)=\Val^+_r(\RR^n)\oplus \Val^-_r(\RR^n)$ into even and odd valuations, where $\varphi\in\Val(\RR^n)$ is called even resp. odd if $\varphi(-K)=\pm\varphi(K)$ for $K\in\mathcal{K}(\RR^n)$.\\

Since $\mathcal{K}(\RR^n)$ is a locally compact metric space, $\Val(\RR^n)$ naturally carries the structure of a Fr\' echet space (it is in fact a Banach space). We have a natural continuous representation of the general linear group $\GL(n,\RR)$ on this space, given by
\begin{align*}
	(g \cdot \varphi)(K) = \varphi(g^{-1}K), \quad K \in \K(\RR^n),
\end{align*}
where $g \in \GL(n,\RR)$ and $\varphi \in \Val(\RR^n)$. The following result is now known as Alesker's Irreducibility Theorem.
\begin{theorem}[\cite{Alesker2001}]\label{thm:AleskerIrredThm}
	For $0 \leq r \leq n$, the spaces $\Val_r^+(\RR^n)$ and $\Val_r^-(\RR^n)$ are topologically irreducible $\GL(n,\RR)$-representations, that is, every nontrivial $\GL(n,\RR)$-invariant subspace is dense.
\end{theorem}
Let us add some comments on this result. It was originally used to obtain an affirmative solution of a conjecture by McMullen \cite{McMullen1980} that mixed volumes (see \cite{Schneider2014} for the definition) span a dense subspace of $\Val_{r}(\RR^n)$. This conjecture was known to hold for $r=0$ and $r=n$, since $\Val_r(\RR^n)$ is $1$-dimensional and spanned by the Euler characteristic and the Lebesgue measure \cite{Hadwiger1957} respectively in this case, as well as for $r=1$ and $r=n-1$, which follows from results by Goodey and Weil \cite{Goodey1984} and McMullen \cite{McMullen1980} respectively. The Irreducibility \autoref{thm:AleskerIrredThm} directly implies the conjecture for all $0\le r\le n$, since linear combinations of mixed volumes form a $\GL(n,\RR)$-invariant subspace of $\Val(\RR^n)$ that intersects $\Val_r^\pm(\RR^n)$ nontrivially. In fact, it can be used to obtain much stronger versions of the conjecture -- it is sufficient to take simplices or ellipsoids (in the even case) as reference bodies for the mixed volumes \cite{Alesker2001}, or combinations of mixed volumes with a bounded number of terms \cite{Knoerr2023d}.\\

Alesker's proof of \autoref{thm:AleskerIrredThm} relies on the existence of three embeddings due to Goodey and Weil \cite{Goodey1984}, Klain \cite{Klain2000}, and Schneider \cite{Schneider1996} in order to apply highly sophisticated tools from the representation theory of real reductive groups. In particular, the proof uses the Beilinson--Bernstein localization theorem to reduce the statement to some purely algebraic computations for certain associated $\mathcal{D}$-modules.\\

In this article, we present a new proof of \autoref{thm:AleskerIrredThm} that relies on more elementary methods. While we hope that this makes our approach accessible to non-experts, this does have the cost of requiring some rather tedious (although simple) calculations involving differential forms, see \autoref{sec:actLieAlgdf}. These calculations do, however, carry slightly more information on the representation of $\GL(n,\RR)$ on $\Val_r^\pm(\RR^n)$ than \autoref{thm:AleskerIrredThm} since they directly relate different $\SO(n)$-irreducible subrepresentations by the action of the Lie algebra $\sln(n)$ of the special linear group $\SL(n,\RR)$. This might have further applications.\\

Let us discuss our approach. First recall that a valuation $\varphi\in\Val(\RR^n)$ is called smooth (or more precisely, $\GL(n,\RR)$-smooth) if the map
\begin{align*}
	\GL(n,\RR)&\rightarrow \Val(\RR^n)\\
	g &\mapsto g \cdot \varphi		
\end{align*}
is a smooth map, i.e.\ if $\varphi$ is a smooth vector of the representation of $\GL(n,\RR)$ on $\Val(\RR^n)$. Since $\Val(\RR^n)$ is complete, smooth valuations form a dense subspace of $\Val(\RR^n)$, compare the discussion in \autoref{sec:infDimRep}. We denote the corresponding subspace by $\Val^\infty(\RR^n)$.\\

Smooth valuations can be constructed by integration with respect to the normal cycle $\nc(\cdot)$, which is defined for $K \in \K(\RR^n)$ by
\begin{align*}
	\nc(K) = \{(x,v):\, x \in K, \pair{y}{v} \leq \pair{x}{v} \, \forall y \in K\}\subset S\RR^n,
\end{align*}
where $S\RR^n=\RR^n\times \S^{n-1}$ denotes the sphere bundle, see \autoref{sec:ValRepByIntNC}. Any smooth differential $(n-1)$-form $\omega\in \Omega^{n-1}(S\RR^n)$ with complex coefficients then induces a continuous valuation on $\mathcal{K}(\RR^n)$ by
\begin{align*}
	K\mapsto \int_{\nc(K)}\omega,
\end{align*} 
compare \cite{Alesker2006b}. We will call valuations of this type \emph{representable by integration with respect to the normal cycle} and denote the corresponding space of continuous valuations by $\mathcal{V}^{\infty}$. Note that such a valuation is not necessarily translation invariant. It follows from \cite{Alesker2006b}*{Prop. 5.1.9} that the subset of $\Val(\RR^n)$ given by
\begin{align*}
	\mathcal{V}^{\infty,\mathrm{tr}}:=\mathcal{V}^{\infty}\cap \Val(\RR^n)
\end{align*}
is a subset of $\Val^\infty(\RR^n)$. In fact, Alesker showed in \cite{Alesker2006b}*{Thm.~5.2.1} that these two spaces coincide, i.e.\ that every $\GL(n,\RR)$-smooth valuation in $\Val(\RR^n)$ is representable by integration with respect to the normal cycle (and for homogeneous valuations of degree $0\le r\le n-1$, the differential  form can also be chosen to be translation invariant). However, this result is obtained by combining the Irreducibility \autoref{thm:AleskerIrredThm} with the Casselman--Wallach Theorem \cite{Casselman1989}, and thus not directly available to us. Indeed, we establish this characterization independently, which is our first main result.
\begin{MainTheorem}\label{mthm:smoothValsByDiffform}
	Suppose that $\varphi \in \Val(\RR^n)$ is $\GL(n,\RR)$-smooth. Then there exists a  differential form $\omega \in \Omega^{n-1}(S\RR^n)$ such that
	\begin{align}\label{eq:defSmValDiffForm}
		\varphi(K) = \int_{\nc(K)} \omega, \quad K \in \K(\RR^n).
	\end{align}
\end{MainTheorem}

The proof of \autoref{mthm:smoothValsByDiffform} is based on a corresponding regularity result for certain spaces of valuations on convex functions obtained by the second named author in \cite{Knoerr2025} using a Paley--Wiener--Schwartz-type characterization result for certain distributions associated to these valuations. In particular, the results in \cite{Knoerr2025} do not use the Irreducibility \autoref{thm:AleskerIrredThm}.\\

We then exploit that these spaces of valuations on convex functions can be related to translation invariant valuations on convex bodies that satisfy restrictions on their vertical support (see \autoref{sec:bghomDecoGW} for the definition), as was investigated in \cites{Knoerr2020a,Knoerr2020b}. Not every valuation in $\Val(\RR^n)$ satisfies these conditions, but we are going to show that any valuation can be written as a sum of valuations satisfying the support restrictions using a localization procedure introduced in \autoref{sec:localization}. This procedure breaks the translation invariance, so we need to work on a larger space of valuations.

Let us call a valuation $\varphi:\mathcal{K}(\RR^n)\rightarrow\CC$  a  \emph{polynomial valuation} if $x \mapsto \varphi(K+x)$ is a polynomial in $x \in \RR^n$ for every $K \in \K(\RR^n)$, where the degree of $\varphi(K+x)$ is uniformly bounded, i.e. independent of $K$. Our second main result is a ''partition of unity``-type result for polynomial valuations.
\begin{MainTheorem}\label{mthm:locVal}
	Let $\varphi:\mathcal{K}(\RR^n)\rightarrow\CC$ be a continuous polynomial valuation and $(U_\alpha)_{\alpha\in\mathcal{A}}$  an open cover of $\S^{n-1}$. Then there exist $M\in\NN$ and continuous polynomial valuations $\varphi_1,\dots,\varphi_M:\mathcal{K}(\RR^n)\rightarrow\CC$ such that the vertical supports $\vsupp\varphi_j$, $1\leq j\leq M$, are subordinate to the cover $(U_\alpha)_{\alpha\in\mathcal{A}}$ and such that
	\begin{align*}
		\varphi=\varphi_1+\dots+\varphi_M.
	\end{align*}
\end{MainTheorem}
Note that while this result follows easily for valuations that are representable by integration over the normal cycle, this is nontrivial for general continuous valuations. Moreover, let us point out that for $\Val_n(\RR^n)$ a decomposition into translation-invariant valuations with these support restrictions is not possible.\\

\autoref{mthm:smoothValsByDiffform} allows us to prove \autoref{thm:AleskerIrredThm} by working with the corresponding action on differential forms. More precisely, we show a corresponding result for the space 
\begin{align*}
	\mathcal{V}_{r,\pm}^{\infty,\mathrm{tr}}:=\mathcal{V}^{\infty,\mathrm{tr}}\cap \Val^\pm_r(\RR^n)
\end{align*}
under the action of the special linear group $\SL(n,\RR)$. This space was investigated from a purely differential geometric viewpoint in \cite{Bernig2007}, which was then used in \cite{Alesker2011} to obtain its decomposition into its $\SO(n)$-isotypical components. The corresponding highest weight vectors were constructed in \cite{Kotrbaty2022}. The relevant parts of these articles do not use the Irreducibility \autoref{thm:AleskerIrredThm}, compare the discussion below.\\

The space $\left(\mathcal{V}_{r,\pm}^{\infty,\mathrm{tr}}\right)^{\SO(n)-\mathrm{fin}}$ of $\SO(n)$-finite vectors (see \autoref{sec:infDimRep} for the definition) carries a joint action of the Lie algebra $\sln(n)$ of $\SL(n,\RR)$ and $\SO(n)$, which gives it the structure of an $(\sln(n),\SO(n))$-module. By calculating the action of certain distinguished elements in the complexification $\sln(n)_\CC:=\sln(n)\otimes\CC$ on the highest weight vectors, we establish the following result.
\begin{MainTheorem}\label{mthm:imDIrred}
	Let $n\ge 3$, $1 \leq r \leq n-1$. Then the subspaces $(\mathcal{V}_{r,\pm}^{\infty,\mathrm{tr}})^{\SO(n)-\mathrm{fin}}$
	of $\Val^\infty(\RR^n)$ are algebraically irreducible $(\mathfrak{sl}(n),\SO(n))$-modules. 
\end{MainTheorem}
Let us point out that the notion of irreducibility in \autoref{mthm:imDIrred} differs from the previous one: an $(\mathfrak{sl}(n),\SO(n))$-module is algebraically irreducible if it does not contain any nontrivial subspace that is invariant under the joint action of $\mathfrak{sl}(n)$ and $\SO(n)$. Note that it follows from general results from representation theory that this statement is in general stronger than topological irreducibility of the associated representation of $\SL(n,\RR)$, however, for the type of representation considered here these notions are in fact equivalent, see the discussion in \autoref{sec:infDimRep}. \autoref{mthm:smoothValsByDiffform} and \autoref{mthm:imDIrred} together therefore imply the following version of \autoref{thm:AleskerIrredThm}.
\begin{MainCorollary}\label{mcor:ValslIrred}
	Let $n\ge 3$, $1 \leq r \leq n-1$. Then the spaces $\Val_r^\pm(\RR^n)$ are topologically irreducible representations of $\SL(n,\RR)$.
\end{MainCorollary}

Let us point out that for $n=2$, the two previous results are both incorrect, since the corresponding representations on $1$-homogeneous odd valuations contain two nonisomorphic irreducible subrepresentations of $\SL(2,\RR)$ (as pointed out in \cite{Alesker2011b}). We discuss this case separately in \autoref{sec:prfAlirred}.\\

As a final remark, we note that, since they are not used in the proof, the characterization results by Klain~\cite{Klain2000} and Schneider~\cite{Schneider1996} can be obtained as a simple consequence of \autoref{mcor:ValslIrred}.

\subsection*{Discussion of relevant results for the proof}
As our proof of \autoref{mcor:ValslIrred} relies heavily on previous results, let us give a short overview.

For the first step, \autoref{mthm:locVal}, we need homogeneous decompositions, polarizations and Goodey--Weil distributions for polynomial valuations (\cites{Goodey1984,Khovanskii1993,Knoerr2024,McMullen1975,McMullen1977,Alesker2000}). \autoref{mthm:smoothValsByDiffform} is then deduced using a construction from \cite{Knoerr2020a} from an analogous, highly nontrivial result for valuations on convex functions, proved in \cite{Knoerr2025}.

Next, we heavily use that the space $\mathcal{V}^{\infty,\mathrm{tr}}$ was studied in \cites{Bernig2007,Alesker2011,Kotrbaty2022}. In \cite{Bernig2007}, all results are actually shown for $\mathcal{V}^{\infty}$ and the Irreducibility \autoref{thm:AleskerIrredThm} is only used to relate the results to $\Val^\infty(\RR^n)$. The same is true for the relevant sections in \cite{Alesker2011} (Section~4 and 5) and \cite{Kotrbaty2022}. We will therefore state their results as results for $\mathcal{V}^{\infty,\mathrm{tr}}$ and not for $\Val^\infty(\RR^n)$ later on.

We then use a similar strategy as in \cite{Howe1999} to relate the different $\SO(n)$-isotypical components by the action of suitable Lie algebra elements from $\sln(n)_\CC$ on the highest weight vectors. By an explicit calculation, we show that the Lie algebra action may be used to pass from a given highest weight vector to all representations with larger highest weight. In order to pass to representations with smaller highest weight, we use an $\SL(n,\RR)$-invariant pairing on the relevant spaces of differential forms in order to argue by duality. This pairing coincides with the pairing induced by the Alesker product~\cite{Alesker2004b}, however, we will not rely on this fact and instead establish the necessary properties directly.

\subsection*{Plan of the article}
\autoref{mthm:locVal} and \autoref{mthm:smoothValsByDiffform} are proved in \autoref{part1}, while \autoref{mthm:imDIrred} and \autoref{mcor:ValslIrred} are proved in \autoref{part2}. \autoref{part1} is independent of \autoref{part2}, and can be read separately. Necessary background is reviewed in \autoref{sec:bg}. The relation between the different notions of irreducibility for infinite dimensional representations is discussed in \autoref{sec:infDimRep}.

\section{Background material and Preliminaries}\label{sec:bg}
In this section we recall some background from representation theory, valuation theory and differential geometry, as well as fix the notation. Moreover, we will prove some preliminary results used in the proofs of the main theorems.

	\subsection{Representation Theory -- Highest weight vectors}\label{sec:HWV}
	Throughout this section, we assume that $n\ge 3$.\\
	
	In the proof of \autoref{mthm:imDIrred}, we will use the decomposition of the given representations into its $\SO(n)$-isotypical components. This decomposition was determined in \cite{Alesker2011} and the corresponding highest weight vectors were constructed in \cite{Kotrbaty2022}. These isotypical components are parametrized by their highest weights, given as tuples of length $l = \lfloor\frac{n}{2}\rfloor$, and contain (after some choices) distinguished vectors called highest weight vectors. We will later use that it is essentially enough to know the $\SL(n,\RR)$-representation on these highest weight vectors in order to show irreducibility.
	
	As a highest weight vector depends on some extra structure on the representation (in particular, a choice of a Cartan subalgebra and a set of positive weights), we will shortly summarize this here. For an introduction to the theory of highest weights, we refer to \cite{Knapp1986}*{Ch.~IV}. More details on our setting can be found in \cite{Kotrbaty2022}, where the highest weight vectors of the representations of smooth valuations are constructed.
	
	\medskip
	
	First, assume that $n=2l$ is even. Then we choose the subgroup of $\SO(n)$ as maximal torus consisting of block-diagonal matrices, with $2\times2$ blocks of the form
	\begin{align*}
		\begin{pmatrix}
			\cos(t) & -\sin(t)\\
			\sin(t) & \cos(t)
		\end{pmatrix}, \quad \text{ for some } t \in \RR.
	\end{align*}
	As this subgroup is abelian and maximal, its Lie algebra $\mathfrak{t}$ is also an abelian and maximal subalgebra of $\so(n)$, that is, a Cartan subalgebra. A basis of its complexification $\mathfrak{t}_\CC$ is given by the elements $H_1, \dots, H_l$, where 
	\begin{align*}
		H_i = \sqrt{-1} E_{2i-1,2i} - \sqrt{-1}E_{2i,2i-1},
	\end{align*}
	and where $E_{i,j} \in \CC^{n \times n}$ denotes the matrix with $1$ at position $(i,j)$ and zeros everywhere else. We denote the dual basis to $\{H_1, \dots, H_l\}$ by $\epsilon_1, \dots, \epsilon_l \in \mathfrak{t}_\CC^\ast$.
	
	The elements of $\mathfrak{t}_\CC$ act on $\so(n)_\CC$ by the adjoint representation $\mathrm{ad}$, that is, for $H \in \mathfrak{t}_\CC$, $\mathrm{ad}_H(X) = [H,X] =HX - XH$, $X \in \so(n)_\CC$. As $\mathfrak{t}_\CC$ is an abelian Lie group, the operators $\mathrm{ad}_{H}$, $H \in \mathfrak{t}_\CC$, all commute and their eigenvalues depend linearly on $H$. A short calulation shows that these eigenvalues, called roots of $\so(n)$, are given by
	\begin{align*}
		\Delta = \{\pm \epsilon_i \pm \epsilon_j: \, 1 \leq i < j \leq l\} \cup \{0\} \subset \mathfrak{t}_\CC^\ast.
	\end{align*}	
	The Lie algebra $\so(n)_\CC$ itself decomposes into common eigenspaces $\so(n)_\alpha$, $\alpha \in \Delta$, called root spaces. We declare the set of positive roots $\Delta^+$ to be
	\begin{align*}
		\Delta^+ = \{\epsilon_i \pm \epsilon_j: \, 1 \leq i < j \leq l\} \subset \Delta,
	\end{align*}
	and write $\mathfrak{n}^\pm = \bigoplus_{\alpha \in \Delta^+} \so(n)_{\pm \alpha}$. Note that $\Delta^+$ induces an ordering of the roots, and that $\so(n)_\CC = \mathfrak{n}^- \oplus \mathfrak{t}_\CC\oplus \mathfrak{n}^+$.
	
	Suppose now that $(V,\pi)$ is an irreducible representation of $\SO(n)$. Then $V$ decomposes into a family of common eigenspaces of the operators $d\pi(H)$, $H \in \mathfrak{t}_\CC$. The appearing eigenvalues (elements of $\mathfrak{t}_\CC^\ast$) are called the weights of the representation. The theorem of the highest weight (see, e.g., \cite{Knapp1986}*{Thm.~4.28}) then states that there exists a unique highest weight that already uniquely determines the representation.	Its weight space is one-dimensional, the non-zero vectors are called highest weight vectors of the representation $(V, \pi)$. It turns out that a highest weight vector $v \in V$ with highest weight $\lambda \in \mathfrak{t}_\CC^\ast$ is characterized by
	\begin{align}\label{eq:defHWV}
		d\pi(H)v = \lambda(H) v, \text{ for all } H \in \mathfrak{t}_\CC, \quad \text{ and } \quad d\pi(X) v = 0, \text{ for all } X \in \mathfrak{n}^+.
	\end{align}
	Let us point out that if $V$ and $W$ are irreducible representations with highest weight vectors $v$ and $w$ (respectively), then \eqref{eq:defHWV} shows that $v \otimes w \in V \otimes W$ is a highest weight vector of some irreducible component of $V \otimes W$ and its weight is given by the sum of the corresponding weights.
	
	As a weight $\lambda = \sum_{i =1}^l \lambda_i \epsilon_i$ is identified with the tuple $(\lambda_1, \dots, \lambda_l)$, all irreducible representations of $\SO(n)$ (up to isomorphism) can be indexed by such tuples $(\lambda_1, \dots, \lambda_l) \in \ZZ^l$. It can be shown that $\SO(n)$ has irreducible representations exactly for the following set $\Lambda$ of highest weights:
	\begin{align*}
		\Lambda = \{(\lambda_1, \dots, \lambda_l) \in \ZZ^l: \, \lambda_1 \geq \lambda_2 \geq \dots \geq \lambda_{l-1} \geq |\lambda_l| \}, \quad n = 2l.
	\end{align*}
	
	\medskip
	
	If $n=2l+1$ is odd, then the maximal torus of $\SO(n)$ is chosen to consist of all block-diagonal matrices with $2 \times 2$ blocks as before in the the first $n-1=2l$ (standard) coordinates and a $1$ in the $(n,n)$-entry. As before, we denote by $\mathfrak{t}_\CC$ the corresponding Cartan subalgebra, with basis given by $H_1, \dots, H_l$ and dual basis $\epsilon_1, \dots, \epsilon_l$. The set of roots is now given by $\Delta = \{\pm \epsilon_i \pm \epsilon_j: \, 1 \leq i < j \leq l\} \cup \{\pm \epsilon_i: \, 1 \leq i \leq l\}$, and we define the sets of positive roots to be
	\begin{align*}
		\Delta^+ = \{\epsilon_i \pm \epsilon_j: \, 1 \leq i < j \leq l\} \cup \{ \epsilon_i: \, 1 \leq i \leq l\} \subset \mathfrak{t}_\CC^\ast.
	\end{align*}
	We can therefore represent all highest weights of irreducible $\SO(n)$ representations by tuples $(\lambda_1, \dots, \lambda_l) \in \ZZ^l$ from the index set
	\begin{align*}
		\Lambda = \{(\lambda_1, \dots, \lambda_l) \in \ZZ^l: \, \lambda_1 \geq \lambda_2 \geq \dots \geq \lambda_l \geq 0 \}, \quad n = 2l+1.
	\end{align*}

	\subsection{Valuations representable by integration with respect to the normal cycle}\label{sec:ValRepByIntNC}
	In this section, we recall the construction of valuations using differential forms on the sphere bundle $S\RR^n = \RR^n \times \S^{n-1}$ of $\RR^n$ from \cite{Alesker2008}.
	
	First, recall that the normal cycle $\nc(K)$ of $K \in \K(\RR^n)$ is defined as the set of pairs $(x,v)$ where $x$ is a boundary point of $K$ and $v$ is an outer unit normal at $x$, that is, $\nc(K) = \{(x,v) \in S\RR^n:\, x \in K, \pair{y}{v} \leq \pair{x}{v} \, \forall y \in K\}$. Then $\nc(K)$ is an $(n-1)$-dimensional Lipschitz submanifold of $S\RR^n$ with a natural orientation induced from the orientation of $\RR^n$. Moreover, if we interpret $\nc(K)$ as an integral current, then the map $K \mapsto \nc(K)$ is a continuous valuation with respect to the flat metric topology (see \cite{Alesker2008}*{Sec.~2} for the definition). This implies in particular that $K\to \int_{\nc(K)}\omega$ defines a continuous valuation for every smooth differential form $\omega\in \Omega^{n-1}(S\RR^n)$. We will call valuations of this type \emph{representable by integration with respect to the normal cycle}. Note that every transla\-tion-invariant differential form $\omega \in \Omega^{n-1}(S\RR^n)^{\mathrm{\mathrm{tr}}}$ induces a valuation $\varphi_\omega \in \Val(\RR^n)$. Moreover, valuations of this type are smooth vectors of the representation of $\GL(n,\RR)$ on $\Val(\RR^n)$, compare \cite{Alesker2006b}*{Prop.~5.1.9}.
	
	The kernel of this procedure was determined in \cite{Bernig2007} and relies on the Rumin differential defined on the contact manifold  $S\RR^n$ (see \cite{Rumin1994} for the definition on a general contact manifold). Recall that $S\RR^n$ is a contact manifold of dimension $2n-1$ with (global) contact form $\alpha \in \Omega^1(S\RR^n)$ given by
	\begin{align}\label{eq:defAlphSRn}
		\alpha|_{(x,v)}(X) = \pair{d\pi_1(X)}{v}, \quad (x,v) \in S\RR^n, X \in T_{(x,v)}S\RR^n,
	\end{align}
	where $\pi_1:S\RR^n \to \RR^n$ is the canonical projection on the first factor.
	
	One can show that for $\omega \in \Omega^{n-1}(S\RR^n)$, there exists $\xi \in \Omega^{n-2}(S\RR^n)$ such that $d(\omega + \alpha \wedge \xi)$ is a multiple of $\alpha$. The form $\alpha\wedge\xi$ is uniquely determined by this property, and the Rumin differential $D\omega$ is then defined as
	\begin{align*}
		D\omega = d(\omega + \alpha \wedge \xi) \in \Omega^n(S\RR^n).
	\end{align*}
	Note that the Rumin differential only depends on the contact distribution $\ker \alpha$, not on the specific $1$-form $\alpha$. In particular, if $\Phi:S\RR^n\rightarrow S\RR^n$ is a contactomorpism, that is, a diffeomorphism satisfying $\Phi^*\alpha= f\alpha$ for some $f\in C^\infty(S\RR^n)$,  then $\Phi^\ast( D\omega) = D(\Phi^\ast \omega)$.
	\begin{theorem}[\cite{Bernig2007}*{Thm.~1}]\label{thm:BernigBroeckerKernel}
		Let $\omega \in \Omega^{n-1}(S\RR^n)$. Then $\int_{\nc(K)}\omega=0$ for all $K\in\K(\RR^n)$ if and only if 
		\begin{enumerate}
			\item $D \omega = 0$ and
			\item\label{it:0homBed} $\int_{\{x\}\times \S^{n-1}}\omega = 0$ for all $x\in\RR^n$.
		\end{enumerate}
	\end{theorem}
	If we restrict ourselves to translation invariant forms, this results simplifies further: The space of translation invariant differential $(n-1)$-forms on $S\RR^n$ naturally decomposes into a direct sum $\Omega^{n-1}(S\RR^n)^{\mathrm{tr}}=\bigoplus_{r=0}^{n-1}\Lambda^r(\RR^n)^*\otimes \Omega^{n-1-r}(\S^{n-1})$ of forms of a given bi-degree $(r,n-r-1)$.  If we consider only differential forms of fixed bi-degree $(r,n-r-1)$, $0 \leq r \leq n-1$, then the corresponding valuation is $r$-homogeneous. Hence, condition \eqref{it:0homBed} is equivalent to saying that the $0$-homogeneous component of the given valuation vanishes. Indeed, this condition corresponds to the evaluation of the valuation in the convex body $\{x\}$, whose normal cycle is $\{x\}\times \S^{n-1}$. In particular,  \eqref{it:0homBed} is only relevant for the component of $\omega$ with bi-degree $(0,n-1)$, since it is automatically satisfied for $1\leq r\leq n-1$.	

	\begin{remark}
		 \autoref{thm:BernigBroeckerKernel} can be used to show that the subspace of $\Val_r(\RR^n)$, $1\leq r\leq n-1$, consisting of valuations representable by integration with respect to the normal cycle is isomorphic to a subspace of the image of the Rumin differential. This subspace was described by Bernig and Bröcker in \cite{Bernig2007}*{Thm.~3.3}, compare \autoref{thm:BernigBroeckerImageRumin} below.
	\end{remark}
	
	\subsection{The representation of $\GL(n,\RR)$ on $\mathcal{V}^{\infty, \mathrm{tr}}$}\label{sec:GLRnrepOnVal}
	Let the group $\GL(n,\RR)$ act on $S\RR^n = \RR^n \times \S^{n-1}$ as follows: If $(x,v) \in S\RR^n$ and $g \in \GL(n,\RR)$, then
	\begin{align*}
		g \cdot(x,v) = G_g(x,v) := \left(gx, \frac{g^{-T} v}{\|g^{-T} v\|}\right).
	\end{align*}
	This action is defined such that $\nc(g K) = \sign(\det(g))(G_{g})_\ast(\nc(K))$ for all $g \in \GL(n,\RR)$ and $K \in \K(\RR^n)$, where we interpret $\nc(K)$ as a current and $G_{g*}(\nc(K))$ denotes the pushforward of currents.  Note that the additional sign reflects the change of orientation if $\det(g)<0$. Moreover, since $G_g^\ast \alpha|_{(x,v)} = \frac{1}{\|g^{-T} v\|} \alpha|_{(x,v)}$ for all $(x,v) \in S\RR^n$, $G_g$ is a contactomorphism for every $g\in \GL(n,\RR)$ and thus intertwines with the Rumin differential.
	
	For $W \in \gl(n)$, the Lie algebra of $\GL(n,\RR)$, consider the elements $\phi_t = \exp(-tW) \in \GL(n,\RR)$, $t \in \RR$, and the induced map on $S\RR^n$ given by
	\begin{align*}
		\psi_t(x,v) = \left(\phi_t(x), \frac{d\phi_t(x)^{-T} v}{\|d\phi_t(x)^{-T} v\|}\right).
	\end{align*}
	Then the fundamental vector field induced by $W$ is given by
	\begin{align}\label{eq:LieAlgactOnSRn}
	\widetilde{W}_{(x,v)}: = \dt \psi_t(x,v) = (-W x, W^T v - v \pair{v}{W^T v})\in T_{(x,v)}S\RR^n,
	\end{align}
	where we identify $T_{(x,v)}S\RR^n \cong \RR^n\oplus T_v\S^{n-1}$. Let $\varphi = \int_{\nc(\cdot)} \omega$ be a valuation representable by integration with respect to the normal cycle, where $\omega \in \Omega^{n-1}(S\RR^n)$. Then the action of $g \in \GL(n,\RR)$ is given by
	\begin{align}
		\label{eq:actGLOnVal}
		\begin{split}
			(g\cdot \varphi)(K) &= \varphi(g^{-1}K) = \int_{\nc(g^{-1}K)} \omega = \sign (\det g) \int_{G_{g^{-1}*}(\nc(K))}\omega\\
			&= \sign (\det g)\int_{\nc(K)} (G_{g^{-1}})^\ast \omega.
		\end{split}
	\end{align}
	The additional sign can be avoided if one twists the construction with the orientation bundle of $\RR^n$ (as for example in \cite{Alesker2006b}*{Section~5}), however, since we are mainly interested in the action of $\SL(n,\RR)$, where this additional sign does not play a role, we will omit this step.\\
	
	The infinitesimal action of $W \in \gl(n)$ is thus given by
	\begin{align*}
		(W\cdot \varphi)(K) = \dt \varphi(\exp(-tX) K) = \int_{\nc(K)} \dt(G_{\phi_t})^\ast \omega = \int_{\nc(K)} \LieDer_{\widetilde{W}}\omega,
	\end{align*}
	that is, $\gl(n)$ acts on the differential form by the Lie derivative with respect to the fundamental vector fields. We extend this operation to the complexification $\gl(n)_\CC$ of $\gl(n)$ by linearity.
	\subsection{An invariant pairing}\label{sec:pairing}
	In this section we examine a pairing on the image of the restriction of the Rumin operator to translation invariant forms. It follows from \cite{Bernig2007} that this pairing naturally corresponds to a pairing on a certain subspace of $\mathcal{V}^{\infty,\mathrm{tr}}$. In fact, it is shown in \cite{Bernig2009}*{Thm.~4.1} and \cite{Wannerer2014b}*{Prop.~4.2} that this pairing corresponds precisely to the pairing induced by the Alesker product \cite{Alesker2004b} and the Bernig--Fu convolution \cite{Bernig2006}. Since these results partially rely on the irreducibilty \autoref{thm:AleskerIrredThm}, we will consider the pairing on differential forms only and show how the properties we require in the following sections can be obtained from this definition directly.\\
	
	The pairing will depend on the choice of a positive volume form $\vol$ on $\RR^n$, which corresponds to the choice of a linear map $\langle \cdot,\vol\rangle :\Lambda^n(\RR^n)^*\rightarrow \CC$ satisfying $\langle a\vol,\vol\rangle=a$ for $a\in\CC$. We use this map to define 
	\begin{align*}
		\pair{\cdot}{\vol}:\Omega^{*}(S\RR^n)^{\mathrm{tr}}\rightarrow \Omega^{*-n}(\S^{n-1})
	\end{align*}
	by 
	\begin{align*}
		\pair{\pi_1^\ast \tau_1 \wedge \pi_2^\ast \tau_2}{\vol} = \begin{cases}
			\pair{\tau_1}{\vol} \pi_2^\ast \tau_2,& s=n,\\
			0, & \text{else},
		\end{cases}
	\end{align*}
	where $\tau_1 \in \Lambda^s(\RR^n)^*$, $\tau_2 \in \Omega^r(\S^{n-1})$, and $\pi_1:S\RR^n \to \RR^n$ and $\pi_2: S\RR^n \to \S^{n-1}$ denote the natural projections. In particular, for $\omega_1,\omega_2 \in \Omega^{*}(S\RR^{n})^{\mathrm{\mathrm{tr}}}$, $i=1,2$, we have
	\begin{align*}
		d\left(\pair{\omega_1 \wedge \omega_2}{\vol}\right) = (-1)^{n}\pair{d(\omega_1 \wedge \omega_2)}{\vol}.
	\end{align*}
	Note that if $\omega_1$ has bi-degree $(r, n-r-1)$ and $\omega_2$ has bi-degree $(n-r,r-1)$, then $\overline{\omega_1} \wedge \omega_2$ has bi-degree $(n, n-1)$, so $\pair{\overline{\omega_1} \wedge \omega_2}{\vol}$ is a volume form on $\S^{n-1}$.\\
	
	\begin{remark}
		\label{remark:PotinwisePairingEqiv}
		For $g\in \GL(n,\RR)$, we have $G_g^*\pi_1^*\vol=\det g\ \pi_1^*\vol$. In particular,
		\begin{align*}
			\pair{G^*_g\omega}{\vol}=\det g\ G^*_g\pair{\omega}{\vol}
		\end{align*}
		for $\omega\in \Omega^{*}(S\RR^n)^{\mathrm{tr}}$.
	\end{remark}
	
	Let $\im(D)^{\mathrm{tr}}\subset \Omega^n(S\RR^n)$ denote the space of translation invariant forms that are contained in the image of the Rumin operator. Then 
	\begin{align*}
		\im (D)^{\mathrm{tr}}=\bigoplus_{r=1}^{n-1}\im (D)^{\mathrm{tr}}_{r,n-r}
	\end{align*}
	where we set $\im(D)^{\mathrm{tr}}_{r,n-r}=\im D\cap \Omega^{r,n-r}(S\RR^n)^{\mathrm{tr}}$ following the convention in \cite{Bernig2007}. We will consider the subspace 
	\begin{align*}
		V_r:=\im (D:\Omega^{r,n-1-r}(S\RR^n)^{\mathrm{tr}}\rightarrow \Omega^{r,n-r}(S\RR^n)^{\mathrm{tr}}),
	\end{align*}
	that is, $V_r$ is the image of the restriction of $D$ to translation invariant forms. In particular, $V_r\subset (\im D)^{\mathrm{tr}}_{r,n-r}$, however, this inclusion is strict in general (compare the discussion in \autoref{sec:prelimImD}). We set $V_\bullet:=\bigoplus_{r=1}^{n-1}V_r$ and define a pairing on $V_\bullet$ in the following way: For $\tau_1,\tau_2\in V_\bullet$, choose $\omega_1, \omega_2 \in \Omega^{n-1}(S\RR^n)^{\mathrm{tr}}$ such that $\tau_i=D\omega_i$. Then we set
	\begin{align}\label{eq:defPairingForm}
		(\tau_1, \tau_2) = \int_{\{0\} \times \S^{n-1}} \pair{\overline{\omega_1} \wedge D\omega_2}{\vol} \in \CC.
	\end{align}
	With this definition $(\tau_1,\tau_2) = 0$ if the bi-degrees of $\omega_1$ and $D\omega_2$ do not add up to $(n,n-1)$, that is, the pairing is only nontrivial on $V_r \times V_{n-r}$.
	
	Note that we need to verify that Eq.~\eqref{eq:defPairingForm} is well defined, i.e. does not depend on the choice of $\omega_1$. This follows from the following result.
	
	\begin{lemma}\label{lem:pairFormWellDef}
		The pairing in Eq.~\eqref{eq:defPairingForm} is well-defined and sesquilinear. Moreover, for $\tau_1,\tau_2\in V_\bullet$, 
		\begin{align}\label{eq:pairingSymm}
			(\tau_1, \tau_2) = (-1)^{n}\overline{(\tau_2, \tau_1)}.
		\end{align}
	\end{lemma}
	\begin{proof}
		Let $\omega_1,\omega_2\in\Omega^{n-1}(S\RR^n)^{\mathrm{tr}}$ be chosen such that $D\omega_i=\tau_i$.	Let $\xi_1, \xi_2$ be differential forms such that $D\omega_i = d(\omega_i + \alpha \wedge \xi_i)$, $i=1,2$. Then, since $\overline{\alpha} = \alpha$ and $\alpha \wedge D\omega_2 = 0$, 
		\begin{align*}
			\overline{\omega_1} \wedge D \omega_2 = \overline{(\omega_1 + \alpha \wedge \xi_1)} \wedge d(\omega_2 + \alpha \wedge \xi_2).
		\end{align*}
		Next, note that $\alpha\wedge\xi_i$ is translation invariant since this form is uniquely determined by the translation invariant form $\omega_i$. Thus all forms are translation invariant, and we obtain
		\begin{align*}
			(-1)^nd\pair{\overline{(\omega_1 + \alpha \wedge \xi_1)} \wedge (\omega_2 + \alpha \wedge \xi_2)}{\vol}  
			=&\pair{\overline{D\omega_1} \wedge (\omega_2 + \alpha \wedge \xi_2) + (-1)^{n-1}\overline{\omega_1} \wedge D\omega_2}{\vol}\\
			=&\pair{\overline{D\omega_1} \wedge \omega_2}{\vol} + (-1)^{n-1}\pair{\overline{\omega_1} \wedge D\omega_2}{\vol}.
		\end{align*}
		Consequently, $\pair{\overline{D\omega_1} \wedge \omega_2}{\vol} + (-1)^{n-1}\pair{\overline{\omega_1} \wedge D\omega_2}{\vol}$ is an exact form on $\{0\}\times\S^{n-1}$, so Stoke's theorem implies 
		\begin{align}\label{eq:pairDiffHerm}
			\int_{\{0\} \times \S^{n-1}}\pair{\overline{D\omega_1} \wedge \omega_2}{\vol} = (-1)^{n}\int_{\{0\} \times \S^{n-1}}\pair{\overline{\omega_1} \wedge D\omega_2}{\vol},
		\end{align}
		which also shows \eqref{eq:pairingSymm}. In particular, if $D\omega_1 = 0$, then the right-hand side of \eqref{eq:pairDiffHerm} vanishes, so \eqref{eq:defPairingForm} does not depend on the choices $\omega_1$ and $\omega_2$.
				
		As the pairing is obviously sesquilinear, this concludes the proof.
	\end{proof}
	
	An important property of the pairing is its invariance under $\SL(n,\RR)$. More generally, it transforms under $\GL(n,\RR)$ as follows.
	\begin{proposition}\label{prop:pairingFormsSLinv}
		For $g \in \GL(n,\RR)$ and $\tau_1,\tau_2\in V_\bullet$,
		\begin{align*}
			(G_g^\ast \tau_1, G_g^\ast \tau_2) = |\det g| (\tau_1, \tau_2),
		\end{align*}
		where $G_g(x,v) = \left(gx, \frac{g^{-T} v}{\|g^{-T} v\|}\right)$, $(x,v) \in S\RR^n$.
	\end{proposition}
	\begin{proof}
		Fix $\omega_1,\omega_2\in \Omega^{n-1}(S\RR^n)^{\mathrm{tr}}$ with $\tau_i=D\omega_i$. Then $\pair{\overline{\omega_1} \wedge D\omega_2}{\vol}$ belongs to $\Omega^{n-1}(\S^{n-1})$ and is thus a closed form, i.e. $d\pair{\overline{\omega_1} \wedge D\omega_2}{\vol}=0$.
		
		Let $\pi: \RR^n \times (\RR^n \setminus \{0\}) \to S\RR^n$, $\pi(x,y) = \left(x, \frac{y}{\|y\|}\right)$, be the natural projection. Then 
		\begin{align}\label{eq:prfSLinvPair1}
			\int_{\{0\} \times \S^{n-1}} \pair{\overline{\omega_1} \wedge D\omega_2}{\vol} = \int_{\{0\} \times \S^{n-1}} \pi^\ast\pair{\overline{\omega_1} \wedge D\omega_2}{\vol},
		\end{align}
		where we consider $\{0\}\times \S^{n-1}$ as a submanifold of $\RR^n\times(\RR^n\setminus\{0\})$ on the right hand side. Note that the following diagram commutes:
		\begin{equation}
			\begin{tikzcd}
				\RR^n \times (\RR^n \setminus \{0\}) \arrow[r, "\pi"]\arrow[d,"{\widehat{G}_g = (g,g^{-T})}" left] & S\RR^n \arrow[d, "{G_g}"]\\
				\RR^n \times (\RR^n \setminus \{0\}) \arrow[r, "\pi"] & S\RR^n
			\end{tikzcd}
		\end{equation}
		\autoref{remark:PotinwisePairingEqiv} and Eq.~\eqref{eq:prfSLinvPair1} thus imply
		\begin{align*}
			(G_g^\ast \tau_1, G_g^\ast \tau_2) &= \int_{\{0\} \times \S^{n-1}} \pi^\ast\pair{G_g^\ast\overline{ \omega_1} \wedge G_g^\ast D\omega_2}{\vol} \\
			&=\det g \int_{\{0\} \times \S^{n-1}} \pi^\ast G_g^\ast\pair{\overline{ \omega_1} \wedge  D\omega_2}{\vol}\\
			&=\det g \int_{\{0\} \times \S^{n-1}} \widehat{G}_g^\ast\pi^\ast\pair{\overline{ \omega_1} \wedge  D\omega_2}{\vol}.
		\end{align*} 
		The restriction of $\widehat{G}_g$ to a map $\{0\}\times \S^{n-1}\rightarrow \{0\}\times g^{-T}(\S^{n-1})$ is orientation preserving (respectively, reversing) if and only if $\det g>0$ (respectively, $\det g<0$). Consequently,
		\begin{align*}
			(G_g^\ast \tau_1, G_g^\ast \tau_2) &=|\det g| \int_{\{0\} \times g^{-T}(\S^{n-1})} \pi^\ast\pair{\overline{ \omega_1} \wedge  D\omega_2}{\vol}.
		\end{align*}
		Next, note that $g^{-T}(\S^{n-1})$ is homotopic to $\S^{n-1}$ on $\RR^n\setminus\{0\}$. Since the form $\pair{\overline{ \omega_1} \wedge  D\omega_2}{\vol}$ is closed, so is its pullback along $\pi$, and we obtain
		\begin{align*}
			(G_g^\ast \tau_1, G_g^\ast \tau_2) = |\det g|\int_{\{0\} \times \S^{n-1}} \pi^\ast\pair{\overline{\omega_1} \wedge  D\omega_2}{\vol},
		\end{align*}
		so the claim follows from Eq.~\eqref{eq:prfSLinvPair1}.
	\end{proof}
	We thus obtain the following for the action of $\sln(n)_\CC$ on this space. 
	\begin{corollary}\label{cor:pairFormsSLnInvLiealg}
		Let $\tau_1,\tau_2\in V_\bullet$. If $W \in \sln(n)_\CC$, then
		\begin{align*}
			(\LieDer_{\widetilde{W}} \tau_1, \tau_2) = -(\tau_1, \LieDer_{\widetilde{\overline{W}}}\tau_2).
		\end{align*}
	\end{corollary}
	\begin{proof}
		For $W\in \sln(n)$, this follows directly from \autoref{prop:pairingFormsSLinv} by differentiating $t\mapsto \left(\left(G_{\exp(-tW)}\right)^\ast\tau_1,\left(G_{\exp(-tW)}\right)^\ast\tau_2\right)$ in $t=0$. The general case is then a consequence of the sesquilinearity of the pairing.
	\end{proof}

	\subsection{Homogeneous decomposition and Goodey--Weil distributions for polynomial valuations}\label{sec:bghomDecoGW} % for polynomial valuations
	
	This section contains the background results on polynomial valuations necessary for the proof of \autoref{mthm:locVal}.\\
	Recall that a valuation $\mu:\mathcal{K}(\RR^n)\rightarrow\CC$ is called a polynomial  valuation (of degree at most $d$) if there exists $d\in\NN_0$ such that the map
	\begin{align*}
		x\mapsto \mu(K+x),\quad x\in\RR^n,
	\end{align*}
	is a polynomial of degree at most $d$ for every $K\in\mathcal{K}(\RR^n)$. Let $\PVal(\RR^n)$ denote the space of all continuous polynomial valuations and  $\mathrm{P}_d\Val(\RR^n)$ the subspace of all polynomial valuations of degree at most $d\in\NN$. We let $\mathrm{P}_d\Val_r(\RR^n)\subset \mathrm{P}_d\Val(\RR^n)$ denote the subspace of all $r$-homogeneous valuations of degree at most $d$.\\

	The following decomposition is a direct consequence of results by Khovanskii and Pukhlikov \cite{Khovanskii1993} (see also \cite{Knoerr2024}*{Thm.~2.2}). 
	\begin{theorem}
		\label{thm:HomDecompPVal}
		 Let $d \in \NN_0$. Then
		\begin{align*}
			\mathrm{P}_d\Val(\RR^n)=\bigoplus_{r=0}^{d+n}\mathrm{P}_d\Val_r(\RR^n).
		\end{align*}
	\end{theorem}
	We equip $\mathrm{P}_d\Val(\RR^n)$ with the topology of uniform convergence on compact subsets. As in the case of translation invariant valuations, the homogeneous decomposition has the following direct consequence.
	\begin{corollary}
		$\mathrm{P}_d\Val(\RR^n)$ is a Banach space. The topology is induced by the norm 
		\begin{align*}
			\|\varphi\|:=\sup_{K\subset B_1(0)}|\varphi(K)|,\quad \varphi\in \mathrm{P}_d\Val(\RR^n),
		\end{align*}
		where $B_1(0)$ denotes the Euclidean ball of radius $1$ around the origin.
	\end{corollary}
	
	Let us denote by $\PVal_r(\RR^n)\subset \PVal(\RR^n)$ the subspace of all $r$-homogeneous valuations. Note that the subspaces $\mathrm{P}_d\Val_r(\RR^n)$ form a natural increasing filtration of this space.
	
	The following is proved in \cite{McMullen1977} (see also \cite{Schneider2014}*{Sec.~6.3} and the references therein) for translation invariant valuations and is equivalent to the corresponding homogeneous decomposition for $\Val(\RR^n)$. The same reasoning applies to polynomial valuations. 
	\begin{theorem}
		\label{theorem:polarization}
		Let $\varphi\in\PVal_r(\RR^n)$. There exists a unique continuous symmetric map $\bar{\varphi}:\K(\RR^n)^r\rightarrow\CC$ that is Minkowski additive in each argument such that
		\begin{align*}
			\varphi(\lambda_1K_1+\dots+\lambda_mK_m)=\sum_{r_1,\dots,r_m=0}^r\binom{r}{r_1,\dots, r_m}\lambda_1^{r_1}\dots\lambda_m^{r_m}\bar{\varphi}(K_1[r_1],\dots,K_m[r_m])
		\end{align*}
		for all $K_1,\dots,K_m\in\K(\RR^n)$ and all real $\lambda_1,\dots,\lambda_m\ge 0$.
	\end{theorem}
	We will call $\bar{\varphi}$ the polarization of $\varphi\in\Val_r(\RR^n)$.

	\begin{corollary}\label{cor:PValFiniteDegree}
		For every $r\in\NN$, $\PVal_r(\RR^n)= \mathrm{P}_r\Val_r(\RR^n)$. In particular, $\PVal_r(\RR^n)$ is a Banach space with respect to the norm $\|\cdot\|$.
	\end{corollary}
	\begin{proof}
		\autoref{theorem:polarization} shows that the map $x\mapsto \varphi(K+x)$ is a polynomial of degree at most $r$ for every $\varphi\in\PVal_r(\RR^n)$, which shows the first claim. Since the subspace of $r$-homogeneous valuations is obviously closed in $\mathrm{P}_r\Val(\RR^n)$, the result follows.
	\end{proof}

	We will extend the polarization to a multilinear functional on differences of support functions, where the support function of $K\in\mathcal{K}(\RR^n)$ is the convex function $h_K:\RR^n\rightarrow\RR$ defined by $h_K(y)=\sup_{x\in K}\langle y,x\rangle$ for $x\in\RR^n$, where $\langle\cdot,\cdot\rangle$ denotes the standard inner product on $\RR^n$. Support functions have the following properties (compare \cite{Schneider2014}*{Sec.~1.7}):
	\begin{lemma}\label{lem:PropSupportFunction}
		The following holds for $K,L\in\mathcal{K}(\RR^n)$:
		\begin{enumerate}
			\item $h_K$ is $1$-homogeneous: $h_K(ty)=th_K(y)$, $t\ge 0$, $y\in\RR^n$,
			\item $h_{gK}(y)=h_K(g^{T}y)$ for $g\in\GL(n,\RR)$,
			\item $h_{K+x}(y)=h_K(y)+\langle y,x\rangle$, $K\in\mathcal{K}(\RR^n)$, $x\in\RR^n$,
			\item $h_{K+L}=h_K+h_L$,
			\item $\max\{h_K,h_L\}=h_{\mathrm{conv}(K\cup L)}$, $\min\{h_K,h_L\}=h_{K\cap L}$, where $\mathrm{conv}(A)$ denotes the convex hull of a set $A\subset\RR^n$.
			\item If $(K_j)_j$ is a sequence in $\mathcal{K}(\RR^n)$, then $K_j\rightarrow K$ in the Hausdorff metric if and only if $h_{K_j}\rightarrow h_K$ uniformly on $\S^{n-1}$.
		\end{enumerate}
		Moreover, a function $h:\RR^n\rightarrow\RR$ is the support function of a convex body if and only if $h$ is $1$-homogeneous and convex.
	\end{lemma}
	Since support functions are $1$-homogeneous, we identify them with functions on the sphere, i.e. $h_K\in C(\S^{n-1})$ for $K\in\mathcal{K}(\RR^n)$, and we set $\|h_K\|_\infty=\sup_{v\in \S^{n-1}}|h_K(v)|$. We need the following standard estimate for $\bar{\varphi}$.
	\begin{lemma}
		 \label{lem:estimatePolarization}
		 There exists a constant $C(r)>0$ such that for every $\varphi\in \PVal_r(\RR^n)$
		 \begin{align*}
		 	|\bar{\varphi}(K_1,\dots,K_r)|\leq C(r)\|\varphi\| \prod_{j=1}^r\|h_{K_j}\|_\infty
		 \end{align*}
	 	for all $K_1,\dots,K_r\in\K(\RR^n)$.
	\end{lemma}
	\begin{proof}
		Since $\bar{\varphi}(K_1,\dots,K_r)$ is essentially the coefficient of the polynomial 
		\begin{align*}
			(\lambda_1,\dots,\lambda_r)\mapsto \varphi(\lambda_1K_1+\dots+\lambda_rK_r)
		\end{align*}
		in front of $\lambda_1\cdots\lambda_r$, where $\lambda_1,\dots,\lambda_r\ge 0$, we may use the inverse of the Vandermonde matrix to obtain constants $c_{j_1,\dots,j_r}$ for $0\leq j_1,\dots,j_r\leq r$ independent of $K_1,\dots,K_r\in\K(\RR^n)$ and $\varphi\in \PVal_r(\RR^n)$ such that
		\begin{align*}
			\bar{\varphi}(K_1,\dots,K_r)=\sum_{j_1,\dots,j_r=0}^r c_{j_1,\dots,j_r}\varphi(j_1K_1+\dots+j_rK_r).
		\end{align*}
		In particular, 
		\begin{align*}
			\sup_{K_1,\dots,K_r\subset B_1(0)}|\bar{\varphi}(K_1,\dots,K_r)|\le& (r+1)^2 \max |c_{j_1,\dots,j_r}| \sup_{K\subset r^2 B_1(0)}|\varphi(K)|\\
			\le&(r+1)^{2+2r}\max |c_{j_1,\dots,j_r}| \cdot\|\varphi\|.
		\end{align*}
		Since $\bar{\varphi}$ is $1$-homogeneous in each argument, this implies the desired inequality.
	\end{proof}

	Let $\mathcal{D}\subset C(\S^{n-1})$ denote the space of all functions that can be written as a difference of support functions. 
	We extend $\bar{\varphi}$ to a multilinear function $\tilde{\varphi}:\mathcal{D}^r\rightarrow \CC$ as follows: For every $\phi\in \mathcal{D}$ choose $K_\phi,L_\phi\in\K(\RR^n)$ such that $\phi=h_{K_\phi}-h_{L_\phi}$. Given $\phi_1,\dots,\phi_r\in \mathcal{D}$, we define
	\begin{align}
		\label{eq:defMultLinExtension}
		\tilde{\varphi}(\phi_1,\dots,\phi_r)=\sum_{l=0}^r \frac{(-1)^{r-l}}{l!(r-l!)}\sum_{\sigma\in S_r} \bar{\varphi}\left(K_{\phi_{\sigma(1)}},\dots,K_{\phi_{\sigma(l)}},L_{\phi_{\sigma(l+1)}},\dots,L_{\phi_{\sigma(r)}}\right),
	\end{align}
	where $S_r$ denotes the group of permutations of $\{1,\dots,r\}$. Since $\bar{\varphi}$ is additive in each argument, it is easy to see that this definition is independent of the specific choice of the bodies $K_{\phi_j}$, $L_{\phi_j}\in \mathcal{K}(\RR^n)$.	Using this fact and that $\bar{\varphi}$ is symmetric, one easily establishes the following properties.
	\begin{lemma}
		For $\varphi\in\PVal_r(\RR^n)$, $\tilde{\varphi}$ has the following properties.
		\begin{enumerate}
			\item $\tilde{\varphi}$ is multilinear.
			\item $\tilde{\varphi}$ is symmetric.
			\item For $K_1,\dots,K_r\in\K(\RR^n)$, $\tilde{\varphi}(h_{K_1},\dots,h_{K_r})=\bar{\varphi}(K_1,\dots,K_r)$.
		\end{enumerate}
	\end{lemma}	
	The following establishes a direct relation between $\varphi$ and $\tilde{\varphi}$.
	\begin{lemma}
		\label{lem:MultLinExtensionAsDerivative}
		For $\phi_1,\dots,\phi_r\in \mathcal{D}$ let $K_1,\dots,K_r\in\K(\RR^n)$ be convex bodies such that 
		\begin{align*}
			h_{L^t_j}:=h_{K_j}+t\phi_j
		\end{align*}
		is convex for all $t\in [0,\varepsilon]$ for some $\varepsilon>0$. Then for every $\varphi\in \PVal_r(\RR^n)$,
		\begin{align*}
			\tilde{\varphi}(\phi_1,\dots,\phi_r)=\frac{1}{r!}\frac{\partial^r}{\partial t_1\dots\partial t_r}\Big|_0 \varphi\left(\sum_{j=1}^r L^{t_j}_j\right).
		\end{align*}
		In particular, 
		\begin{align*}
			\tilde{\varphi}(\phi_1,\dots,\phi_1)=\frac{1}{r!}\frac{d^r}{dt^r}\Big|_0\varphi(L^t_1).
		\end{align*}
	\end{lemma}
	\begin{proof}
		First note that bodies with this property always exist: If $\phi=h_K-h_L$ for $K,L\in\K(\RR^n)$, then $h_L+t\phi$ is convex for all $t\in[-1,1]$. If we fix such bodies, then 
		\begin{align*}
			\varphi\left(\sum_{j=1}^r L^{t_j}_j\right)=\tilde{\varphi}\left(\sum_{j=1}^r h_{L^{t_j}_j},\dots,\sum_{j=1}^r h_{L^{t_j}_j}\right)
		\end{align*}
		by the definition of $\tilde{\varphi}$ and \autoref{lem:PropSupportFunction}, so the claim follows by multilinearity.
	\end{proof}
	
	Note that every function in $C^\infty(\S^{n-1})$ can be written as difference of support functions, so $\tilde{\varphi}$ restricts to a multilinear functional on $(C^\infty(\S^{n-1}))^r$. It turns out that this functional is continuous and thus extends to a distribution on $(\S^{n-1})^r$, as shown by the following result due to Goodey and Weil.
	\begin{theorem}[\cite{Goodey1984}*{Thm.~2.1}]\label{thm:GWbyGW}
		For every $\varphi\in\PVal_r(\RR^n)$ there exists a unique distribution $\GW(\varphi)$ on $(\S^{n-1})^r$ such that
		\begin{align}
			\label{eq:definingPropertyGW}
			\GW(\varphi)[h_{K_1}\otimes\dots\otimes h_{K_r}]=\bar{\varphi}(K_1,\dots,K_r)
		\end{align}
		for all smooth convex bodies $K_1,\dots,K_r\in\K(\RR^n)$ with strictly positive Gauss curvature.
	\end{theorem}
	\begin{remark}
		\label{remark:extensionGWtoC2Functions}
%		\begin{enumerate}
%			\item 
			The distribution $\GW(\varphi)$ is called the Goodey--Weil distribution associated to $\varphi\in \PVal_r(\RR^n)$. Note that $\GW(\varphi)[h_K\otimes\dots\otimes h_K]=\varphi(K)$ for $K$ smooth and with strictly positive Gauss curvature, so since $\varphi$ is continuous, it is uniquely determined by its Goodey--Weil distribution.
	\end{remark}
	The following was shown in \cite{Alesker2000}*{Prop.~3.3} for translation invariant valuations. The proof holds verbatim in the polynomial case.
	\begin{proposition}
		For $\varphi\in \PVal_r(\RR^n)$, the support of $\GW(\varphi)$ is contained in the diagonal in $(\S^{n-1})^r$.
	\end{proposition}
	Let $\Delta_r:\S^{n-1}\rightarrow (\S^{n-1})^r$ denote the diagonal embedding. We define the \emph{vertical support} of $\varphi\in\PVal_r(\RR^n)$ to be the subset of $\S^{n-1}$ given by
	\begin{align}\label{eq:defvsuppPVal}
		\vsupp\varphi:=\Delta_r^{-1}(\supp\GW(\varphi)).
	\end{align}
	If $\varphi$ is an arbitrary polynomial valuation, fix $d\ge 0$ such that $\varphi\in \mathrm{P}_d\Val(\RR^n)$ and consider the decomposition $\varphi=\sum_{r=0}^{d+n}\varphi_r$ into its homogeneous components. Then
	\begin{align*}
		\vsupp\varphi:=\bigcup_{r=1}^{d+n}\vsupp\varphi_r.
	\end{align*}
	In particular, the vertical support of a $0$-homogeneous valuation is empty by definition.	The vertical support of a translation-invariant valuation was characterized in \cite[Prop.~6.14]{Knoerr2020a}. The proof holds verbatim in the polynomial case.
	\begin{proposition}\label{prop:CharVsupp}
		\label{prop:characterizationVerticalSupp}
		Let $\varphi\in\PVal(\RR^n)$. The vertical support of $\varphi$ is minimal (with respect to inclusion) among all closed sets $A\subset \S^{n-1}$ with the following property: If $K,L\in\mathcal{K}(\RR^n)$ are two convex bodies with $h_K=h_L$ on a neighborhood of $A$, then $\varphi(K)=\varphi(L)$.
	\end{proposition}
	Given a closed subset $A\subset \S^{n-1}$, let $\PVal_{r,A}(\RR^n)$ denote the subspace of all $\varphi\in \PVal_r(\RR^n)$ such that $\vsupp\varphi\subset A$. \autoref{prop:characterizationVerticalSupp} directly implies the following (which is stated in the translation invariant case in \cite[Corollary~6.15]{Knoerr2020a}).
	\begin{corollary}
		\label{cor:SuppRestrClosedSubspace}
		For a closed set $A\subset \S^{n-1}$, $\PVal_{r,A}(\RR^n)$ is a closed subspace of $\PVal_r(\RR^n)$. In particular, it is a Banach space.
	\end{corollary}
	
\subsection{Valuations on convex functions}
As a general reference on this section, we refer to \cites{Knoerr2020a, Knoerr2020b, Colesanti2019, Colesanti2019b, Colesanti2020} and the references therein.\\
Let $\Conv(\RR^n,\RR)$ denote the space of all convex functions $f:\RR^n\to\RR$. This space carries a natural topology induced by epi-convergence (which in this setting coincides with pointwise or locally uniform convergence, see \cite{Rockafellar1970}*{Thm.~7.17}). A map $\mu:\Conv(\RR^n, \RR) \to \CC$ is called a valuation if
\begin{align*}
	\mu( \max\{f,g\}) + \mu(\min\{f,g\}) = \mu(f) + \mu(g)
\end{align*}
whenever $f, g, \min\{f,g\} \in \Conv(\RR^n, \RR)$, where $\max\{f,g\}$ resp.\ $\min\{f,g\}$ denote the pointwise maximum resp.\ minimum.\\

We denote by $\VConv(\RR^n)$ the space of all continuous valuations $\mu:\Conv(\RR^n,\RR)\rightarrow \CC$ that are in addition \emph{dually epi-translation invariant}, that is, that satisfy
\begin{align*}
	\mu(f+\ell)=\mu(f)
\end{align*}
for all $f\in\Conv(\RR^n,\RR)$, and $\ell:\RR^n\rightarrow\RR$ affine. This notion is intimately related to translation invariance for valuations on convex bodies, compare \autoref{sec:RelationFctsBodies}.\\

Similar to the construction of valuations on convex bodies in terms of integration with respect to the normal cycle, the following construction of valuations on convex functions was examined in \cite{Knoerr2020b}: Let $D(f)$ denote the differential cycle of $f \in \Conv(\RR^n, \RR)$ as defined by Fu in \cite{Fu1989}. This is an integral current on the cotangent bundle $T^\ast\RR^n$ of $\RR^n$, which for smooth functions coincides with the current given by integration over the graph of the differential of $f$. As shown in \cite{Knoerr2020b}, any smooth differential form $\tau \in \Omega^{n}(T^\ast \RR^n)$ whose support is bounded in the first argument of $T^\ast \RR^n = \RR^n \times (\RR^n)^\ast$ induces a continuous valuation on $\Conv(\RR^n,\RR)$ by setting
\begin{align}\label{eq:defSmValConvDiffForm}
	\mu(f) = D(f)[\tau], \quad f \in \Conv(\RR^n, \RR).
\end{align}
In general, such a valuation will not be dually epi-translation invariant, however, if the differential form is invariant with respect to translations in the second factor  of $T^\ast \RR^n = \RR^n \times (\RR^n)^\ast$, then the valuation belongs to $\VConv(\RR^n)$. Conversely, if a valuation in $\VConv(\RR^n)$ admits such a representation, then the differential form can be chosen to be invariant with respect to translations in the second factor (compare \cite{Knoerr2020b}*{Thm.~5.5}). We will call valuations of this form \emph{representable by integration with respect to the differential cycle}.

\begin{remark}
	In \cite{Knoerr2020b}, valuations representable by integration with respect to the differential cycle were called "smooth valuations" in analogy with \autoref{mthm:smoothValsByDiffform}. Following the convention in \cite{Knoerr2025}, we reserve this terminology for a different space of valuations, which is discussed below. The main result of \cite{Knoerr2025} shows that these two notions are equivalent, but we will distinguish between them to avoid further ambiguity.
\end{remark}

We equip $\VConv(\RR^n)$ with the topology of uniform convergence on compact subsets of $\Conv(\RR^n,\RR)$ (see \cite{Knoerr2020a}*{Prop.~2.4} for a description of these subsets). We have a natural continuous representation of the group of translations (which we identify with $\RR^n$) on $\VConv(\RR^n)$, defined by associating to $x\in\RR^n$, $\mu\in\VConv(\RR^n)$ the valuation $\pi(x)\mu\in\VConv(\RR^n)$ given by
\begin{align*}
	[\pi(x)\mu](f) := \mu(f(\cdot + x))
\end{align*} 
for $f\in\Conv(\RR^n,\RR)$. We will be interested in the smooth vectors of this representation.
\begin{definition}[\cite{Knoerr2025}*{Def.~1.3}]\label{def:affSmValConv}
	A valuation $\mu \in \VConv(\RR^n)$ is called a smooth valuation if the map
	\begin{align*}
		\RR^n &\to \VConv(\RR^n)\\
		x &\mapsto \left[f \mapsto \mu(f(\cdot + x))\right]
	\end{align*}
	is smooth.
\end{definition}

The main result of \cite{Knoerr2025} relates smooth valuations and valuations that are representable by integration with respect to the differential cycle.
\begin{theorem}[\cite{Knoerr2025}*{Thm.~D}]\label{thm:SmVConvGivenByDiffForm}
	Let $\mu \in \VConv_r(\RR^n)$. Then $\mu$ is a smooth valuation in the sense of \autoref{def:affSmValConv} if and only if it is representable by integration with respect to the differential cycle.
\end{theorem}
\begin{remark}
	Let us again point out that the proof of this result does not rely on Alesker's Irreducibility \autoref{thm:AleskerIrredThm}.
\end{remark}

\subsection{Relation between valuations on convex bodies and convex functions}
\label{sec:RelationFctsBodies}

We will obtain \autoref{mthm:smoothValsByDiffform} from \autoref{thm:SmVConvGivenByDiffForm} using a relation between valuations on convex functions and convex bodies introduced in \cite{Knoerr2020a}. The construction relates $\VConv(\RR^{n})$ with a certain subspace of $\Val(\RR^{n+1})$ defined by restrictions on the vertical support. We will restrict the discussion to the translation invariant case (i.e. to polynomial valuations of degree $d=0$). The construction generalizes to polynomial valuations on convex functions, which were considered in \cite{Knoerr2024}.\\

Similar to the notion of vertical support for polynomial valuations on convex bodies, the support of an element of $\VConv(\RR^n)$ was introduced in \cite{Knoerr2020a} in terms of certain distributions associated to homogeneous valuations, mirroring the construction by Goodey and Weil \cite{Goodey1984}. Similar to \autoref{prop:characterizationVerticalSupp}, we have the following characterization of the support.
\begin{proposition}[\cite{Knoerr2020a}*{Prop.~6.3}]\label{prop:CharSupport}
	Let $\mu\in\VConv(\RR^n)$. The support of $\mu$ is minimal (with respect to inclusion) among all closed sets $A\subset \RR^n$ with the following property: If $f,g\in\Conv(\RR^n, \RR)$ satisfy $f=g$ on an open neighborhood of $A$, then $\mu(f)=\mu(g)$.
\end{proposition}
\begin{remark}
	\label{remark:suppCompact}
	For $\CC$-valued valuations, the support is always a compact subset of $\RR^n$, compare \cite{Knoerr2020b}*{Thm.~2}.
\end{remark}
For a given closed set $A \subset \RR^n$, we denote by $\VConv_A(\RR^n) \subset \VConv(\RR^n)$ the subspace of all $\mu \in \VConv(\RR^n)$ with $\supp \mu \subset A$. For a closed set $B \subset \S^{n-1}$, we let $\Val_{B}(\RR^n) = \PVal_B(\RR^n) \cap \Val(\RR^n)$.

\begin{theorem}[\cite{Knoerr2020a}*{Thm.~3.3 \& Thm.~6.18}]\label{thm:defTmapProp}
	The map
	\begin{align}
		T: \VConv(\RR^n) &\to \Val(\RR^{n+1})\\
		\mu &\mapsto \left[ K \mapsto \mu(h_K(\cdot, -1)) \right] \nonumber
	\end{align}
	is well-defined, continuous, and injective. Its image consists precisely of all valuations $\varphi \in \Val(\RR^{n+1})$ with $\vsupp\, \varphi \subset \S^n_-$, where $\S^n_- = \{y \in \S^n:\, y_{n+1}<0\}$ denotes the negative half sphere.
	
	Moreover, $T: \VConv_A(\RR^n) \to \Val_{P(A)}(\RR^{n+1})$ is a topological isomorphism for all compact $A \subset \RR^n$, where $P: \RR^n \to \S^n$ denotes the map $v \mapsto \frac{(v,-1)}{\sqrt{1+\|v\|^2}}$.
\end{theorem}

The previous map admits a direct interpretation for valuations that are representable with respect to the normal and differential cycle. Let $\mathrm{pr}_{\RR^n}:\RR^{n+1}\rightarrow \RR^n$ denote the projection onto the first $n$ coordinates. It was shown in \cite{Knoerr2020b}*{Prop.~6.1} that the map 
\begin{align*}
	Q:\RR^{n+1}\times \S_-^n&\rightarrow \RR^n\times\RR^n\\
	(x,v)&\mapsto (-P^{-1}(v),\mathrm{pr}_{\RR^n}(x))
\end{align*}
satisfies $Q_*(\nc(K)|_{\RR^{n+1}\times \S^n})=(-1)^{n+1}D(h_K(\cdot,-1))$ for every $K\in\mathcal{K}(\RR^{n+1})$ (the additional sign is a consequence of a different choice of orientation in \cite{Knoerr2020b} in comparison with the standard orientation of $\RR^{n+1}$). Here, we identify $\RR^n\times\RR^n\cong \RR^n\times (\RR^n)^*$ using the standard inner product on $\RR^n$. This has the following direct consequence.
\begin{lemma}\label{lem:diffformSmIffValVConv}
	If $\mu\in\VConv_r(\RR^n)$ is representable by integration with respect to the differential cycle, then $T(\mu)$ is representable by integration with respect to the normal cycle. Moreover, in this case, $T(\mu)=\int_{\nc(\cdot)}\omega$ for a translation invariant differential form $\omega\in \Omega^{n}(S\RR^{n+1})^{\mathrm{tr}}$ with support contained in $\RR^{n+1}\times \S^n_-$.
\end{lemma}
\begin{proof}
	If $\mu\in\VConv_r(\RR^n)$ is representable by integration with respect to the differential cycle, then by \cite{Knoerr2020b}*{Thm.~5.5} there exists a differential form $\omega\in \Omega^{n-r}(\RR^n)\otimes\Lambda^r (\RR^n)^*$ with bounded support in the first component, such that $\mu=D(\cdot)[\omega]$. Then $(-1)^{n+1}Q^*\omega$ is a translation invariant form in $\Omega^n(\RR^{n+1}\times \S^n_-)$ and the projection of its support onto $\S^n_-$ is compact. We may in particular extend the form $(-1)^{n+1}Q^*\omega$ by $0$ to a smooth form $\tilde{\omega}$ on $\RR^{n+1}\times \S^n$. Then $T(\mu)=\int_{\nc(\cdot)}\tilde{\omega}$ due to \cite{Knoerr2020b}*{Prop.~6.1}. Obviously, $\tilde{\omega}$ has the desired properties.
\end{proof}
\begin{remark}
	In \cite{Knoerr2020b}*{Prop.~6.4} a stronger version of this result was shown, however, we will only need the weaker version in \autoref{lem:diffformSmIffValVConv}.
\end{remark}

\part{Localization, proofs of \autoref{mthm:smoothValsByDiffform} and \autoref{mthm:locVal}}\label{part1}

\section{Localization of polynomial valuations}
\label{sec:localization}

In this section, we establish \autoref{mthm:locVal}. We will use the multilinear extension of the polarization of a polynomial valuation discussed in \autoref{sec:bghomDecoGW} to define a multilinear functional on support functions with a given support restriction using a partition of unity and then verify that this construction preserves the valuation property.

\medskip

We start by showing that $\rho h_K$ can be written as difference of support functions for all $\rho \in C^2(\S^{n-1})$ and $K \in \K(\RR^n)$.
	
\begin{lemma}
	\label{lem:localizedSuppFunctionIsDifference}
	There exists a constant $C_n>0$ such that the following holds: For every $\rho\in C^2(\S^{n-1})$ and $K\in\K(\RR^n)$, the convex body
	\begin{align*}
		L(K,\rho):=C_n\|\rho\|_{C^2(\S^{n-1})} (K+\|h_K\|_\infty B_1(0))
	\end{align*}
	has the property that $h_{L(K,\rho)}+t\rho h_K$ is the support function of a convex body for every $t\in [-1,1]$.
\end{lemma}
\begin{proof}
	If $K$ is smooth with strictly positive Gauss curvature, then its support function $h_K:\RR^n\rightarrow\RR$ is smooth on $\RR^n\setminus\{0\}$. We extend $\rho$ to a $0$-homogeneous function on $\RR^n\setminus\{0\}$. The Hessian of the $1$-homogeneous function $\rho h_K$ in $x\in \S^{n-1}$ is then given by
	\begin{align*}
		D^2(\rho h_K)=D^2\rho \cdot h_K+\nabla\rho \cdot\nabla h_K^T+\nabla h_K \cdot\nabla \rho^T+\rho D^2 h_K.
	\end{align*}
	In particular, 
	\begin{align*}
		&D(h_{L(K,\rho)}+t\rho h_K)\\
		=&\left(C_n\|\rho\|_{C^2(\S^{n-1})}+t\rho\right) D^2h_K+C_n\|\rho\|_{C^2(\S^{n-1})}\|h_K\|_\infty D^2h_{B_1(0)}\\
		 &+t\left(D^2\rho \cdot h_K+\nabla\rho \cdot\nabla h_K^T+\nabla h_K \cdot\nabla \rho^T\right).
	\end{align*} 
	For $v\in x^\perp$ and $t \in [-1,1]$, we have
	\begin{align*}
		&|t\pair{v}{\left(D^2\rho \cdot h_K+\nabla\rho \cdot\nabla h_K^T+\nabla h_K \cdot\nabla \rho^T\right)v}|\\
		\le& D_n\|\rho\|_{C^2(\S^{n-1})}(\|h_K\|_\infty+2\|\nabla h_K\|_\infty)|v|^2
	\end{align*}
	for some constant $D_n > 0$ independent of $\rho$ and $K$. Since $\nabla h_K:\S^{n-1}\rightarrow\partial K$ is the inverse Gauss map, $\|\nabla h_K\|_\infty=\max_{x\in K}|x|=\max\{r>0: K\subset B_r(0)\}=\|h_K\|_\infty$. Thus for $C_n\ge \max\{1,3D_n\}$, we obtain for $v\in x^\perp$,
	\begin{align*}
		&\pair{v}{D^2(h_{L(K,\rho)}+t\rho h_K)v}\\
		\ge&\left(C_n\|\rho\|_{C^2(\S^{n-1})}-\|\rho\|_\infty\right) \pair{v}{D^2h_Kv} +C_n\|\rho\|_{C^2(\S^{n-1})}\|h_K\|_\infty|v|^2\\
		&-3D_n\|\rho\|_{C^2(\S^{n-1})}\|h_K\|_\infty|v|^2\ge 0.
	\end{align*}
	Since $\rho h_K$ is $1$-homogeneous, this implies that $h_{L(K,\rho)}+t\rho h_K$ is convex for all $t\in [-1,1]$ and thus the support function of a convex body.\\
	
	In the general case, take a sequence $(K_j)_j$ of smooth convex bodies with strictly positive Gauss curvature converging to $K$ in the Hausdorff metric. Then $h_{L(K_j,\rho)}+t\rho h_{K_j}$ is convex for every $t\in[-1,1]$ and $j\in\NN$ by the previous discussion, and from the definition  we directly see that $L(K_j,\rho)$ converges to $L(K,\rho)$ for $j\rightarrow\infty$. Thus $h_{L(K,\rho)}+t\rho h_{K}$ is the pointwise limit of a sequence of convex functions and thus convex as well.
\end{proof}

\autoref{lem:localizedSuppFunctionIsDifference} implies, in particular, that the multilinear extension $\tilde{\varphi}$ can be evaluated in functions of the form $\rho h_K$ for $\rho\in C^2(\S^{n-1})$ and $K\in\K(\RR^n)$.

\begin{lemma}
	\label{lem:localizationPolarization}
	Let $\varphi\in \PVal_r(\RR^n)$. For every $\rho_1,\dots,\rho_r\in C^2(\S^{n-1})$, the map
	\begin{align*}
		\hat{\varphi}_{\rho_1,\dots,\rho_r}:\K(\RR^n)^r&\rightarrow \CC\\
		(K_1,\dots,K_r)&\mapsto \tilde{\varphi}(\rho_1 h_{K_1},\dots,\rho_r h_{K_r})
	\end{align*}
	is well-defined, jointly continuous, and Minkowski additive in each argument. Moreover, there exists a constant $C_{n,r}>0$ such that
	\begin{align*}
		\sup_{K_1,\dots,K_r\subset B_1(0)}|\hat{\varphi}_{\rho_1,\dots,\rho_r}(K_1,\dots,K_r)|\leq C_{n,r}\|\varphi\| \prod_{j=1}^r\|\rho_j\|_{C^2(\S^{n-1})}
	\end{align*}
	for all $\varphi\in \PVal_r(\RR^n)$ and $\rho_1,\dots,\rho_r\in C^2(\S^{n-1})$.
\end{lemma}
\begin{proof}
	Since $\rho_jh_{K_j}$ is a difference of support functions by \autoref{lem:localizedSuppFunctionIsDifference}, $\hat{\varphi}_{\rho_1,\dots,\rho_r}$ is well defined. Moreover, since $\tilde{\varphi}$ is multilinear, $\hat{\varphi}_{\rho_1,\dots,\rho_r}$ is Minkowski additive in each argument. If we let $\tilde{L}(K,\rho)\in \K(\RR^n)$ be the convex body with support function 
	\begin{align*}
		h_{\tilde{L}(K,\rho)}=h_{L(K,\rho)}+\rho h_K
	\end{align*} 
	for $K\in\K(\RR^n)$ and $\rho\in C^2(\S^{n-1})$, then Eq.~\eqref{eq:defMultLinExtension} implies that
	\begin{align}
		\label{eq:localizationInTermsOfBodies}
		&\hat{\varphi}_{\rho_1,\dots,\rho_r}(K_1,\dots,K_r)\\
		\notag
		=&\sum_{l=0}^r \frac{(-1)^{r-l}}{l!(r-l!)}\sum_{\sigma\in S_r} \bar{\varphi}\left(\tilde{L}(K_{\sigma(1)},\rho_{\sigma(1)}),\dots,\tilde{L}(K_{\sigma(l)},\rho_{\sigma(l)}),L(K_{\sigma(l+1)},\rho_{\sigma(l+1)}),\dots,L(K_{\sigma(r)},\rho_{\sigma(r)})\right).
	\end{align}
	Since $\tilde{L}(K,\rho)$ and $L(K,\rho)$ are both contained in $3C_n\|\rho\|_{C^2(\S^{n-1})}\|h_K\|_\infty B_1(0)$ by construction, \autoref{lem:estimatePolarization} implies
	\begin{align*}
		|\hat{\varphi}_{\rho_1,\dots,\rho_r}(K_1,\dots,K_r)|\leq 2^nC(r)(3C_n)^r \prod_{j=1}^r \|\rho_j\|_{C^2(\S^{n-1})}\|h_{K_j}\|_\infty,
	\end{align*}
	which shows the desired estimate.
	
	It remains to see that $\hat{\varphi}_{\rho_1,\dots,\rho_r}$ is jointly continuous. Let $(K^m_j)_m$,  be sequences of convex bodies converging to $K_j$ for $1\leq j\leq r$ respectively. Since
	\begin{align*}
		\lim\limits_{m\rightarrow\infty}\tilde{L}(K_j^m,\rho_j)=\tilde{L}(K_j,\rho_j),
		\lim\limits_{m\rightarrow\infty}L(K_j^m,\rho_j)=\tilde{L}(K_j,\rho_j), 
	\end{align*}
	the representation of $\hat{\varphi}_{\rho_1,\dots,\rho_r}$ given by Eq.~\eqref{eq:localizationInTermsOfBodies} shows that
	\begin{align*}
		\lim\limits_{m\rightarrow\infty}\hat{\varphi}_{\rho_1,\dots,\rho_r}(K^m_1,\dots,K^m_r)=\hat{\varphi}_{\rho_1,\dots,\rho_r}(K_1,\dots,K_r),
	\end{align*}
	since the polarization is continuous by \autoref{theorem:polarization}. Consequently, $\hat{\varphi}_{\rho_1,\dots,\rho_r}$ is jointly continuous.
\end{proof}

Next, we are going to use the functionals constructed so far to obtain valuations. We split the construction into two parts.
\begin{proposition}
	\label{prop:localizationValuationProperty}
	Let $\varphi\in \PVal_r(\RR^n)$, and $\rho\in C^2(\S^{n-1})$ be nonnegative. Then $\varphi_{\rho}:\K(\RR^n)\rightarrow\CC$ defined by
	\begin{align*}
		\varphi_\rho(K)=\hat{\varphi}_{\rho,\dots,\rho}(K,\dots,K)
	\end{align*}
	belongs to $\PVal_r(\RR^n)$.
\end{proposition}
\begin{proof}
	\autoref{lem:localizationPolarization} implies that $\varphi_\rho$ is continuous. In order to see that $\varphi_\rho$ is a valuation, let $K,L\in\K(\RR^n)$ be two convex bodies such that $K\cup L$ is convex. Consider the convex body
	\begin{align*}
		M=C_n\|\rho\|_{C^2(\S^{n-1})} (K+L+K\cup L+K\cap L+(\|h_K\|_\infty+\|h_L\|_\infty) B_1(0)).
	\end{align*}
	\autoref{lem:localizedSuppFunctionIsDifference} implies that for $t\in[0,1]$, the functions
	\begin{align*}
		h_M+t\rho h_K, h_M+t\rho h_L,h_M+t\rho h_{K\cup L},h_M+t\rho h_{K\cap L}
	\end{align*}
	are all convex. Let us denote the corresponding bodies by $M^t_K,M^t_L,M^t_{K\cup L}$ and $M^t_{K\cap L}$ (compare \autoref{lem:PropSupportFunction}). Since $\rho$ is nonnegative and $t\ge 0$, we have
	\begin{align*}
		\max\{h_M+t\rho h_K,h_M+t\rho h_L\}=h_M+t\rho \max\{h_K,h_L\}=h_M+t\rho h_{K\cup L},\\
		\min\{h_M+t\rho h_K,h_M+t\rho h_L\}=h_M+t\rho \min\{h_K,h_L\}=h_M+t\rho h_{K\cap L}.
	\end{align*}
	In other words,
	\begin{align*}
		&\max\{h_{M^t_K},h_{M^t_L}\}=h_{M^t_{K\cup L}}, &&\min\{h_{M^t_K},h_{M^t_L}\}=h_{M^t_{K\cap L}},
	\end{align*}
	so in particular, $M^t_K\cup M^t_L=M^t_{K\cup L}$ is convex and $M^t_K\cap M^t_L=M^t_{K\cap L}$ by \autoref{lem:PropSupportFunction}. As $\varphi$ is a valuation, we thus obtain 
	\begin{align*}
		\varphi(M^t_K)+\varphi(M^t_L)=\varphi(M^t_{K\cup L})+\varphi(M^t_{K\cap L})
	\end{align*} 
	for every $t\in [0,1]$. Taking derivatives in $t$ on both sides, \autoref{lem:MultLinExtensionAsDerivative} implies that
	\begin{align*}
		\varphi_\rho(K)+\varphi_\rho(L)=&\frac{1}{r!}\frac{d^r}{dt^r}\Big|_{0^+}\left(\varphi(M^t_K)+\varphi(M^t_L)\right)=\frac{1}{r!}\frac{d^r}{d^r}\Big|_{0^+}\left(\varphi(M^t_{K\cup L})+\varphi(M^t_{K\cap L})\right)\\
		=&\varphi_\rho(K\cup L)+\varphi_\rho(K\cap L).
	\end{align*}
	Thus $\varphi_\rho$ is a valuation as well. Since $\tilde{\varphi}$ is multilinear, it is now easy to check that $\varphi_\rho$ is a polynomial valuation.
\end{proof}
\begin{corollary}
	\label{cor:LocalizationGeneralCase}
	Let $\rho_1,\dots,\rho_r\in C^2(\S^{n-1})$. For every $\varphi\in \PVal_r(\RR^n)$, the map
	\begin{align*}
		\varphi_{\rho_1,\dots,\rho_r}(K):=	\hat{\varphi}_{\rho_1,\dots,\rho_r}(K,\dots,K)
	\end{align*}
	defines an element of $\PVal_r(\RR^n)$. Moreover, $\|\varphi_{\rho_1,\dots,\rho_r}\|\leq C_{n,r}\prod_{j=1}^r\|\rho_j\|_{C^2(\S^{n-1})}\|\varphi\|$.
\end{corollary}
\begin{proof}
	The inequality is a direct consequence of \autoref{lem:localizationPolarization}, which also shows that $\varphi_{\rho_1,\dots,\rho_r}$ is continuous. In order to see that it defines a valuation, note that the function $\rho(t_1,\dots,t_r):=\sum_{j=1}^r\|\rho_j\|_\infty+\sum_{j=1}^rt_j\rho_j$ is nonnegative and belongs to $C^2(\S^{n-1})$ for all $0\leq t_1,\dots,t_r\leq 1$. In particular, $\varphi_{\rho(t_1,\dots,t_r)}$ is a valuation for every $0\leq t_1,\dots,t_r\leq 1$ by \autoref{prop:localizationValuationProperty}. On the other hand, $(t_1,\dots,t_r)\mapsto \varphi_{\rho(t_1,\dots,t_r)}(K)$ is a polynomial in $t_1,\dots,t_r$ for every $K\in\K(\RR^n)$ by construction, and we have
	\begin{align*}
		\varphi_{\rho_1,\dots,\rho_r}(K)=\frac{1}{r!}\frac{\partial^r}{\partial t_1\dots\partial t_r}\Big|_{0^+}\varphi_{\rho(t_1,\dots,t_r)}(K).
	\end{align*} 
	Given $K,L\in \K(\RR^n)$ with $K\cup L$ convex, we thus obtain
	\begin{align*}
		\varphi_{\rho_1,\dots,\rho_r}(K)+\varphi_{\rho_1,\dots,\rho_r}(L)=&\frac{1}{r!}\frac{\partial^r}{\partial t_1\dots\partial t_r}\Big|_{0^+}\left(\varphi_{\rho(t_1,\dots,t_r)}(K)+\varphi_{\rho(t_1,\dots,t_r)}(L)\right)\\
		=&\frac{1}{r!}\frac{\partial^r}{\partial t_1\dots\partial t_r}\Big|_{0^+}\left(\varphi_{\rho(t_1,\dots,t_r)}(K\cup L)+\varphi_{\rho(t_1,\dots,t_r)}(K\cap L)\right)\\
		=&\varphi_{\rho_1,\dots,\rho_r}(K\cup L)+\varphi_{\rho_1,\dots,\rho_r}(K\cap L).
	\end{align*}
	Thus $\varphi_{\rho_1,\dots,\rho_r}$ is a valuation as well. Since $\tilde{\varphi}$ is multilinear, it is now easy to check that $\varphi_{\rho_1,\dots,\rho_r}$ is a polynomial valuation.
\end{proof}
\begin{remark}
	In general, $\hat{\varphi}_{\rho_1,\dots,\rho_r}$ is not the polarization of $\varphi_{\rho_1,\dots,\rho_r}$, however, from the definition of the polarization, we directly obtain
	\begin{align}
		\label{eq:relationPolarizationPhiHat}
		\bar{\varphi}_{\rho_1,\dots,\rho_r}(K_1,\dots,K_r)=\frac{1}{r!}\sum_{\sigma\in S_r}\hat{\varphi}_{\rho_{1}, \dots, \rho_r}(K_{\sigma(1)}, \dots, K_{\sigma(r)})
	\end{align}
	i.e. the polarization agrees with the symmetric part of $\hat{\varphi}_{\rho_1,\dots,\rho_r}$.
\end{remark}

The construction admits the following simple description in terms of the associated Goodey--Weil distributions.

\begin{corollary}
	\label{cor:LocalizationGWDistroAndVSupport}
	Let $\varphi\in \PVal_r(\RR^n)$ and $\rho_1,\dots,\rho_r\in C^\infty(\S^{n-1})$. Then 
	\begin{align}\label{eq:LocGWDistr}
		\GW(\varphi_{\rho_1,\dots,\rho_r})[\phi_1\otimes\dots\otimes \phi_r]=\frac{1}{r!}\sum_{\sigma\in S_r}\GW(\varphi)[\rho_1\phi_{\sigma(1)}\otimes\dots\otimes \rho_r\phi_{\sigma(r)}]
	\end{align}
	for all $\phi_1,\dots,\phi_r\in C^\infty(\S^{n-1})$. In particular,
	\begin{align*}
		\vsupp\varphi_{\rho_1,\dots,\rho_r}\subset \vsupp\varphi\cap \bigcap_{j=1}^r\supp\rho_j.
	\end{align*}
\end{corollary}
\begin{proof}
	Let $K_1, \dots, K_r \in \K(\RR^n)$ be smooth convex bodies with strictly positive Gauss curvature. Then, by \autoref{thm:GWbyGW} and Eq.~\eqref{eq:relationPolarizationPhiHat},
	\begin{align*}
		\GW(\varphi_{\rho_1,\dots,\rho_r})[h_{K_1}\otimes\dots\otimes h_{K_r}] &= \frac{1}{r!}\sum_{\sigma\in S_r}\hat{\varphi}_{\rho_{1}, \dots, \rho_r}(K_{\sigma(1)}, \dots, K_{\sigma(r)})\\
		 &= \frac{1}{r!}\sum_{\sigma\in S_r}\tilde{\varphi}(\rho_1h_{K_{\sigma(1)}}, \dots, \rho_rh_{K_{\sigma(r)}})\\
		 &= \frac{1}{r!}\sum_{\sigma\in S_r}\GW(\varphi)[\rho_1h_{K_{\sigma(1)}}\otimes \dots\otimes \rho_rh_{K_{\sigma(r)}}].
	\end{align*}
	Since every function in $C ^\infty(\S^{n-1})$ can be expressed as a difference of support functions of smooth convex bodies with strictly positive Gauss curvature, the first claim follows by multilinearity. The second follows from Eq.~\eqref{eq:LocGWDistr} using that the support of $\GW(\varphi)$ is contained in the diagonal.
\end{proof}

Note that we have a natural operation of $\GL(n,\RR)$ on $\PVal_r(\RR^n)$ given by $(g\cdot\varphi)(K):=\varphi(g^{-1}K)$ for $\varphi\in\PVal_r(\RR^n)$, $g\in\GL(n,\RR)$, $K\in\mathcal{K}(\RR^n)$. The next result shows that the construction above preserves the corresponding space of $\GL(n,\RR)$-smooth valuations, which we denote by $\PVal^\infty_r(\RR^n)$.
\begin{proposition}\label{prop:LocPolysAreSm}
	Let $\varphi\in \PVal_r(\RR^n)$ be $\GL(n,\RR)$-smooth and $\rho_1,\dots,\rho_r\in C^\infty(\S^{n-1})$. Then $\varphi_{\rho_1,\dots,\rho_r}\in \PVal_r(\RR^n)$ is $\GL(n,\RR)$-smooth.
\end{proposition}
\begin{proof}
	Define a representation of $\GL(n,\RR)$ on $C^2(\S^{n-1})$ by
	\begin{align*}
		(g\cdot\rho)(v)=\phi\left(\frac{g^Tv}{\|g^Tv\|}\right)
	\end{align*}
	for $g\in \GL(n,\RR)$, $\rho\in C^2(\S^{n-1})$, $v\in \S^{n-1}$. By \autoref{cor:LocalizationGeneralCase}, the map
	\begin{align}
		\label{eq:localizationMap}
		\begin{split}
			C^2(\S^{n-1})^r\times \PVal_r(\RR^n)&\rightarrow \PVal_r(\RR^n)\\		(\rho_1,\dots,\rho_r,\varphi)&\mapsto\varphi_{\rho_1,\dots,\rho_r}
		\end{split}
	\end{align}
	is well defined and continuous. We claim that it is $\GL(n,\RR)$-equivariant. In order to see this, let $\rho_1,\dots,\rho_r\in C^2(\S^{n-1})$, and $K\in\K(\RR^n)$ be given. Choose a convex body $M\in \K(\RR^n)$ such that $h_M+t\rho_j h_K$ is convex for all $t\in[-1,1]$ and denote the associated convex body by $M^t_j$. Since 
	\begin{align*}
		h_{g^{-1}M^t_j}(v)=&h_{M^t_j}(g^{-T}v)=\|g^{-T}v\|h_{M^t_j}\left(\frac{g^{-T}v}{\|g^{-T}v\|}\right)\\
		=&\|g^{-T}v\|h_{M}\left(\frac{g^{-T}v}{\|g^{-T}v\|}\right)+t\rho_j\left(\frac{g^{-T}v}{\|g^{-T}v\|}\right) \|g^{-T}v\|h_K\left(\frac{g^{-T}v}{\|g^{-T}v\|}\right)\\
		=&h_{g^{-1}M}(v)+t(g^{-1} \cdot \rho_j)(v) h_{g^{-1}K}(v),
	\end{align*}
	is convex for all $t\in[-1,1]$, \autoref{lem:MultLinExtensionAsDerivative} implies
	\begin{align*}
		\varphi_{g^{-1}\cdot\rho_1,\dots,g^{-1}\cdot\rho_r}(g^{-1}K)=&\tilde{\varphi}\left((g^{-1}\cdot \rho_1) h_{g^{-1}K},\dots,(g^{-1}\cdot \rho_r) h_{g^{-1}K}\right)\\
		=&\frac{1}{k!}\frac{\partial^r}{\partial t_1\dots\partial t_r}\Big|_0 \varphi\left(\sum_{j=1}^r g^{-1}M^{t_j}_j\right)\\
		=&\frac{1}{k!}\frac{\partial^r}{\partial t_1\dots\partial t_r}\Big|_0 (g\cdot\varphi)\left(\sum_{j=1}^r M^{t_j}_j\right)\\
		=&(g\cdot\varphi)_{\rho_1,\dots,\rho_r}(K).
	\end{align*}
	Replacing $\rho_1,\dots,\rho_r$ by $g\cdot \rho_1,\dots,g\cdot \rho_r$, we obtain the desired result.
	
	Since the map in Eq.~\eqref{eq:localizationMap} is $\GL(n,\RR)$-equivariant, multilinear, and continuous, it maps $\GL(n,\RR)$-smooth vectors to $\GL(n,\RR)$-smooth vectors. Thus we obtain a well defined map
	\begin{align*}
		C^\infty(\S^{n-1})^r\times \PVal^\infty_r(\RR^n)&\rightarrow \PVal^\infty_r(\RR^n)\\		(\rho_1,\dots,\rho_r,\varphi)&\mapsto\varphi_{\rho_1,\dots,\rho_r},
	\end{align*}
	which completes the proof.
\end{proof}

We may in particular apply the previous results to a partition of unity on $\S^{n-1}$, which provides the following decomposition.
\begin{theorem}\label{thm:locToPolyVal}
	Let $r\in\NN_0$, $\varphi \in \PVal_r(\RR^n)$, and $\rho_1, \dots, \rho_N \in C^\infty(\S^{n-1})$ be a partition of unity. Then there exist $\varphi_1, \dots, \varphi_M \in \PVal_r(\RR^n)$ such that
	\begin{align*}
		\varphi = \varphi_1 + \dots + \varphi_M
	\end{align*}
	and such that the vertical supports $\vsupp \varphi_i$, $i=1, \dots, M$, are subordinate to the supports $\supp \rho_i$, $i=1, \dots, N$. If $\varphi$ is $\GL(n,\RR)$-smooth, then $\varphi_1, \dots, \varphi_M$ can be chosen $\GL(n,\RR)$-smooth as well.
\end{theorem}
\begin{proof}
	Note that this is trivial for $r=0$, since $\vsupp\varphi=\emptyset$ in this case and we may thus choose $M=1$, $\varphi_1=\varphi$. Thus let $r\ge 1$. By the definition of $\tilde{\varphi}$, we have
	\begin{align*}
		\varphi(K)=&\tilde{\varphi}\left(h_K,\dots,h_K\right)=\tilde{\varphi}\left(\sum_{j_1=1}^N\rho_{j_1}h_K,\dots,\sum_{j_r=1}^N\rho_{j_r}h_K\right)\\
		=&\sum_{j_1,\dots,j_r=1}^N\tilde{\varphi}\left(\rho_{j_1}h_K,\dots,\rho_{j_r}h_K\right)\\
		=&\sum_{j_1,\dots,j_r=1}^N\varphi_{\rho_{j_1},\dots,\rho_{j_r}}(K).
	\end{align*}
	Since $\varphi_{\rho_{j_1},\dots,\rho_{j_r}}\in \PVal_r(\RR^n)$ by \autoref{cor:LocalizationGeneralCase}, and since the vertical supports of these valuations are subordinate to the supports of $\rho_j$ by \autoref{cor:LocalizationGWDistroAndVSupport}, this completes the proof of the first statement. The second claim is a direct consequence of \autoref{prop:LocPolysAreSm}.
\end{proof}

The proof of \autoref{mthm:locVal} now follows easily.
\begin{proof}[Proof of \autoref{mthm:locVal}]
	Using the homogeneous decomposition in \autoref{thm:HomDecompPVal}, we may assume that $\varphi$ is a homogeneous valuation. If we choose a finite partition of unity $\rho_1,\dots,\rho_N$ subordinate to the cover $(U_\alpha)_{\alpha\in\mathcal{A}}$, then the valuations constructed in \autoref{thm:locToPolyVal} have the desired properties.
\end{proof}

\section{Representation of smooth valuations by integration over the normal cycle}
\label{sec:RepSmoothVal}
The main goal of this section is the proof of \autoref{mthm:smoothValsByDiffform}. We will investigate how the the map $T: \VConv(\RR^n) \to \Val(\RR^{n+1})$ defined in \autoref{thm:defTmapProp} relates smooth valuation (in the sense of \autoref{def:affSmValConv}) to $\GL(n+1,\RR)$-smooth valuations and then establish a variant of \autoref{mthm:smoothValsByDiffform} for smooth polynomial valuations with support restrictions (\autoref{thm:PolySmDiffForm}). In combination with the localization procedure in \autoref{thm:locToPolyVal}, this finally completes the proof of \autoref{mthm:smoothValsByDiffform}.

\subsection{Smooth polynomial valuations}

We consider $\RR^n$ as a subgroup of $\GL(n+1,\RR)$ using the map $x\mapsto g_x$ defined by
\begin{align*}
	g^T_x(v,\lambda):=(v+\lambda x,\lambda), \quad (v, \lambda) \in \RR^n \times \RR \cong \RR^{n+1}.
\end{align*}
It is easy to check that this defines a homomorphism of Lie groups, in particular, $g^{-1}_x=g_{-x}$. Using the action of $\GL(n+1,\RR)$, we obtain a continuous representation of $\RR^n$ on $\Val(\RR^{n+1})$. This action is related to the representation of $\RR^n$ on $\VConv(\RR^n)$ in the following way.
\begin{lemma}\label{lem:Tequiv}
	The map $T:\VConv(\RR^n) \to \Val(\RR^{n+1})$ from \autoref{thm:defTmapProp} is equivariant with respect to the action of $\RR^n$ on both spaces, i.e.
	\begin{align*}
		T(\pi(x)\mu)=g_x\cdot T(\mu)
	\end{align*}
	for all $\mu\in\VConv(\RR^n)$ and $x\in\RR^n$.
\end{lemma}
\begin{proof}
	This is a straightforward calculation using the the properties of support functions in \autoref{lem:PropSupportFunction}: For $K\in\mathcal{K}(\RR^{n+1})$,
	\begin{align*}
		g_x \cdot (T(\mu)) (K) &= T(\mu)(g_x^{-1}K) = \mu(h_{g_x^{-1}K}(\cdot, -1)) = \mu(h_{K}(g_x^{-T}(\cdot, -1))) \\
		&= \mu(h_K(\cdot+x, -1)) = (\pi(x)\mu)(h_K(\cdot, -1)) = T(\pi(x)\mu)(K).
	\end{align*}
\end{proof}
This has the following consequence for smooth valuations.
\begin{lemma}\label{lem:TmuGLsmThenMuAffSm}
	Let $\mu \in \VConv(\RR^n)$. If $T(\mu)\in\Val(\RR^{n+1})$ is $\GL(n+1,\RR)$-smooth, then $\mu$ is smooth in the sense of \autoref{def:affSmValConv}.
\end{lemma}
\begin{proof}
	Since the representation of $\RR^n$ on $\VConv(\RR^n)$ is continuous (compare \cite{Knoerr2025}*{Lem.~3.4}), it is sufficient to show that the map 
	\begin{align*}
		\RR^n&\rightarrow \VConv(\RR^n)\\
		x&\mapsto \pi(x)\mu
	\end{align*}
	is smooth on a neighborhood $U$ of $0$. Assume that $U$ is bounded and let $A\subset \RR^n$ be a compact set containing $U$. From \autoref{prop:CharSupport}, we obtain that the support of $\pi(x)\mu$, $x \in A$, is contained in $B:=A+\supp\mu$, which is a compact subset since $\supp\mu$ is compact (compare \autoref{remark:suppCompact}). In particular, $T:\VConv_B(\RR^n)\rightarrow\Val_{P(B)}(\RR^{n+1})$ is a topological isomorphism by \autoref{thm:defTmapProp}.\\
	In combination with \autoref{lem:Tequiv}, this also implies $g_x\cdot T(\mu)=T(\pi(x)\mu)\in \Val_{P(B)}(\RR^{n+1})$ for all $x\in U$. Since $T(\mu)$ is a $\GL(n+1,\RR)$-smooth valuation by assumption, the map
	\begin{align*}
		U&\rightarrow \Val_{P(B)}(\RR^{n+1})\\
		x&\mapsto g_x\cdot T(\mu)
	\end{align*}
	is smooth. We may thus write
	\begin{align*}
		\pi(x)\mu=T^{-1}T(\pi(x)\mu))=T^{-1}(g_x\cdot T(\mu)),
	\end{align*}
	where the map $U\rightarrow \Val_{P(B)}(\RR^{n+1})$, $x\mapsto g_x\cdot T(\mu)$, is smooth, and $T^{-1}:\Val_{P(B)}(\RR^{n+1})\rightarrow\VConv_B(\RR^n)$ is continuous and linear. Thus $x\mapsto \pi(x)\mu$ is smooth on $U$.
	 This shows that $\mu$ is a smooth valuation in the sense of \autoref{def:affSmValConv}.
\end{proof}

Next, we will combine the previous result with \autoref{thm:SmVConvGivenByDiffForm} to obtain a version of \autoref{mthm:smoothValsByDiffform} under support restrictions. Since the localization procedure from \autoref{sec:localization} expresses any translation invariant valuation as a sum of \emph{polynomial} valuations with given support restrictions, we establish a more general version of this result for smooth polynomial valuations.

\begin{theorem}\label{thm:PolySmDiffForm}
	Let $r\in\NN_0$ and $\varphi \in \PVal_r(\RR^n)$ with $\vsupp \varphi \subset \S^{n-1}_-$ be given. If $\varphi$ is $\GL(n,\RR)$-smooth, then there exists a differential form $\omega \in \Omega^{n-1}(S\RR^n)$ such that
	\begin{align}\label{eq:thmPolySmDiffForm}
		\varphi(K) = \int_{\nc(K)} \omega, \quad K \in \K(\RR^n).
	\end{align}
\end{theorem}
\begin{proof}
	Throughout the proof, we will use the map $T:\VConv(\RR^{n-1})\rightarrow\Val(\RR^n)$ from \autoref{thm:defTmapProp}.\\
	Recall that $\PVal_r(\RR^n)$ admits a natural filtration given by $\mathrm{P}_d\Val_r(\RR^n)$, $0\leq d\leq r$, compare \autoref{cor:PValFiniteDegree}. We will establish the claim by induction on $d$.\\
	For $d=0$ and a smooth valuation $\varphi\in \Val_r(\RR^n)$ with the given support restriction, we obtain a valuation $\mu\in \VConv(\RR^{n-1})$ such that $T(\mu)=\varphi$ due to the description of the image of $T$ in \autoref{thm:defTmapProp}. By \autoref{lem:TmuGLsmThenMuAffSm}, $\mu$ is a smooth valuation in the sense of \autoref{def:affSmValConv}, so \autoref{thm:SmVConvGivenByDiffForm} shows that $\mu$ is representable by integration with respect to the differential cycle. Thus \autoref{lem:diffformSmIffValVConv} shows that $\varphi=T(\mu)$ is representable by integration with respect to the normal cycle. In fact, \autoref{lem:diffformSmIffValVConv} implies that $\varphi=\int_{\nc(\cdot)}\omega$ for a translation invariant differential form $\omega\in\Omega^{n-1}(S\RR^n)^{\mathrm{tr}}$ with support contained in $\RR^n\times \S^{n-1}_-$. This stronger version for $d=0$ will be used in the induction step. This completes the case $d=0$ for every $r\in\NN_0$.\\
	Now assume that the claim holds for all valuations in $\mathrm{P}_{d-1}\Val_s(\RR^n)$ for all $s\in\NN_0$. Let $\varphi\in \mathrm{P}_d\Val_r(\RR^n)$ be a $\GL(n,\RR)$-smooth valuation with the given support restriction. We may thus uniquely write
	\begin{align}\label{eq:prfPolyDiffValPolynom}
		\varphi(K + x) = \sum_{i=0}^{d} Y_{i,\varphi}(K)[x],
	\end{align}
	where $Y_{i,\varphi}:\mathcal{K}(\RR^n)\rightarrow \Sym^i((\RR^n)^\ast)_\CC$ is a continuous valuation with values in the space of complex-valued $i$-homogeneous polynomials on $\RR^n$, $i=0, \dots, d$. Moreover, comparing the degrees in $x\in\RR^n$, $Y_i$ is a polynomial valuation of degree at most $d-i$, i.e. $Y_{i,\varphi}\in \mathrm{P}_{d-i}\Val_{r-i}(\RR^n)\otimes \Sym^i((\RR^n)^\ast)_\CC$. In particular, $Y_{d,\varphi} \in \Val_{r-d}(\RR^n) \otimes \Sym^d((\RR^n)^\ast)_\CC$.\\
	
	We claim that the vertical support of $Y_{i,\varphi}$ is contained in $\vsupp \varphi \subset \S^{n-1}_-$. In order to see that, note that for $t \in \RR$ and $x \in \RR^n$, 
	\begin{align*}
		\varphi(K + tx) = \sum_{i=0}^{d} t^i Y_{i,\varphi}(K)[x]
	\end{align*}
	so plugging in $t=0,\dots,d$ and inverting the associated Vandermonde matrix, we obtain constants $c_{ij}\in\RR$, $0\leq i,j\leq d$, independent of $\varphi$, $K$, and $x$ such that
	\begin{align}
		\label{eq:expressionYd}
		Y_{i,\varphi}(K)[x]=\sum_{j=0}^dc_{ij}\varphi(K + jx).
	\end{align} 
	If two bodies $K, L \in \K(\RR^n)$ satisfy $h_K=h_L$ on a neighborhood of $\vsupp \varphi$, so do the support functions of $K + t x$ and $L+t x$ for $t\in\RR$, so $\varphi(K + tx) = \varphi(L + tx)$ by \autoref{prop:CharVsupp}. In particular, $Y_{i,\varphi}(K)[x] = Y_{i,\varphi}(L)[x]$, which implies $\vsupp Y_{i,\varphi}\subset \vsupp\varphi$ by \autoref{prop:CharVsupp}.\\
	
	Write $Y_{d,\varphi} = \sum_{l} Y_d^l \otimes p_l$, $Y_d^l \in \Val(\RR^n)$ for some arbitrary basis $p_l$ of $\Sym^d((\RR^n)^\ast)_\CC$. We claim that each $Y_d^l$ is $\GL(n,\RR)$-smooth. \\
	From Eq.~\eqref{eq:expressionYd} we obtain for $g\in \GL(n,\RR)$,
	\begin{align}
		\label{eq:YdEquiv}
		Y_{d,\varphi}(g^{-1}K)[g^{-1}x]=Y_{d,g\cdot \varphi}(K)[x].
	\end{align}
	Consider the map
	\begin{align}
		\label{eq:equivMapTensorProduct}
		\begin{split}
			\mathrm{P}_d\Val_r(\RR^n)&\rightarrow \Val_{r-d}(\RR^n) \otimes \Sym^d((\RR^n)^\ast)_\CC\\
			\nu&\mapsto Y_{d,\nu}.
		\end{split}
	\end{align}
	From Eq.~\eqref{eq:expressionYd}, we see that this map is continuous. If we equip the tensor product $\Val_{r-d}(\RR^n) \otimes \Sym^d((\RR^n)^\ast)_\CC$ with the natural action of $\GL(n,\RR)$ on both factors, then Eq.~\eqref{eq:YdEquiv} implies that the map in Eq.~\eqref{eq:equivMapTensorProduct} is $\GL(n,\RR)$-equivariant. In particular, since $\varphi$ is $\GL(n,\RR)$-smooth, so is $Y_{d,\varphi}$. Since $\Sym((\RR^n)^\ast)_\CC$ is a finite dimensional representation of $\GL(n,\RR)$, \cite{Alesker2004b}*{Lem.~1.5} implies that the $\GL(n,\RR)$-smooth vectors in $\Val(\RR^n)\otimes \Sym^d((\RR^n)^\ast)_\CC$ coincide with $\Val^\infty(\RR^n)\otimes \Sym^d((\RR^n)^\ast)_\CC$. Thus the components $Y_d^l\in\Val_{r-d}(\RR^n)$ are $\GL(n,\RR)$-smooth.\\ 

	We may apply the case $d=0$ to the valuations $Y_d^l\in\Val_{r-d}(\RR^n)$ and obtain differential forms $\omega_{d,l}$ such $Y_d^l(K)=\int_{\nc(K)}\omega_{d,l}$. Moreover, by the previous discussion, we may in addition assume that $\omega_{d,l}$ is translation invariant and that its support is contained in $\RR^n\times \S_-^{n-1}$. Define $\omega_d \in \Omega^{n-1}(S\RR^n)$ by setting
	\begin{align*}
		\omega_d|_{(x,v)} = \sum_{l} p_l(x) \omega_d^l|_{(x,v)}, \quad (x,v) \in S\RR^n.
	\end{align*}
	Then the support of $\omega_d$ is contained in $\RR^n\times \S_-^{n-1}$. Let 
	$\varphi_d=\int_{\nc(\cdot)}\omega_d$. Then $\varphi_d$ is a $\GL(n,\RR)$-smooth valuation and it is easy to see that it is polynomial of degree at most $d$. In fact, for $K \in \K(\RR^n)$
	\begin{align*}
		\varphi_d(K + x) &= \int_{\nc(K+x)} \omega_d = \int_{\nc(K)} \sum_{l} p_l(\cdot + x) \omega_d^l\\
		&=\int_{\nc(K)} \sum_{l} \left( p_l(x) +  q_l(\cdot, x)\right) \omega_d^l,
	\end{align*}
	where $q_l(\cdot, x)$ is a polynomial with degree in $x$ strictly less than $d$. Hence,
	\begin{align*}
		\varphi_d(K + x) &= \sum_{l} p_l(x) \int_{\nc(K)} \omega_d^l + \sum_{l}  \int_{\nc(K)} q_l(\cdot, x) \omega_d^l \\
		&= \sum_{l} p_l(x) Y_d^l(K) + \sum_{l}  \int_{\nc(K)} q_l(\cdot, x) \omega_d^l = Y_{d,\varphi}(K)[x]+ \sum_{l}  \int_{\nc(K)} q_l(\cdot, x) \omega_d^l,
	\end{align*}
	and we conclude that $\varphi - \varphi_d$ is a polynomial valuation of degree strictly less than $d$. Since $\omega_d$ is supported on $\RR^n\times \S^{n-1}_-$, \autoref{prop:CharVsupp} implies that the vertical support of $\varphi_d$ is contained in $\S^{n-1}_-$. In particular, the vertical support of $\varphi - \varphi_d$ is again contained in $\S^{n-1}_-$. We may thus apply the induction assumption to the $\GL(n,\RR)$-smooth valuation $\varphi - \varphi_d\in \mathrm{P}_{d-1}\Val(\RR^n)$ to see that this valuation is representable by integration with respect to the normal cycle. This completes the proof.
\end{proof}

\subsection{Proof of \autoref{mthm:smoothValsByDiffform}}

\autoref{mthm:smoothValsByDiffform} is the special case $d=0$ of the following result for polynomial valuations. The proof of this result combines \autoref{thm:PolySmDiffForm} with the localization procedure in \autoref{thm:locToPolyVal}.
\begin{theorem}\label{thm:smoothPValRepresentable}
	Let $d\in\NN_0$ and $\varphi\in \mathrm{P}_d\Val(\RR^n)$ be a $\GL(n,\RR)$-smooth valuation. Then $\varphi$ is representable by integration with respect to the normal cycle. 
\end{theorem}
\begin{proof}
	Due to the homogeneous decomposition in \autoref{thm:HomDecompPVal}, we may assume that $\varphi\in \PVal_r(\RR^n)$. Choose a smooth partition of unity $\rho_1, \dots, \rho_N \in C^\infty(\S^{n-1})$ with the property that the support of $\rho_i$ is contained in an open hemisphere, $\supp\rho_i \subset \{u \in \S^{n-1}: \pair{u}{u_i} < 0\}$ for some $u_i \in \S^{n-1}$, for every $i=1, \dots, N$.
	
	By \autoref{thm:locToPolyVal}, there exist smooth polynomial valuations $\varphi_1, \dots, \varphi_M \in \PVal_r(\RR^n)$ such that $\varphi = \varphi_1 + \dots + \varphi_M$, with vertical supports subordinate to the supports of the $\rho_i$. In particular, $\vsupp \varphi_j \subset \supp \rho_{i_j} \subset \{u \in \S^{n-1}: \pair{u}{u_{i_j}} < 0\}$ for some $1\leq i_j\leq N$. We may thus apply \autoref{thm:PolySmDiffForm} to obtain for $1\leq j\leq N$ a differential form $\omega_j \in \Omega^{n-1}(S\RR^n)$, such that
	\begin{align*}
		\varphi_j(K) = \int_{\nc(K)} \omega_j, \quad K \in \K(\RR^n).
	\end{align*}
	We conclude that
	\begin{align*}
		\varphi(K) = \varphi_1(K) + \dots + \varphi_M(K) = \int_{\nc(K)} \left( \omega_1 + \dots + \omega_M\right), \quad K \in \K(\RR^n),
	\end{align*}
	so $\varphi$ is representable by integration with respect to the normal cycle. This concludes the proof.
\end{proof}
\begin{remark}
	Note that the differential form constructed in the previous proof does not inherit the invariance properties of the valuation $\varphi$. In particular, if $\varphi$ is translation invariant, then the differential form does not need to be translation invariant. In order to pass from this representation to invariant differential forms, we require some results by Bernig and Bröcker from \cite{Bernig2007}, which we discuss in \autoref{sec:prelimImD}. Due to these results, the calculations in the next section only involve translation invariant differential forms.
\end{remark}

\part{Irreducibility, proofs of \autoref{mthm:imDIrred} and \autoref{mcor:ValslIrred}}\label{part2}
\section{The action of the Lie algebra $\sln(n)_\CC$ on differential forms}\label{sec:actLieAlgdf}
In this section we collect all calculations needed in \autoref{sec:prfIrredThm} in the proof of \autoref{mthm:imDIrred}. Since we will reduce the proof of \autoref{mthm:imDIrred} to a result for a corresponding space of differential forms, these calculations mostly involve applying the Lie derivative along the fundamental vector fields of suitable elements of $\sln(n)_\CC$ to highest weight vectors in the relevant spaces of differential forms and showing that the pairing from \autoref{sec:pairing} does not vanish on certain combinations of these expressions. Since these differential forms were constructed in \cite{Kotrbaty2022} using double forms, all calculations will be performed in the same framework. Let us remark that despite the simplifications enabled by the use of double forms, most of the calculations are very lengthy and the relations established in this section are not very insightful in isolation. At a first reading, the authors recommend to skip this section and continue with \autoref{sec:prfIrredThm}.

\subsection{Double forms}\label{sec:doubleforms}
We will use the conventions and notation from \cite{Kotrbaty2022} and refer to the same article for a more thorough discussion of double forms.

In our setting, we will consider double forms on $M=\RR^n \times \RR^n$ or $M=S\RR^n$. A double form on $M$ is a section of the bundle $(\CC \otimes \Lambda^\ast T^\ast M) \otimes (\CC \otimes \Lambda^\ast (\RR^n \times \RR^n)^\ast)$ over $M$. It is of bi-degree $(j,k) \in \NN_0 \times \NN_0$ if it is a section of $(\CC\otimes \Lambda^j T^\ast M) \otimes (\CC\otimes \Lambda^k(\RR^n \times \RR^n)^\ast)$. We denote by $\overline{\omega}$ complex conjugation on the first factor of a double form $\omega$, that is,
\begin{align*}
	\overline{\omega} = \overline{\eta} \otimes \tau
\end{align*}
whenever $\omega= \eta \otimes \tau$ for some $\eta \in \Omega^\ast(M)$ and $\tau \in \Omega^\ast(\RR^n \times \RR^n)$. Similarly, the exterior differential naturally extends to double forms such that $d(\eta\otimes\tau)=d\eta\otimes \tau$. There is a natural wedge product on double forms that respects the grading and is determined by
\begin{align*}
	(\eta \otimes \tau) \wedge (\eta' \otimes \tau') = (\eta \wedge \eta') \otimes (\tau \wedge \tau').
\end{align*}
The $\wedge$-sign will be omitted later on to obtain more concise formulas. We further set $\omega^{[k]} := \frac{1}{k!}\omega^k$, where $\omega$ is a double form and $\omega^k$ the $k$-fold wedge product of $\omega$ with itself. Note that for double forms $\alpha, \beta$ of bidegree $(1,1)$ the following binomial formula holds:
\begin{align}\label{eq:binomDFCalc}
	(\alpha + \beta)^{[k]} = \sum_{j=0}^k \alpha^{[j]}\beta^{[k-j]}.
\end{align}

\medskip

Next, we fix some notation needed to define the relevant differential forms. Recall that $n \geq 2$ and denote by $x_1, \dots, x_n, \xi_1, \dots, \xi_n$ the standard coordinates on $\RR^n \times \RR^n$. Put $l = \lfloor\frac{n}{2}\rfloor$. Then for every $j \in \{1, \dots, l\}$, we define
\begin{align*}
	z_j &= \frac{1}{\sqrt{2}}(x_{2j-1} + \I x_{2j}), \quad & z_{\overline{j}} &= \frac{1}{\sqrt{2}}(x_{2j-1} - \I x_{2j}),\\
	\zeta_j &= \frac{1}{\sqrt{2}}(\xi_{2j-1} + \I \xi_{2j}), &	\zeta_{\overline{j}} &= \frac{1}{\sqrt{2}}(\xi_{2j-1} - \I \xi_{2j}).
\end{align*}
If $n=2l+1$ is odd, we also set $z_{l+1} = x_{2l+1}$ and $\zeta_{l+1} = \xi_{2l+1}$.

We define the corresponding set of indices by
\begin{align*}
	\Iind = \begin{cases}
		\{1, \bar{1}, \dots, l, \bar{l}\}, & \text{ for } n = 2l,\\
		\{1, \bar{1}, \dots, l, \bar{l}, l+1\}, & \text{ for } n =2l+1.
	\end{cases}
\end{align*}
and order $\Iind$ by $1 \prec \bar 1 \prec \dots \prec l \prec \bar{l} \prec l+1$.

Using the forms above, we define the following double forms on $\RR^n \times \RR^n$ for every $I = \{i_1 \prec \dots \prec i_j\} \subset \Iind$:
\begin{align*}
	\nu_I &= \sum_{i \in I} \zeta_{\overline{i}}\zeta_{i},&&\\
	\alpha_I &= \sum_{i \in I} \zeta_{\overline{i}} dz_i, \quad &\gamma_I &= \sum_{i \in I} \zeta_{\overline{i}} d\zeta_i,\\
	z_I &= \sum_{i \in I} z_i \otimes dz_i, & \zeta_I &= \sum_{i \in I} \zeta_{i} \otimes dz_i,\\	
	\Theta_{1,I} &= dz_{i_1}\dots dz_{i_j}, & \Theta_{2,I} &=d\zeta_{i_1} \dots d\zeta_{i_j}, \quad \Theta_I = \Theta_{1,I}\Theta_{2,I}.
\end{align*}
Here, we write $\overline{\overline{i}} = i$ and $\overline{l+1} = l+1$.
Note that $\nu_I$ is a function, $\alpha_I$ and $\gamma_I$ can be interpreted either as a $1$-form or as a $(1,0)$-double form, while $z_I$ and $\zeta_I$ are $(0,1)$-double forms.

If $I = \Iind$, then we write $\alpha = \alpha_{\Iind}$, $\gamma = \gamma_{\Iind}$. Note that $\alpha|_{S\RR^n}$ coincides with the contact form on $S\RR^n$ as defined in \eqref{eq:defAlphSRn}. The form $\gamma$ vanishes on all submanifolds $\RR^n \times \lambda \S^{n-1}$, $\lambda > 0$. Therefore, if $\omega \wedge \gamma = 0$ for some $\omega \in \Omega(\RR^n \times \RR^n)$, then necessarily $\omega|_{S\RR^n} = 0$. Let us further point out here that $\nu(x,\xi) = \| \xi \|_2^2$, that is, $\nu|_{\RR^n \times \S^{n-1}} \equiv 1$.

For $k \in \NN$, $k \leq l$, we write $K = \{1, \dots, k\} \subset \Iind$, $J = \Iind \setminus K$ and $L = J \setminus \overline{K}$, where $\overline{I} = \{\overline{i}: i \in I\}$ for any $I \subset \Iind$. If $k \leq l-1$, then we write $K^+ = K \cup \{k+1\}$ and $L^- = L \setminus\{k+1, \overline{k+1}\}$. For $I \subset \Iind$ and $i_1, \dots, i_j \in \Iind$, we write $I_{i_1, \dots, i_j}=I \setminus\{i_1, \dots, i_j\}$.\\

The following forms $\omega_{r,k}$ and $\omega_{r,k,m}$ were introduced in \cite{Kotrbaty2022}:
\begin{align} \label{eq:defOmegaRkm}
	\omega_{r,k} \otimes \Theta_1 &= \zeta_J d\zeta_J^{[n-r-1]} dz_J^{[r-k]}\overline{dz_K}^{[k]},\\
	\label{eq:defOmegaRkm_m}
	  \omega_{r,k,m} &= \zeta_{\overline{1}}^{m-2} \omega_{r,k} \in \Omega^{r,n-r-1}(S\RR^{n})^{\mathrm{\mathrm{tr}}},
\end{align}
where $r,k,m \in \NN$ with $1\leq r \leq n-1$, $1\leq k \leq \min\{r, n-r\}$ and $m \geq 2$. If $n=2l$ is even, then we additionally define
\begin{align}\label{eq:defOmegalmlm}
	\omega_{l,-l} \otimes \Theta_1 = \zeta_{\overline{M}} d\zeta_{\overline{M}}^{[l-1]}\overline{dz_M}^{[l]}, \qquad \omega_{l,-l, m} = \zeta_{\overline{1}}^{m-2}\omega_{l,-l}\in \Omega^{l,l-1}(S\RR^{n})^{\mathrm{\mathrm{tr}}},
\end{align}
where $M = \{1, \dots, l-1, \overline{l}\}$. These forms are the highest weight vectors of weight $\lambda_{k,m}$ of the $\SO(n)$-representation on differential forms defined by Eq.~\eqref{eq:actGLOnVal} (see \cite{Kotrbaty2022}*{Thm.~4.2}), where 
\begin{align*}
	\lambda_{k,m} = \begin{cases}
		(m, \underbrace{2, \dots, 2}_{k-1}, 0, \dots, 0) \in \ZZ^l & \text{ for } k=1, \dots, l,\\
		(m, 2 \dots, 2, -2) \in \ZZ^l & \text{ for } k=-l.
	\end{cases}
\end{align*}
In particular, these differential forms correspond to highest weight vectors in $\mathcal{V}^{\infty,\mathrm{tr}}$, compare \autoref{sec:prelimImD}. In \autoref{sec:weight0Form} we will also consider the relevant forms for the weight $\lambda_{1,0}:=0$, which corresponds to $\SO(n)$-invariant differential forms.\\
\begin{remark}
	\label{remark:holomorphicExtensionOmega_rkm}
	We may use Eq.~\eqref{eq:defOmegaRkm_m} to define $\omega_{r,k,m}$ and $\omega_{-l,l,m}$ for arbitrary $m\in\mathbb{C}$ by setting $\zeta_{\bar1}^{m-2}:=\exp((m-2)\mathrm{Ln}(\zeta_{\bar 1}))$ for the principal value of the logarithm. Then $\omega_{r,k,m}$ and $\omega_{-l,l,m}$ are smooth differential forms on the open subset of $S\RR^n$ given by $\zeta_1\notin(-\infty,0]$. We will use this extension to simplify some calculations in \autoref{sec:weight0Form}.
\end{remark}

We further need the differential forms
\begin{align}
	\sigma_{r,k}\otimes \Theta_1 &= \overline{d\zeta_K}^{[k]} \zeta_J d\zeta_J^{[n-r-k]}dz_J^{[r-1]},\\
	\label{eq:defTau}
	\tau_{r,k} \otimes \Theta_1 &= \overline{\zeta_K} \overline{d\zeta_K}^{[k-1]}d\zeta_J^{[n-r-k+1]}dz_J^{[r-1]},\\
	\delta_{r,k} \otimes \Theta_1 &= d\zeta_J^{[n-r]}dz_J^{[r-k]} \overline{dz_K}^{[k]},\\
	\theta_{r,k} \otimes \Theta_1 &= \overline{\zeta_K}\, \overline{d\zeta_K}^{[k-1]} \zeta_J d\zeta_J^{[n-r-k]}dz_J^{[r-1]}. \label{eq:defThetark}
\end{align}

For later reference we note the following relations, mostly contained in \cite{Kotrbaty2022}.
\begin{proposition}\label{prop:OmSigTauSepL}
	Let $r,k,m \in \NN$ with $r \leq n-1$, $k \leq \min\{r, n-r\}$ and $m \geq 2$. Then the following relations hold on $\RR^n \times \RR^n$:
	\begin{align}
		\omega_{r,k} \otimes \Theta_1 =& \zeta_{\overline{K}}d\zeta_{\overline{K}}^{[k-1]}d\zeta_L^{[n-r-k]}dz_L^{[r-k]}\overline{dz_K}^{[k]}
		+\zeta_L d\zeta_{\overline{K}}^{[k]}d\zeta_L^{[n-r-k-1]}dz_L^{[r-k]}\overline{dz_K}^{[k]} \label{eq:OmSepL} \\
		\sigma_{r,k} \otimes \Theta_1 =& \overline{d\zeta_K}^{[k]}\zeta_{\overline{K}} d\zeta_L^{[n-r-k]} dz_L^{[r-k]}dz_{\overline{K}}^{[k-1]} + \overline{d\zeta_K}^{[k]}\zeta_L d\zeta_L^{[n-r-k]}dz_L^{[r-k-1]}dz_{\overline{K}}^{[k]}\label{eq:sigmarkSep}\\
		\begin{split}\tau_{r,k} \otimes \Theta_1 =& (-1)^{k-1}\overline{d\zeta_K}^{[k]}\zeta_{\overline{K}} d\zeta_L^{[n-r-k]}dz_L^{[r-k]}dz_{\overline{K}}^{[k-1]} \\
		&+ \overline{\zeta_K}\overline{d\zeta_K}^{[k-1]}d\zeta_L^{[n-r-k+1]}dz_L^{[r-k-1]}dz_{\overline{K}}^{[k]}
		\end{split}\label{eq:taurkSep}\\
		\delta_{r,k} \otimes \Theta_1 =& (-1)^k \overline{d\zeta_K}^{[k]}d\zeta_L^{[n-r-k]}dz_L^{[r-k]} dz_{\overline{K}}^{[k]}\label{eq:deltarkSep}
	\end{align}
	Moreover, the Rumin differential of (the restriction to $S\RR^n$ of) $\omega_{r,k,m}$ is given by 
	\begin{align}\label{eq:DOmrkm}
		D \omega_{r,k,m} = c_{r,m} \zeta_{\overline{1}}^{m-2} \left[(m+k-1) \sigma_{r,k} + (-1)^{k+1}(n-r-k+1)\tau_{r,k}\right]\alpha,
	\end{align}
	where $c_{r,m} = (-1)^{n+1}(n+m-r-2)$.
\end{proposition}
\begin{proof}
	Eq.~\eqref{eq:OmSepL} is the content of \cite{Kotrbaty2022}*{Prop.~4.3}, Eq.~\eqref{eq:sigmarkSep} is \cite{Kotrbaty2022}*{Prop.~5.2(17)}, Eq.~\eqref{eq:deltarkSep} is \cite{Kotrbaty2022}*{Prop.~5.2(16)}, and Eq.~\eqref{eq:DOmrkm} is the content of \cite{Kotrbaty2022}*{Thm.~5.1}. It therefore remains to show Eq.~\eqref{eq:taurkSep}, which is similar to the calculation for $\sigma_{r,k}$ in \cite{Kotrbaty2022}. Indeed, the binomial formula \eqref{eq:binomDFCalc} shows that
	\begin{align*}
		\tau_{r,k}\otimes \Theta_1=& \overline{\zeta_K}\overline{(d\zeta_K)}^{[k-1]}d\zeta_{\overline{K}}(d\zeta_L)^{[n-r-k]}(dz_L)^{[r-k]}(dz_{\overline{K}})^{[k-1]}\\
		&+\overline{\zeta_K}\overline{(d\zeta_K)}^{[k-1]}(d\zeta_L)^{[n-r-k+1]}(dz_L)^{[r-k-1]}(dz_{\overline{K}})^{[k]},
	\end{align*}
	where we can now exchange the terms between $\overline{\zeta_K}$ and $d\zeta_{\overline{K}}$ in the first product.
\end{proof}
\begin{remark}
	\label{remark:holomorphicDomegakm}
	The proof in \cite{Kotrbaty2022}*{Thm.~ 5.1} shows that Eq.~\eqref{eq:DOmrkm} holds for arbitrary $m\in\CC$ (for the forms $\omega_{r,k,m}$ from \autoref{remark:holomorphicExtensionOmega_rkm}) on the open set of $S\RR^n$ where $\zeta_{\bar{1}}\notin(-\infty,0]$.
\end{remark}

\subsection{The action of $\sln(n)_\CC$ on double forms on $S\RR^n$}\label{sec:actGlncVal}
Recall from \autoref{sec:GLRnrepOnVal} that the fundamental vector field $\widetilde{W}$ of $W \in \sln(n)$ is given by
\begin{align}\label{eq:LieAlgActVecField}
	\widetilde{W} = (-W x, W^T v - v \pair{v}{W^T v})
\end{align}
and the action of $W$ on $\Omega^*(S\RR^n)^{\mathrm{tr}}$ corresponding to the action on valuations is given by taking the Lie derivative with respect to $\widetilde{W}$.\\

The action of $\GL(n,\RR)$ on $\Omega^*(S\RR^n)^{\mathrm{tr}}$ naturally extends to the double forms considered in the previous section.  More precisely, all of the relevant forms belong to $\Omega^{*}(S\RR^n)\otimes(\Lambda^*(\RR^n)^*)_\CC$, i.e., the double forms do not depend on the second set of variables of $\RR^n\times\RR^n$. We consider the second component $\Lambda^*(\RR^n)^*)_\CC$ as a subspace of $\Omega^*(S\RR^n)^{\mathrm{tr}}$ and define an action of $g\in\GL(n,\RR)$ on this space by
\begin{align*}
	g\cdot (\omega\otimes \nu):=(G_{g^{-1}}^\ast \omega) \otimes (G_{g^{-1}}^\ast \nu)
\end{align*}
for $\omega\in \Omega^*(S\RR^n)^{\mathrm{tr}}$, $ \nu \in \Lambda^*(\RR^n)^*_\CC$. Note that this action commutes with the product on $\Omega^*(S\RR^n)^{\mathrm{tr}}\otimes \Lambda^*(\RR^n)^*_\CC$ i.e. $g\cdot(\omega_1\omega_2)=(g\cdot\omega_1)(g\cdot\omega_2)$ for $\omega_1,\omega_2\in \Omega^*(S\RR^n)^{\mathrm{tr}}\otimes \Lambda^*(\RR^n)^*_\CC$.\\
Thus, $W\in \gl(n)_\CC$ acts on $\Omega^*(S\RR^n)^{\mathrm{tr}}\otimes \Lambda^*(\RR^n)^*_\CC$ by 
\begin{align}\label{eq:lieDerDoubleForm}
	W\bullet (\omega \otimes \nu) = (\LieDer_{\widetilde{W}} \omega) \otimes \nu + \omega \otimes (\LieDer_{\widetilde{W}} \nu)
\end{align}
for $\omega \in \Omega^*(S\RR^n)^{\mathrm{tr}}$ and $\nu \in \Lambda^*(\RR^n)^*_\CC$.
In particular, as $\Theta_1$ is the volume form on the first component of $\RR^n \times \RR^n$ and thus invariant under $\SL(n,\RR)$, 
\begin{align}\label{eq:LieDerYabTensorTheta}
	W\bullet (\omega \otimes \Theta_1) = (\LieDer_{\widetilde{W}} \omega) \otimes \Theta_1
\end{align}
for all $W \in \sln(n)_\CC$. Since the action of $\GL(n,\RR)$ commutes with the product of these double forms, this action satisfies
\begin{align}
	\label{eq:actionProductRule}
	W\bullet (\omega_1\omega_2)=(W\bullet \omega_1)\omega_2+\omega_1(W\bullet \omega_2)
\end{align}
for $\omega_1,\omega_2\in \Omega^*(S\RR^n)^{\mathrm{tr}}\otimes \Lambda^*(\RR^n)^*_\CC$. In particular, since the forms we are considering are products of suitable combinations of double forms, this reduces the calculation of the action of $\sln(n)_\CC$ on these forms to calculations of the action of $W$ on these building blocks, which is given by the Lie derivatives in Eq.~\eqref{eq:lieDerDoubleForm}.

\medskip

In addition, we will only need to calculate the action of a selected number of elements in $\sln(n)_\CC$. To define them, consider the $(2\times 2)$-matrix
\begin{align*}
%	x_{-} = \begin{pmatrix}
%		1 & \I \\
%		-\I & 1
%	\end{pmatrix}, \quad 
	x_{+} = \begin{pmatrix}
		1 & -\I \\
		-\I & -1
	\end{pmatrix},
\end{align*}
and denote by $Y_{ab} = Y_{ba} \in \sln(n)_\CC$ for $1 \leq a,b \leq l$ the matrix given block-wise by 
\begin{align}
	(Y_{ab})_{ij} = \begin{cases}
		-\frac{1}{4} x_{+} & \text{ if $i=a$ and $j=b$},\\
		-\frac{1}{4} x_{+}^T & \text{ if $i=b$ and $j=a$},\\
		0 & \text{ else},
	\end{cases} \quad i,j \in \{1, \dots, l\},
\end{align}
whenever $a \neq b$, and by
\begin{align}
	(Y_{aa})_{ij} = \begin{cases}
		-\frac{1}{2}x_+   & \text{ if $i=a$ and $j=a$},\\
		0 & \text{ else},
	\end{cases} \quad i,j \in \{1, \dots, l\}.
\end{align}
Let us point out that $Y_{ab}$ belongs to the weight space of $\epsilon_a + \epsilon_b$ of the representation of $\SO(n)$ on $\sln(n)_\CC$ induced by conjugation. In particular, applying $Y_{ab}$ to a weight vector in $\Omega^*(S\RR^n)^{\tr}$ will produce a weight vector whose weight is given by the sum of both weights. For highest weight vectors, this will in general not produce a highest weight vector, but a vector with components belonging to different irreducible representations. Now the main idea is to choose the element $Y_{ab}$ with largest possible weight in order to obtain a vector with non-vanishing component proportional to the desired highest weight vector, which can be checked by computing the invariant pairing from \autoref{sec:pairing}.\\

In order to determine the Lie derivatives along $\widetilde{Y_{ab}}$, we need some calculations.
\begin{lemma}\label{lem:actionYabdzZeta}
	For $1 \leq a,b, k \leq l$, we have for $x\in \RR^n$, $v\in\S^{n-1}$,
	\begin{align*}
		dz_k (-Y_{ab} x) &= \frac{1}{2}\delta_{ka} z_{\overline{b}} + \frac{1}{2}\delta_{kb}z_{\overline{a}}, & d\zeta_k(Y_{ab}^Tv - \pair{v}{Y_{ab}^T v}v) &= -\frac{1}{2}\delta_{ka} z_{\overline{b}} - \frac{1}{2}\delta_{kb}z_{\overline{a}} + \zeta_{\overline{a}}\zeta_{\overline{b}}\zeta_{k},\\
		dz_{\overline{k}} (-Y_{ab} x) &=0, & d\zeta_{\overline{k}}(Y_{ab}^Tv - \pair{v}{Y_{ab}^T v}v) &= \zeta_{\overline{a}}\zeta_{\overline{b}}\zeta_{\overline{k}}.
	\end{align*}
\end{lemma}
\begin{proof}
	As
	\begin{align*}
		dx_{2k-1}\left(-Y_{a,b}x\right)=&\frac{1}{4}\delta_{ka}\left(	x_{2b-1} - i x_{2b}\right)+ \frac{1}{4}\delta_{kb}\left(x_{2a-1} - i x_{2a}\right),\\
		dx_{2k}\left(-Y_{a,b}x\right)=&\frac{1}{4}\delta_{ka}\left(-x_{2b} - ix_{2b-1}\right)+\frac{1}{4}\delta_{kb}\left(	-x_{2a} - ix_{2a-1}\right),
	\end{align*}
	we conclude that $dz_k(-Y_{a,b}x) = \frac{1}{2}\delta_{ka}z_{\overline{b}} + \frac{1}{2}\delta_{kb}z_{\overline{a}}$ and $dz_{\overline{k}} (-Y_{ab} x) =0$.
	The claim for $d\zeta_k$ follows from a similar calculation.
\end{proof}
As a direct consequence, we obtain the Lie derivatives of the basic forms.
\begin{lemma}\label{lem:lieDerYabzZeta}
	For $1 \leq a,b \leq l$, $k \in \Iind$, 
	\begin{align*}
		\LieDer_{\widetilde{Y_{ab}}} z_k &= \frac{1}{2}\delta_{ka} z_{\overline{b}} + \frac{1}{2}\delta_{kb}z_{\overline{a}},&
		\LieDer_{\widetilde{Y_{ab}}} \zeta_k &= -\frac{1}{2}\delta_{ka} \zeta_{\overline{b}} - \frac{1}{2}\delta_{kb}\zeta_{\overline{a}} + \zeta_{\overline{a}}\zeta_{\overline{b}}\zeta_k,\\
		\LieDer_{\widetilde{Y_{ab}}} dz_k &= \frac{1}{2}\delta_{ka} dz_{\overline{b}} + \frac{1}{2}\delta_{kb}dz_{\overline{a}}, &
		\LieDer_{\widetilde{Y_{ab}}} d\zeta_k &= -\frac{1}{2}\delta_{ka} d\zeta_{\overline{b}} - \frac{1}{2}\delta_{kb}d\zeta_{\overline{a}} + (d\zeta_{\overline{a}}\zeta_{\overline{b}} + \zeta_{\overline{a}}d\zeta_{\overline{b}})\zeta_k + \zeta_{\overline{a}}\zeta_{\overline{b}}d\zeta_k
	\end{align*}
\end{lemma}
\begin{proof}
	Since $\LieDer_{X}=i_X\circ d+d\circ i_X$ for a vector field $X$, Eq.~\eqref{eq:LieAlgactOnSRn} and \autoref{lem:actionYabdzZeta} imply
	\begin{align*}
		\LieDer_{\widetilde{Y_{ab}}} z_k (x,v) = (dz_k)_{(x,v)}(\widetilde{Y_{ab}}) = \frac{1}{2}\delta_{ka} z_{\overline{b}}(x,v) + \frac{1}{2}\delta_{kb}z_{\overline{a}}(x,v), \quad (x,v) \in S\RR^n,
	\end{align*}
	and the calculation for $\LieDer_{\widetilde{Y_{ab}}} \zeta_k$ is similar. The formulas for $\LieDer_{\widetilde{Y_{ab}}} dz_k$ and $\LieDer_{\widetilde{Y_{ab}}} d\zeta_k$ then follow directly, as the exterior derivative and the Lie derivative commute.
\end{proof}

In the following, we will apply \autoref{lem:lieDerYabzZeta} to compute the action of $Y_{ab}$ on the highest weight vectors $\omega_{r,k,m}$ (and thereby $D\omega_{r,k,m}$).

\subsection{The action of $Y_{11}$}
The goal of this section is to relate $\omega_{r,k,m}$ to $\omega_{r,k,m+2}$ by the action of $Y_{11}\in\sln(n)_\CC$. Throughout this section, $1\leq r\leq n-1$ and $1\leq k\leq \min\{r,n-r\}$ are fixed and we use the notation $K = \{1, \dots, k\}$, $J = \Iind \setminus K$, $L = J \setminus\overline{K}$, $K_1 = K \setminus \{1\}$ introduced in \autoref{sec:doubleforms}.

\begin{lemma}\label{lem:LieDerY11simple}
	For $1 \leq k \leq \min\{r, n-r\}$, we have
	\begin{align}
		Y_{11}\bullet  \overline{\zeta_{K_1}} &= \zeta_{\overline{1}}^2 \overline{\zeta_{K_1}},
		&Y_{11}\bullet  \overline{d\zeta_{K_1}} &= \zeta_{\overline{1}}^2 \overline{d\zeta_{K_1}} + 2\zeta_{\overline{1}}d\zeta_{\overline{1}} \overline{\zeta_{K_1}},\\
		Y_{11}\bullet  \zeta_{\overline{K_1}} &= \zeta_{\overline{1}}^2 \zeta_{\overline{K_1}},
		&Y_{11}\bullet  d\zeta_{\overline{K_1}} &= \zeta_{\overline{1}}^2 d\zeta_{\overline{K_1}} + 2\zeta_{\overline{1}}d\zeta_{\overline{1}} \zeta_{\overline{K_1}},\\
		Y_{11}\bullet  \zeta_L &= \zeta_{\overline{1}}^2 \zeta_L,
		&Y_{11}\bullet  d\zeta_L &= \zeta_{\overline{1}}^2 d\zeta_L + 2\zeta_{\overline{1}}d\zeta_{\overline{1}} \zeta_L,\\
		Y_{11}\bullet  \overline{dz_{K_1}} &= 0,\\
		Y_{11}\bullet  dz_{\overline{K_1}} &= 0, \\
		Y_{11}\bullet  dz_L &= 0.
	\end{align}
	Moreover,
	\begin{align}
		Y_{11}\bullet  (\zeta_{\overline{1}} \otimes dz_{\overline{1}}) &= \zeta_{\overline{1}}^2(\zeta_{\overline{1}} \otimes dz_{\overline{1}})\\
		Y_{11}\bullet  (dz_{\overline{1}} \otimes dz_1) &= dz_{\overline{1}} \otimes dz_{\overline{1}}.
	\end{align}
\end{lemma}
\begin{proof}
	The claims are direct consequences of \autoref{lem:lieDerYabzZeta} and Eq.~\eqref{eq:lieDerDoubleForm}, as well as the fact that the Lie derivative and the exterior derivative commute. In particular, we note that $Y_{11} $ acts as multiplication by $\zeta_{\overline{1}}^2$ on all $\zeta$-terms except $\zeta_{1}$, and annihilates all $dz$-terms except $dz_1$, which is replaced by $dz_{\overline{1}}$.
\end{proof}

\begin{lemma}
	\label{lem:LieY11Applied} For $1\leq k\leq \min\{r,n-r\}$, 
	\begin{align}
		Y_{11}\bullet  (\omega_{r,k}\otimes\Theta_1) = (n-r)\zeta_{\overline{1}}^2 \omega_{r,k} \otimes\Theta_1.
	\end{align}
\end{lemma}
\begin{proof}
	As the factors $1\otimes dz_j$ each have to appear exactly once, we can decompose $\omega_{r,k} \otimes \Theta_1$ as follows, extracting the terms where the indices $1$ and $\overline{1}$ appear, see Eq.~\eqref{eq:OmSepL} in \autoref{prop:OmSigTauSepL},
	\begin{align*}
		\omega_{r,k}\otimes \Theta_1 =& \zeta_{\overline{K}}d\zeta_{\overline{K}}^{[k-1]}d\zeta_L^{[n-r-k]}dz_L^{[r-k]}\overline{dz_K}^{[k]}
		+\zeta_L d\zeta_{\overline{K}}^{[k]}d\zeta_L^{[n-r-k-1]}dz_L^{[r-k]}\overline{dz_K}^{[k]}\\
		=&(\zeta_{\overline{1}}\otimes dz_{\overline{1}})(dz_{\overline{1}}\otimes dz_1) d\zeta_{\overline{K_1}}^{[k-1]}d\zeta_L^{[n-r-k]}dz_L^{[r-k]}\overline{dz_{K_1}}^{[k-1]}\\
		+&(d\zeta_{\overline{1}} \otimes dz_{\overline{1}})(dz_{\overline{1}}\otimes dz_1)\zeta_{\overline{K_1}}d\zeta_{\overline{K_1}}^{[k-2]}d\zeta_L^{[n-r-k]}dz_L^{[r-k]}\overline{dz_{K_1}}^{[k-1]}\\
%		\\
		%
		+& (-1)^{k+1}(d\zeta_{\overline{1}} \otimes dz_{\overline{1}})(dz_{\overline{1}}\otimes dz_1)d\zeta_{\overline{K_1}}^{[k-1]}\zeta_Ld\zeta_L^{[n-r-k-1]}dz_L^{[r-k]}\overline{dz_{K_1}}^{[k-1]}.
	\end{align*}
	Call the appearing double forms in the first to third summand $\Omega_1, \Omega_2, \Omega_3$, respectively, so that $\omega_{r,k} \otimes \Theta_1 = \Omega_1 + \Omega_2 + \Omega_3$. This notation will be used only here.
	
	We consider the action individually on the three terms. First,
	\begin{align*}
		Y_{11}\bullet (\zeta_{\overline{1}} \otimes dz_{\overline{1}})(dz_{\overline{1}}\otimes dz_1) &= \zeta_{\overline{1}}^2 (\zeta_{\overline{1}} \otimes dz_{\overline{1}})(dz_{\overline{1}}\otimes dz_1) + (\zeta_{\overline{1}} \otimes dz_{\overline{1}})(dz_{\overline{1}}\otimes dz_{\overline{1}})\\
		&=\zeta_{\overline{1}}^2 (\zeta_{\overline{1}} \otimes dz_{\overline{1}})(dz_{\overline{1}}\otimes dz_1),
	\end{align*}
	by \autoref{lem:LieDerY11simple}, and as the $1\otimes dz_{\overline{1}}$ terms cancel in the product. Consequently, 
	\begin{align*}
		Y_{11}\bullet (d\zeta_{\overline{1}} \otimes dz_{\overline{1}})(dz_{\overline{1}}\otimes dz_1) &= \zeta_{\overline{1}}^2 (d\zeta_{\overline{1}} \otimes dz_{\overline{1}})(dz_{\overline{1}}\otimes dz_1) + 2\zeta_{\overline{1}}d\zeta_{\overline{1}} (\zeta_{\overline{1}} \otimes dz_{\overline{1}})(dz_{\overline{1}}\otimes dz_1)\\
		&= 3\zeta_{\overline{1}}^2 (d\zeta_{\overline{1}} \otimes dz_{\overline{1}})(dz_{\overline{1}}\otimes dz_1),
	\end{align*}
	as the Lie derivative and the exterior derivative commute, where we exchanged $d\zeta_{\overline{1}}$ and $\zeta_{\overline{1}}$ in the second term. Next, again by \autoref{lem:LieDerY11simple} and the product rule,
	\begin{align*}
		Y_{11}\bullet \left(\zeta_{\overline{K_1}}d\zeta_{\overline{K_1}}^{[k-2]}\right) &= \zeta_{\overline{1}}^2 \zeta_{\overline{K_1}}d\zeta_{\overline{K_1}}^{[k-2]} + \zeta_{\overline{K_1}}d\zeta_{\overline{K_1}}^{[k-3]}(\zeta_{\overline{1}}^2 d\zeta_{\overline{K_1}} + 2\zeta_{\overline{1}}d\zeta_{\overline{1}} \zeta_{\overline{K_1}})\\
		&=(1+k-2)\zeta_{\overline{1}}^2 \zeta_{\overline{K_1}}d\zeta_{\overline{K_1}}^{[k-2]},
	\end{align*}
	as $\zeta_{\overline{K_1}}\zeta_{\overline{K_1}} = 0$ in the last term, and %Since $d(\zeta_{K_1}d\zeta_{K_1}^{[k-2]}) = (k-1)d\zeta_{K_1}^{[k-1]}$,
	\begin{align*}
		Y_{11}\bullet d\zeta_{\overline{K_1}}^{[k-1]} = d\zeta_{\overline{K_1}}^{[k-2]}(\zeta_{\overline{1}}^2 d\zeta_{\overline{K_1}} + 2\zeta_{\overline{1}}d\zeta_{\overline{1}} \zeta_{\overline{K_1}}) = (k-1)\zeta_{\overline{1}}^2d\zeta_{\overline{K_1}}^{[k-1]} + 2\zeta_{\overline{1}}d\zeta_{\overline{1}} \zeta_{\overline{K_1}}d\zeta_{\overline{K_1}}^{[k-2]}.
	\end{align*}
	For the $L$-terms, we obtain, as $\zeta_L \zeta_L = 0$,
	\begin{align*}
		Y_{11}\bullet \left(\zeta_Ld\zeta_L^{[n-r-k-1]}\right) &= \zeta_{\overline{1}}^2\zeta_Ld\zeta_L^{[n-r-k-1]} + \zeta_L d\zeta_L^{[n-r-k-2]}(\zeta_{\overline{1}}^2 d\zeta_L + 2 \zeta_{\overline{1}}d\zeta_{\overline{1}}\zeta_L)\\
		&=(1+n-r-k-1)\zeta_{\overline{1}}^2\zeta_Ld\zeta_L^{[n-r-k-1]},
	\end{align*}
	and
	\begin{align*}
		Y_{11}\bullet d\zeta_L^{[n-r-k]} &= d\zeta_L^{[n-r-k-1]}(\zeta_{\overline{1}}^2 d\zeta_L + 2 \zeta_{\overline{1}}d\zeta_{\overline{1}}\zeta_L)\\
		&=(n-r-k)\zeta_{\overline{1}}^2d\zeta_L^{[n-r-k]} + 2\zeta_{\overline{1}}d\zeta_{\overline{1}}\zeta_Ld\zeta_L^{[n-r-k-1]}.
	\end{align*}
	The $dz_L$ and $\overline{dz_{K_1}}$ terms, finally, do not appear, as their Lie derivatives vanish. Combining these results, we obtain the following table, where we split the equations into the part which just gets multiplied by $c\zeta_{\overline{1}}^2$, $c\in\CC$, and the remainder for clarity.
	\begin{center}
		\begin{tabular}{|c|c|c|}
			\hline
			$\tau$ & $c$ & $Y_{11}\bullet \tau - c \zeta_{\overline{1}}^2 \tau$\\
			\hline\hline
			$(\zeta_{\overline{1}} \otimes dz_{\overline{1}})(dz_{\overline{1}}\otimes dz_1)$ & 1 & -- \\
			\hline
			$(d\zeta_{\overline{1}} \otimes dz_{\overline{1}})(dz_{\overline{1}}\otimes dz_1)$ & 3 & -- \\
			\hline
			$\zeta_{\overline{K_1}}d\zeta_{\overline{K_1}}^{[k-2]}$ & $k-1$ & --\\
			\hline
			$d\zeta_{\overline{K_1}}^{[k-1]}$ & $k-1$ & $2\zeta_{\overline{1}}d\zeta_{\overline{1}} \zeta_{\overline{K_1}}d\zeta_{\overline{K_1}}^{[k-2]}$\\
			\hline
			$\zeta_Ld\zeta_L^{[n-r-k-1]}$ & $n-r-k$ & --\\
			\hline
			$d\zeta_L^{[n-r-k]}$ & $n-r-k$ & $2\zeta_{\overline{1}}d\zeta_{\overline{1}}\zeta_Ld\zeta_L^{[n-r-k-1]}$\\
			\hline
		\end{tabular}
	\end{center}
	Using the Leibniz rule~\eqref{eq:actionProductRule}, we thus only need to add up the entries of the central column for the corresponding parts of the forms to obtain the desired multiple of the form and add the remainders. In all cases, this multiple is either $n-r$ or $n-r+2$. For $\Omega_1$, we obtain
	\begin{align*}
		Y_{11}\bullet  \Omega_1 &= Y_{11}\bullet \left((\zeta_{\overline{1}}\otimes dz_{\overline{1}})(dz_{\overline{1}}\otimes dz_1) d\zeta_{\overline{K_1}}^{[k-1]}d\zeta_L^{[n-r-k]}dz_L^{[r-k]}\overline{dz_{K_1}}^{[k-1]}\right)\\
		&=(n-r)\zeta_{\overline{1}}^2\Omega_1 \\
		&+(\zeta_{\overline{1}}\otimes dz_{\overline{1}})(dz_{\overline{1}}\otimes dz_1) 2\zeta_{\overline{1}}d\zeta_{\overline{1}} \zeta_{\overline{K_1}}d\zeta_{\overline{K_1}}^{[k-2]}d\zeta_L^{[n-r-k]}dz_L^{[r-k]}\overline{dz_{K_1}}^{[k-1]}\\
		&+(\zeta_{\overline{1}}\otimes dz_{\overline{1}})(dz_{\overline{1}}\otimes dz_1) d\zeta_{\overline{K_1}}^{[k-1]}2\zeta_{\overline{1}}d\zeta_{\overline{1}}\zeta_Ld\zeta_L^{[n-r-k-1]}dz_L^{[r-k]}\overline{dz_{K_1}}^{[k-1]}\\
		&=\zeta_{\overline{1}}^2\left((n-r)\Omega_1 -2\Omega_2 -2\Omega_3 \right),
	\end{align*}
	where for the last line, we exchanged $d\zeta_{\overline{1}}$ and $\zeta_{\overline{1}}$ in the two remainder terms. For $\Omega_2$, we obtain in a similar way
	\begin{align*}
		Y_{11}\bullet  \Omega_2 &= Y_{11}\bullet \left((d\zeta_{\overline{1}} \otimes dz_{\overline{1}})(dz_{\overline{1}}\otimes dz_1)\zeta_{\overline{K_1}}d\zeta_{\overline{K_1}}^{[k-2]}d\zeta_L^{[n-r-k]}dz_L^{[r-k]}\overline{dz_{K_1}}^{[k-1]} \right)\\
		&=(n-r+2)\zeta_{\overline{1}}^2 \Omega_2\\
		&+(d\zeta_{\overline{1}} \otimes dz_{\overline{1}})(dz_{\overline{1}}\otimes dz_1)\zeta_{\overline{K_1}}d\zeta_{\overline{K_1}}^{[k-2]}2\zeta_{\overline{1}}d\zeta_{\overline{1}}\zeta_Ld\zeta_L^{[n-r-k-1]}dz_L^{[r-k]}\overline{dz_{K_1}}^{[k-1]}\\
		&=(n-r+2)\zeta_{\overline{1}}^2 \Omega_2,
	\end{align*}
	as the second term vanishes since $d\zeta_{\overline{1}}$ appears twice in the first component. The same argument yields
	\begin{align*}
		Y_{11}\bullet  \Omega_3 &= Y_{11}\bullet \left((-1)^{k+1}(d\zeta_{\overline{1}} \otimes dz_{\overline{1}})(dz_{\overline{1}}\otimes dz_1)d\zeta_{\overline{K_1}}^{[k-1]}\zeta_Ld\zeta_L^{[n-r-k-1]}dz_L^{[r-k]}\overline{dz_{K_1}}^{[k-1]}\right)\\
		&=(n-r+2)\zeta_{\overline{1}}^2 \Omega_3\\
		&+(-1)^{k+1}(d\zeta_{\overline{1}} \otimes dz_{\overline{1}})(dz_{\overline{1}}\otimes dz_1)2\zeta_{\overline{1}}d\zeta_{\overline{1}} \zeta_{\overline{K_1}}d\zeta_{\overline{K_1}}^{[k-2]}\zeta_Ld\zeta_L^{[n-r-k-1]}dz_L^{[r-k]}\overline{dz_{K_1}}^{[k-1]}\\
		&=(n-r+2)\zeta_{\overline{1}}^2 \Omega_3.
	\end{align*}
	Combining these equations, we obtain the desired result:
	\begin{align*}
		Y_{11}\bullet  (\omega_{r,k} \otimes \Theta_1) = Y_{11}\bullet (\Omega_1 + \Omega_2 + \Omega_3) = (n-r)\zeta_{\overline{1}}^2 \omega_{r,k}\otimes \Theta_1.
	\end{align*}
\end{proof}

\begin{corollary}\label{cor:LieY11onOmegarkm}
	For $1\leq k\le\min\{r,n-r\}$, $m\in \NN$, $m\ge 2$,
	\begin{align}
		\LieDer_{\widetilde{Y}_{11}} \omega_{r,k,m} = (n-r+m-2)\omega_{r,k,m+2}.
	\end{align}
\end{corollary}
\begin{proof}
	A short application of the product rule and \autoref{lem:lieDerYabzZeta} yield
	\begin{align*}
		Y_{11}\bullet  (\omega_{r,k,m} \otimes \Theta_1 )&= Y_{11}\bullet  \left(\zeta_{\overline{1}}^{m-2}\omega_{r,k} \otimes \Theta_1\right)\\
		&= (m-2)\zeta_{\overline{1}}^{m-3}(\zeta_{\overline{1}}^2 \zeta_{\overline{1}}) \omega_{r,k} \otimes \Theta_1 + (n-r)\zeta_{\overline{1}}^{m-2}\zeta_{\overline{1}}^2\omega_{r,k} \otimes \Theta_1\\
		&=(n-r+m-2)\zeta_{\overline{1}}^{m}\omega_{r,k} \otimes \Theta_1 = (n-r+m-2)\omega_{r,k,m+2} \otimes \Theta_1,
	\end{align*}
	which, by \eqref{eq:LieDerYabTensorTheta}, shows the claim.
\end{proof}

\begin{remark}
	\label{remark:actionY11omegak}
	As in \autoref{remark:holomorphicDomegakm}, we may define the relevant forms on the open set of $S\RR^n$ given by $\zeta_1\notin(-\infty,0]$ for arbitrary $m\in\CC$. Then \autoref{cor:LieY11onOmegarkm} holds for arbitrary $m\in\CC$.
\end{remark}

\subsection{The action of $Y_{k+1,k+1}$, $1\leq k<\min\{r,n-r\}$}
\label{sec:Ykkaction}

\subsubsection{Lie derivative}
Recall that  $L^- = L \setminus \{k+1, \overline{k+1}\}$.
\begin{lemma}\label{lem:actYkp1_simpleDf}
	\begin{align}
		Y_{k+1,k+1}\bullet  \zeta_{\overline{K}} &= \zeta_{\overline{k+1}}^2 \zeta_{\overline{K}},
		& Y_{k+1,k+1}\bullet  d\zeta_{\overline{K}} &=\zeta_{\overline{k+1}}^2 d\zeta_{\overline{K}} + 2\zeta_{\overline{k+1}}d\zeta_{\overline{k+1}} \zeta_{\overline{K}},\\
		Y_{k+1,k+1}\bullet  \zeta_{L^-} &= \zeta_{\overline{k+1}}^2 \zeta_{L^-},
		& Y_{k+1,k+1}\bullet  d\zeta_{L^-} &=\zeta_{\overline{k+1}}^2 d\zeta_{L^-} + 2\zeta_{\overline{k+1}}d\zeta_{\overline{k+1}} \zeta_{L^-},\\
		Y_{k+1,k+1}\bullet  \overline{dz_K} &= 0,\\
		Y_{k+1,k+1}\bullet  dz_{L^-} &= 0.
	\end{align}
\end{lemma}
\begin{proof}
	The claims are direct consequences of \autoref{lem:lieDerYabzZeta} and Eq.~\eqref{eq:lieDerDoubleForm}, as well as the fact that Lie derivative and exterior derivative commute, noting that $k+1$ and $\overline{k+1}$ do not appear in the index sets. 
\end{proof}

Recall that
\begin{align*}
	\omega_{r,k} \otimes \Theta_1 = \zeta_{\overline{K}}d\zeta_{\overline{K}}^{[k-1]}d\zeta_L^{[n-r-k]}dz_L^{[r-k]}\overline{dz_K}^{[k]}
	+\zeta_L d\zeta_{\overline{K}}^{[k]}d\zeta_L^{[n-r-k-1]}dz_L^{[r-k]}\overline{dz_K}^{[k]}.
\end{align*}
We will need to further decompose this form in the proof of the next lemma.
\begin{lemma}\label{lem:actYkp1OmegaRk}
	For $1\leq k< \min\{r,n-r\}$,
	\begin{align}\label{eq:actYkp1OnOmegaRk} 
		Y_{k+1,k+1}\bullet  \left(\omega_{r,k} \otimes \Theta_1\right) = (n-r)\zeta_{\overline{k+1}}^2\omega_{r,k} \otimes \Theta_1 + 2 \omega_{r,k+1} \otimes \Theta_1
	\end{align}
\end{lemma}
\begin{proof}
	The proof is a simple but tedious calculation, which is mostly a challenge in book keeping. For a systematic approach, we split each term of $\omega_{r,k} \otimes \Theta_1$ according to the three blocks, $K \cup \overline{K}$, $\{k+1, \overline{k+1}\}$ and $L^-$ of the index set $\Iind$. This yields twelve terms, which we label $\Omega_1, \dots, \Omega_{12}$:
	\begin{align*}
		\omega_{r,k} \otimes \Theta_1
		&= \zeta_{\overline{K}}d\zeta_{\overline{K}}^{[k-1]}\overline{dz_K}^{[k]} (d\zeta_{k+1} \otimes dz_{k+1})(d\zeta_{\overline{k+1}} \otimes dz_{\overline{k+1}}) d\zeta_{L^-}^{[n-r-k-2]} dz_{L^-}^{[r-k]} \displaybreak[1]\\
		&+\zeta_{\overline{K}}d\zeta_{\overline{K}}^{[k-1]}\overline{dz_K}^{[k]} (d\zeta_{k+1} \otimes dz_{k+1})(dz_{\overline{k+1}} \otimes dz_{\overline{k+1}}) d\zeta_{L^-}^{[n-r-k-1]} dz_{L^-}^{[r-k-1]} \displaybreak[1]\\
		&+\zeta_{\overline{K}}d\zeta_{\overline{K}}^{[k-1]}\overline{dz_K}^{[k]} (dz_{k+1} \otimes dz_{k+1})(d\zeta_{\overline{k+1}} \otimes dz_{\overline{k+1}}) d\zeta_{L^-}^{[n-r-k-1]} dz_{L^-}^{[r-k-1]} \displaybreak[1]\\
		&+\zeta_{\overline{K}}d\zeta_{\overline{K}}^{[k-1]}\overline{dz_K}^{[k]} (dz_{k+1} \otimes dz_{k+1})(dz_{\overline{k+1}} \otimes dz_{\overline{k+1}}) d\zeta_{L^-}^{[n-r-k]} dz_{L^-}^{[r-k-2]} \displaybreak[1]\\
		\displaybreak[1]\\
		&+ d\zeta_{\overline{K}}^{[k]}\overline{dz_K}^{[k]}(\zeta_{k+1} \otimes dz_{k+1})(d\zeta_{\overline{k+1}} \otimes dz_{\overline{k+1}})d\zeta_{L^-}^{[n-r-k-2]}dz_{L^-}^{[r-k]}\displaybreak[1]\\
		&+ d\zeta_{\overline{K}}^{[k]}\overline{dz_K}^{[k]}(\zeta_{k+1} \otimes dz_{k+1})(dz_{\overline{k+1}} \otimes dz_{\overline{k+1}})d\zeta_{L^-}^{[n-r-k-1]}dz_{L^-}^{[r-k-1]}\displaybreak[1]\\
		&+ d\zeta_{\overline{K}}^{[k]}\overline{dz_K}^{[k]}(\zeta_{\overline{k+1}} \otimes dz_{\overline{k+1}})(d\zeta_{k+1} \otimes dz_{k+1})d\zeta_{L^-}^{[n-r-k-2]}dz_{L^-}^{[r-k]}\displaybreak[1]\\
		&+ d\zeta_{\overline{K}}^{[k]}\overline{dz_K}^{[k]}(\zeta_{\overline{k+1}} \otimes dz_{\overline{k+1}})(dz_{k+1} \otimes dz_{k+1})d\zeta_{L^-}^{[n-r-k-1]}dz_{L^-}^{[r-k-1]}\displaybreak[1]\\
		\displaybreak[1]\\
		&+ d\zeta_{\overline{K}}^{[k]}\overline{dz_K}^{[k]}(d\zeta_{k+1} \otimes dz_{k+1})(d\zeta_{\overline{k+1}} \otimes dz_{\overline{k+1}})\zeta_{L^-}d\zeta_{L^-}^{[n-r-k-3]}dz_{L^-}^{[r-k]}\displaybreak[1]\\
		&+ d\zeta_{\overline{K}}^{[k]}\overline{dz_K}^{[k]}(d\zeta_{k+1} \otimes dz_{k+1})(dz_{\overline{k+1}} \otimes dz_{\overline{k+1}})\zeta_{L^-}d\zeta_{L^-}^{[n-r-k-2]}dz_{L^-}^{[r-k-1]}\displaybreak[1]\\
		&+ d\zeta_{\overline{K}}^{[k]}\overline{dz_K}^{[k]}(dz_{k+1} \otimes dz_{k+1})(d\zeta_{\overline{k+1}} \otimes dz_{\overline{k+1}})\zeta_{L^-}d\zeta_{L^-}^{[n-r-k-2]}dz_{L^-}^{[r-k-1]}\displaybreak[1]\\
		&+ d\zeta_{\overline{K}}^{[k]}\overline{dz_K}^{[k]}(dz_{k+1} \otimes dz_{k+1})(dz_{\overline{k+1}} \otimes dz_{\overline{k+1}})\zeta_{L^-}d\zeta_{L^-}^{[n-r-k-1]}dz_{L^-}^{[r-k-2]}\displaybreak[1]\\
		&=: \Omega_1 + \dots + \Omega_{12}.
	\end{align*}
	We first calculate the action on the building blocks of $\Omega_1, \dots, \Omega_{12}$, using \autoref{lem:actYkp1_simpleDf} and \autoref{lem:lieDerYabzZeta} and apply the product rule~\eqref{eq:actionProductRule} afterward. In total, there are fourteen building blocks, and we will provide the calculation for the first two as an example and omit the remaining ones, since they are very similar. For $\zeta_{\overline{K}}d\zeta_{\overline{K}}^{[k-1]}$, the product rule and \autoref{lem:actYkp1_simpleDf} yield
	\begin{align*}
		Y_{k+1,k+1}\bullet  \zeta_{\overline{K}}d\zeta_{\overline{K}}^{[k-1]} &= \zeta_{\overline{k+1}}^2\zeta_{\overline{K}}d\zeta_{\overline{K}}^{[k-1]} + \zeta_{\overline{K}}d\zeta_{\overline{K}}^{[k-2]}(\zeta_{\overline{k+1}}^2 d\zeta_{\overline{K}} + 2\zeta_{\overline{k+1}}d\zeta_{\overline{k+1}} \zeta_{\overline{K}})\\
		&=(1+k-1)\zeta_{\overline{k+1}}^2\zeta_{\overline{K}}d\zeta_{\overline{K}}^{[k-1]},
	\end{align*}
	where we used in the last step that $\zeta_{\overline{K}}\zeta_{\overline{K}} = 0$. For $d\zeta_{\overline{K}}^{[k]}$, we obtain in a similar way
	\begin{align*}
		Y_{k+1,k+1}\bullet  d\zeta_{\overline{K}}^{[k]} &= d\zeta_{\overline{K}}^{[k-1]}(\zeta_{\overline{k+1}}^2 d\zeta_{\overline{K}} + 2\zeta_{\overline{k+1}}d\zeta_{\overline{k+1}} \zeta_{\overline{K}})\\
		&=k \zeta_{\overline{k+1}}^2 d\zeta_{\overline{K}}^{[k]} + 2\zeta_{\overline{k+1}}d\zeta_{\overline{k+1}} \zeta_{\overline{K}}d\zeta_{\overline{K}}^{[k-1]}.
	\end{align*}
	In both cases, we obtain the original form multiplied with $\zeta_{\overline{k+1}}^2$ and a constant plus extra terms. Summing up these multiples will eventually give the first term in \eqref{eq:actYkp1OnOmegaRk}. \autoref{tab:lieDerYkp1BuildBlocks} shows the action of $Y_{k+1,k+1}$ on the building blocks, decomposed into a product of the form with $c\zeta_{\overline{k+1}}^2$, $c\in\CC$, and the remainder. The terms
	\begin{align*}
		Y_{k+1,k+1}\bullet  \overline{dz_K}^{[k]} = 0 \quad \text{ and } \quad Y_{k+1,k+1}\bullet  dz_{L^-}^{[j]} = 0,
	\end{align*}
	are omitted from the table.
	
\begin{extracalc}
	\begin{align*}
		Y_{k+1,k+1}\bullet  \zeta_{L^-}d\zeta_{L^-}^{[j-1]} &= \zeta_{\overline{k+1}}^2\zeta_{L^-}d\zeta_{L^-}^{[j-1]} + \zeta_{L^-}d\zeta_{L^-}^{[j-2]}(\zeta_{\overline{k+1}}^2d\zeta_{L^-} + 2\zeta_{\overline{k+1}}d\zeta_{\overline{k+1}}\zeta_{L^-})\\
		&=(1+j-1)\zeta_{\overline{k+1}}^2\zeta_{L^-}d\zeta_{L^-}^{[j-1]}
	\end{align*}
	\begin{align*}
		Y_{k+1,k+1}\bullet  d\zeta_{L^-}^{[j]} &= d\zeta_{L^-}^{[j-1]}(\zeta_{\overline{k+1}}^2d\zeta_{L^-} + 2\zeta_{\overline{k+1}}d\zeta_{\overline{k+1}}\zeta_{L^-})\\
		&=j\zeta_{\overline{k+1}}^2d\zeta_{L^-}^{[j]} + 2\zeta_{\overline{k+1}}d\zeta_{\overline{k+1}}\zeta_{L^-}d\zeta_{L^-}^{[j-1]}
	\end{align*}
	
	\begin{align*}
		Y_{k+1,k+1}\bullet  \overline{dz_K}^{[k]} &= 0\\
		Y_{k+1,k+1}\bullet  dz_{L^-}^{[j]} &= 0
	\end{align*}
	
	\begin{align*}
		&Y_{k+1,k+1}\bullet  (d\zeta_{k+1} \otimes dz_{k+1})(d\zeta_{\overline{k+1}} \otimes dz_{\overline{k+1}})\\
		 &=(-d\zeta_{\overline{k+1}}+2\zeta_{\overline{k+1}}d\zeta_{\overline{k+1}}\zeta_{k+1} + \zeta_{\overline{k+1}}^2d\zeta_{k+1})\otimes dz_{k+1}(d\zeta_{\overline{k+1}} \otimes dz_{\overline{k+1}})\\
		 &+(d\zeta_{k+1} \otimes dz_{k+1})(\zeta_{\overline{k+1}}^2d\zeta_{\overline{k+1}} + 2 \zeta_{\overline{k+1}}d\zeta_{\overline{k+1}}\zeta_{\overline{k+1}}) \otimes dz_{\overline{k+1}}\\
		 &=(1+3)\zeta_{\overline{k+1}}^2(d\zeta_{k+1}\otimes dz_{k+1})(d\zeta_{\overline{k+1}} \otimes dz_{\overline{k+1}})
	\end{align*}
	\begin{align*}
		&Y_{k+1,k+1}\bullet  (d\zeta_{k+1} \otimes dz_{k+1})(dz_{\overline{k+1}} \otimes dz_{\overline{k+1}})\\
		&=(-d\zeta_{\overline{k+1}}+2\zeta_{\overline{k+1}}d\zeta_{\overline{k+1}}\zeta_{k+1} + \zeta_{\overline{k+1}}^2d\zeta_{k+1})\otimes dz_{k+1}(dz_{\overline{k+1}} \otimes dz_{\overline{k+1}})\\
		&= \zeta_{\overline{k+1}}^2(d\zeta_{k+1} \otimes dz_{k+1})(dz_{\overline{k+1}} \otimes dz_{\overline{k+1}}) + (2\nu_{k+1}-1)(d\zeta_{\overline{k+1}}\otimes dz_{k+1})(dz_{\overline{k+1}} \otimes dz_{\overline{k+1}})
	\end{align*}
	\begin{align*}
		&Y_{k+1,k+1}\bullet  (dz_{k+1} \otimes dz_{k+1})(d\zeta_{\overline{k+1}} \otimes dz_{\overline{k+1}})\\
		&=(dz_{\overline{k+1}} \otimes dz_{k+1})(d\zeta_{\overline{k+1}} \otimes dz_{\overline{k+1}})\\
		&+(dz_{k+1} \otimes dz_{k+1})(\zeta_{\overline{k+1}}^2d\zeta_{\overline{k+1}} + 2 \zeta_{\overline{k+1}}d\zeta_{\overline{k+1}}\zeta_{\overline{k+1}}) \otimes dz_{\overline{k+1}}\\
		&=3\zeta_{\overline{k+1}}^2(dz_{k+1} \otimes dz_{k+1})(d\zeta_{\overline{k+1}} \otimes dz_{\overline{k+1}}) + (dz_{\overline{k+1}} \otimes dz_{k+1})(d\zeta_{\overline{k+1}} \otimes dz_{\overline{k+1}})
	\end{align*}
	\begin{align*}
		&Y_{k+1,k+1}\bullet  (dz_{k+1} \otimes dz_{k+1})(dz_{\overline{k+1}} \otimes dz_{\overline{k+1}})\\
		&=(dz_{\overline{k+1}} \otimes dz_{k+1})(dz_{\overline{k+1}} \otimes dz_{\overline{k+1}}) = 0
	\end{align*}
	
	\begin{align*}
		&Y_{k+1,k+1}\bullet (\zeta_{k+1} \otimes dz_{k+1})(d\zeta_{\overline{k+1}} \otimes dz_{\overline{k+1}})\\
		&=(\zeta_{\overline{k+1}}^2\zeta_{k+1}-\zeta_{\overline{k+1}}) \otimes dz_{k+1}(d\zeta_{\overline{k+1}} \otimes dz_{\overline{k+1}})\\
		&+(\zeta_{k+1} \otimes dz_{k+1})(3\zeta_{\overline{k+1}}^2d\zeta_{\overline{k+1}} \otimes dz_{\overline{k+1}})\\
		&=4\zeta_{\overline{k+1}}^2(\zeta_{k+1} \otimes dz_{k+1})(d\zeta_{\overline{k+1}} \otimes dz_{\overline{k+1}}) - (\zeta_{\overline{k+1}} \otimes dz_{k+1})(d\zeta_{\overline{k+1}} \otimes dz_{\overline{k+1}})
	\end{align*}
	\begin{align*}
		&Y_{k+1,k+1}\bullet (\zeta_{k+1} \otimes dz_{k+1})(dz_{\overline{k+1}} \otimes dz_{\overline{k+1}})=\\
		&=(\zeta_{\overline{k+1}}^2\zeta_{k+1} - \zeta_{\overline{k+1}}) \otimes dz_{k+1}(dz_{\overline{k+1}} \otimes dz_{\overline{k+1}})
	\end{align*}
	
	\begin{align*}
		&Y_{k+1,k+1}\bullet (\zeta_{\overline{k+1}} \otimes dz_{\overline{k+1}})(d\zeta_{k+1} \otimes dz_{k+1})\\
		&=\zeta_{\overline{k+1}}^2(\zeta_{\overline{k+1}} \otimes dz_{\overline{k+1}})(d\zeta_{k+1} \otimes dz_{k+1})\\
		&+(\zeta_{\overline{k+1}} \otimes dz_{\overline{k+1}})(\zeta_{\overline{k+1}}^2d\zeta_{k+1} -d\zeta_{\overline{k+1}} + 2\nu_{k+1}d\zeta_{\overline{k+1}}) \otimes dz_{k+1}\\
		&=2\zeta_{\overline{k+1}}^2(\zeta_{\overline{k+1}} \otimes dz_{\overline{k+1}})(d\zeta_{k+1} \otimes dz_{k+1})\\
		&+(2\nu_{k+1}-1)(\zeta_{\overline{k+1}} \otimes dz_{\overline{k+1}})(d\zeta_{\overline{k+1}} \otimes dz_{k+1})
	\end{align*}
	
	\begin{align*}
		&Y_{k+1,k+1}\bullet (\zeta_{\overline{k+1}} \otimes dz_{\overline{k+1}})(dz_{k+1} \otimes dz_{k+1})\\
		&=\zeta_{\overline{k+1}}^2(\zeta_{\overline{k+1}} \otimes dz_{\overline{k+1}})(dz_{k+1} \otimes dz_{k+1})+(\zeta_{\overline{k+1}} \otimes dz_{\overline{k+1}})(dz_{\overline{k+1}} \otimes dz_{k+1})
	\end{align*}
\end{extracalc}

\begin{table}[h]
	\begin{center}
	\begin{tabular}{|c|c|c|}
		\hline
		$\tau$ & $c$ & $Y_{k+1,k+1}\bullet \tau -c\zeta_{\overline{k+1}}^2 \tau$\\% & yields terms\\
		\hline\hline
		$\zeta_{\overline{K}}d\zeta_{\overline{K}}^{[k-1]}$  & $k$ & -- \\%& --\\
		\hline
		$d\zeta_{\overline{K}}^{[k]}$ & $k$ & $2\zeta_{\overline{k+1}}d\zeta_{\overline{k+1}} \zeta_{\overline{K}}d\zeta_{\overline{K}}^{[k-1]}$ \\%& n.a.\\
		\hline\hline
		$\zeta_{L^-}d\zeta_{L^-}^{[j-1]}$ & $j$ & -- \\%& --\\
		\hline
		$d\zeta_{L^-}^{[j]}$ & $j$ & $2\zeta_{\overline{k+1}}d\zeta_{\overline{k+1}}\zeta_{L^-}d\zeta_{L^-}^{[j-1]}$ \\%& $\Xi_{2,10}$, $\Xi_{4,12}$\\
		\hline\hline
		$(d\zeta_{k+1} \otimes dz_{k+1})(d\zeta_{\overline{k+1}} \otimes dz_{\overline{k+1}})$ & $4$ & -- \\%& --\\
		\hline
		$(d\zeta_{k+1} \otimes dz_{k+1})(dz_{\overline{k+1}} \otimes dz_{\overline{k+1}})$ & $1$ & $(1-2\nu_{k+1})(dz_{\overline{k+1}}\otimes dz_{k+1})(d\zeta_{\overline{k+1}} \otimes dz_{\overline{k+1}})$ \\%& $\Omega_{k+1, 2}$\\
		\hline
		$(dz_{k+1} \otimes dz_{k+1})(d\zeta_{\overline{k+1}} \otimes dz_{\overline{k+1}})$ & $3$ & $(dz_{\overline{k+1}} \otimes dz_{k+1})(d\zeta_{\overline{k+1}} \otimes dz_{\overline{k+1}})$ \\%& $\Omega_{k+1, 2}$\\
		\hline
		$(dz_{k+1} \otimes dz_{k+1})(dz_{\overline{k+1}} \otimes dz_{\overline{k+1}})$ & -- & -- \\%& --\\
		\hline\hline
		$(\zeta_{k+1} \otimes dz_{k+1})(d\zeta_{\overline{k+1}} \otimes dz_{\overline{k+1}})$ & $4$ & $- (\zeta_{\overline{k+1}} \otimes dz_{k+1})(d\zeta_{\overline{k+1}} \otimes dz_{\overline{k+1}})$ \\%& n.a.\\
		\hline
		$(\zeta_{k+1} \otimes dz_{k+1})(dz_{\overline{k+1}} \otimes dz_{\overline{k+1}})$ & $1$ & $-(\zeta_{\overline{k+1}} \otimes dz_{k+1})(dz_{\overline{k+1}} \otimes dz_{\overline{k+1}})$ \\%& $\Omega_{k+1, 1}$\\
		\hline
		$(\zeta_{\overline{k+1}} \otimes dz_{\overline{k+1}})(d\zeta_{k+1} \otimes dz_{k+1})$ & $2$ & $(1-2\nu_{k+1})(\zeta_{\overline{k+1}} \otimes dz_{k+1})(d\zeta_{\overline{k+1}} \otimes dz_{\overline{k+1}})$ \\%& n.a.\\
		\hline
		$(\zeta_{\overline{k+1}} \otimes dz_{\overline{k+1}})(dz_{k+1} \otimes dz_{k+1})$ &$1$& $-(\zeta_{\overline{k+1}} \otimes dz_{k+1})(dz_{\overline{k+1}} \otimes dz_{\overline{k+1}})$ \\%& $\Omega_{k+1, 1}$\\
		\hline 
	\end{tabular}
\end{center}
\caption{Lie derivatives of the building blocks for \autoref{lem:actYkp1OmegaRk}}
\label{tab:lieDerYkp1BuildBlocks}
\end{table}

Next, we also decompose $\omega_{r,k+1} \otimes \Theta_1$ according to the three blocks $K\cup\overline{K}$, $\{k+1, \overline{k+1}\}$ and $L^-$, resulting in three terms $\Omega_{k+1,1}, \Omega_{k+1,2}, \Omega_{k+1,3}$, given by
\begin{align}\label{eq:prfactYkp1OmegaRk_defOmegaRkp1}
	\notag
	\omega_{r,k+1} \otimes \Theta_1  &= 
	d\zeta_{\overline{K}}^{[k]}\overline{dz_{K}}^{[k]}(\zeta_{\overline{k+1}}\otimes dz_{\overline{k+1}}) (dz_{\overline{k+1}}\otimes dz_{k+1}) d\zeta_{L^-}^{[n-r-k-1]} dz_{L^-}^{[r-k-1]} \\
	\notag%
	&+ \zeta_{\overline{K}}d\zeta_{\overline{K}}^{[k-1]}\overline{dz_{K}}^{[k]} (d\zeta_{\overline{k+1}}\otimes dz_{\overline{k+1}})(dz_{\overline{k+1}}\otimes dz_{k+1}) d\zeta_{L^-}^{[n-r-k-1]} dz_{L^-}^{[r-k-1]} \\
	\notag%
	&+ d\zeta_{\overline{K}}^{[k]}\overline{dz_{K}}^{[k]}(d\zeta_{\overline{k+1}}\otimes dz_{\overline{k+1}})(dz_{\overline{k+1}}\otimes dz_{k+1})\zeta_{L^-}d\zeta_{L^-}^{[n-r-k-2]}dz_{L^-}^{[r-k-1]} \\
	&=: \Omega_{k+1, 1} + \Omega_{k+1, 2} + \Omega_{k+1, 3}.
\end{align}
Using \autoref{tab:lieDerYkp1BuildBlocks} and comparing the terms with Eq.~\eqref{eq:prfactYkp1OmegaRk_defOmegaRkp1}, we obtain the following relations:
\begin{align*}
	Y_{k+1,k+1}\bullet \Omega_1 &= \zeta_{\overline{k+1}}^2(n-r+2)\Omega_1,&&\displaybreak[1]\\
	Y_{k+1,k+1}\bullet \Omega_2 &= \zeta_{\overline{k+1}}^2(n-r)\Omega_2 &+& (1-2\nu_{k+1})\Omega_{k+1, 2} +2\zeta_{\overline{k+1}}d\zeta_{k+1} \Xi_{2,10},\displaybreak[1]\\
	Y_{k+1,k+1}\bullet \Omega_3 &= \zeta_{\overline{k+1}}^2(n-r+2)\Omega_3 &+&  \Omega_{k+1, 2},\displaybreak[1]\\
	Y_{k+1,k+1}\bullet \Omega_4 &= \zeta_{\overline{k+1}}^2(n-r)\Omega_4 &-& 2 \zeta_{\overline{k+1}}d\zeta_{\overline{k+1}} \Xi_{4,12},\displaybreak[1]\\
	Y_{k+1,k+1}\bullet \Omega_5 &= \zeta_{\overline{k+1}}^2(n-r+2)\Omega_5 &-&  \Xi_{5,7},\displaybreak[1]\\
	Y_{k+1,k+1}\bullet \Omega_6 &= \zeta_{\overline{k+1}}^2(n-r)\Omega_6 &+& 2\nu_{k+1}\Omega_{k+1, 2} + \Omega_{k+1, 1} + 2\nu_{k+1}\Omega_{k+1, 3},\displaybreak[1]\\
	Y_{k+1,k+1}\bullet \Omega_7 &= \zeta_{\overline{k+1}}^2(n-r)\Omega_7 &-& 2\zeta_{\overline{k+1}}^2\Omega_1 + (1-2\nu_{k+1})\Xi_{5,7} - 2\zeta_{\overline{k+1}}^2\Omega_9,\displaybreak[1]\\
	Y_{k+1,k+1}\bullet \Omega_8 &= \zeta_{\overline{k+1}}^2(n-r)\Omega_8 &-& 2 \zeta_{\overline{k+1}}^2\Omega_3 + \Omega_{k+1, 1} - 2 \zeta_{\overline{k+1}}^2\Omega_{11},\displaybreak[1]\\
	Y_{k+1,k+1}\bullet \Omega_9 &= \zeta_{\overline{k+1}}^2(n-r+2)\Omega_9,&&\displaybreak[1]\\
	Y_{k+1,k+1}\bullet \Omega_{10} &= \zeta_{\overline{k+1}}^2(n-r)\Omega_{10} &-& 2\zeta_{\overline{k+1}}d\zeta_{k+1}\Xi_{2,10} + (1-2\nu_{k+1})\Omega_{k+1, 3},\displaybreak[1]\\
	Y_{k+1,k+1}\bullet \Omega_{11} &= \zeta_{\overline{k+1}}^2(n-r+2)\Omega_{11} &+& \Omega_{k+1, 3},\displaybreak[1]\\
	Y_{k+1,k+1}\bullet \Omega_{12} &= \zeta_{\overline{k+1}}^2(n-r)\Omega_{12} &+& 2\zeta_{\overline{k+1}}d\zeta_{\overline{k+1}} \Xi_{4,12},
\end{align*}
where the auxiliary forms $\Xi_{2,10}, \Xi_{4,12}$ and $\Xi_{5,7}$ are defined as
\begin{align*}
	\Xi_{2,10} &= \zeta_{\overline{K}} d\zeta_{\overline{K}}^{[k-1]}dz_{\overline{K}}^{[k]} (d\zeta_{\overline{k+1}} \otimes dz_{k+1})(dz_{\overline{k+1}} \otimes dz_{\overline{k+1}}) \zeta_{L^-}d\zeta_{L^-}^{[n-r-k-2]}dz_{L^-}^{[r-k-1]},\\
	\Xi_{4,12} &=\zeta_{\overline{K}} d\zeta_{\overline{K}}^{[k-1]}dz_{\overline{K}}^{[k]} (dz_{k+1} \otimes dz_{k+1})(dz_{\overline{k+1}} \otimes dz_{\overline{k+1}}) \zeta_{L^-}d\zeta_{L^-}^{[n-r-k-1]}dz_{L^-}^{[r-k-2]},\\
	\Xi_{5,7} &=d\zeta_{\overline{K}}^{[k]}dz_{\overline{K}}^{[k]} (\zeta_{\overline{k+1}} \otimes dz_{k+1})(d\zeta_{\overline{k+1}} \otimes dz_{\overline{k+1}}) d\zeta_{L^-}^{[n-r-k-2]}dz_{L^-}^{[r-k]}.
\end{align*}
As the calculation is always very similar, we will restrict ourselves to present only the calculation for $Y_{k+1,k+1}\bullet  \Omega_2$. In this case, \autoref{tab:lieDerYkp1BuildBlocks} shows that
\begin{align*}
	&Y_{k+1,k+1}\bullet  \Omega_2  =(k+1+n-r-k-1)\zeta_{\overline{k+1}}^2\Omega_2 \\
	&+\zeta_{\overline{K}}d\zeta_{\overline{K}}^{[k-1]}\overline{dz_K}^{[k]} (1-2\nu_{k+1})(dz_{\overline{k+1}}\otimes dz_{k+1})(d\zeta_{\overline{k+1}} \otimes dz_{\overline{k+1}}) d\zeta_{L^-}^{[n-r-k-1]} dz_{L^-}^{[r-k-1]}\\
	&+\zeta_{\overline{K}}d\zeta_{\overline{K}}^{[k-1]}\overline{dz_K}^{[k]} (d\zeta_{k+1} \otimes dz_{k+1})(dz_{\overline{k+1}} \otimes dz_{\overline{k+1}}) 2\zeta_{\overline{k+1}}d\zeta_{\overline{k+1}}\zeta_{L^-}d\zeta_{L^-}^{[n-r-k-2]} dz_{L^-}^{[r-k-1]}\\
	&=(n-r)\zeta_{\overline{k+1}}^2\Omega_2 + (1-2\nu_{k+1})\Omega_{k+1, 2} + 2\zeta_{\overline{k+1}}d\zeta_{k+1} \Xi_{2,10}.
\end{align*}

\begin{extracalc}
\begin{align*}
	&Y_{k+1,k+1}\bullet  \Omega_1 = \zeta_{\overline{k+1}}^2(k+4+n-r-k-2)\Omega_1
\end{align*}
as the remainder term for $d\zeta_{L^-}$ cancels with the $d\zeta_{\overline{k+1}}$.
\begin{align*}
	&Y_{k+1,k+1}\bullet  \Omega_3 = \zeta_{\overline{k+1}}^2(k+3+n-r-k-1)\Omega_3 \\
	&+\zeta_{\overline{K}}d\zeta_{\overline{K}}^{[k-1]}\overline{dz_K}^{[k]} (dz_{\overline{k+1}} \otimes dz_{k+1})(d\zeta_{\overline{k+1}} \otimes dz_{\overline{k+1}}) d\zeta_{L^-}^{[n-r-k-1]} dz_{L^-}^{[r-k-1]}\\
	&=\zeta_{\overline{k+1}}^2(n-r+2)\Omega_3 + \Omega_{k+1, 2}
\end{align*}
\begin{align*}
	&Y_{k+1,k+1}\bullet  \Omega_4 = \zeta_{\overline{k+1}}^2(k+n-r-k)\Omega_4 \\
	&+\zeta_{\overline{K}}d\zeta_{\overline{K}}^{[k-1]}\overline{dz_K}^{[k]} (dz_{k+1} \otimes dz_{k+1})(dz_{\overline{k+1}} \otimes dz_{\overline{k+1}}) 2\zeta_{\overline{k+1}}d\zeta_{\overline{k+1}}\zeta_{L^-}d\zeta_{L^-}^{[n-r-k-1]} dz_{L^-}^{[r-k-2]}\\
	&=\zeta_{\overline{k+1}}^2(n-r)\Omega_4 - 2\zeta_{\overline{k+1}}d\zeta_{\overline{k+1}}\Xi_{4,12}
\end{align*}
\begin{align*}
	&Y_{k+1,k+1}\bullet  \Omega_5 = 
	\zeta_{\overline{k+1}}^2(k+4+n-r-k-2)\Omega_5 \\
	&-d\zeta_{\overline{K}}^{[k]}\overline{dz_K}^{[k]} (\zeta_{\overline{k+1}} \otimes dz_{k+1})(d\zeta_{\overline{k+1}} \otimes dz_{\overline{k+1}})d\zeta_{L^-}^{[n-r-k-2]}dz_{L^-}^{[r-k]}\\
	&=\zeta_{\overline{k+1}}^2(n-r+2)\Omega_5 - \Xi_{5,7}
\end{align*}
\begin{align*}
	&Y_{k+1,k+1}\bullet  \Omega_6 = \zeta_{\overline{k+1}}^2(k+1+n-r-k-1)\Omega_6 \\
	&+2\zeta_{\overline{k+1}}d\zeta_{\overline{k+1}} \zeta_{\overline{K}}d\zeta_{\overline{K}}^{[k-1]}\overline{dz_K}^{[k]}(\zeta_{k+1} \otimes dz_{k+1})(dz_{\overline{k+1}} \otimes dz_{\overline{k+1}})d\zeta_{L^-}^{[n-r-k-1]}dz_{L^-}^{[r-k-1]}\\
	&-d\zeta_{\overline{K}}^{[k]}\overline{dz_K}^{[k]}(\zeta_{\overline{k+1}} \otimes dz_{k+1})(dz_{\overline{k+1}} \otimes dz_{\overline{k+1}})d\zeta_{L^-}^{[n-r-k-1]}dz_{L^-}^{[r-k-1]}\\
	&+d\zeta_{\overline{K}}^{[k]}\overline{dz_K}^{[k]}(\zeta_{k+1} \otimes dz_{k+1})(dz_{\overline{k+1}} \otimes dz_{\overline{k+1}})2\zeta_{\overline{k+1}}d\zeta_{\overline{k+1}}\zeta_{L^-}d\zeta_{L^-}^{[n-r-k-2]}dz_{L^-}^{[r-k-1]}\\
	&= \zeta_{\overline{k+1}}^2(n-r)\Omega_6 + 2\nu_{k+1} \Omega_{k+1, 2} + \Omega_{k+1, 1} + 2\nu_{k+1}\Omega_{k+1, 3}
\end{align*}
\begin{align*}
	&Y_{k+1,k+1}\bullet  \Omega_7 =  \zeta_{\overline{k+1}}^2(k+2+n-r-k-2)\Omega_7 \\
	&+2\zeta_{\overline{k+1}}d\zeta_{\overline{k+1}} \zeta_{\overline{K}}d\zeta_{\overline{K}}^{[k-1]}\overline{dz_K}^{[k]}(\zeta_{\overline{k+1}} \otimes dz_{\overline{k+1}})(d\zeta_{k+1} \otimes dz_{k+1})d\zeta_{L^-}^{[n-r-k-2]}dz_{L^-}^{[r-k]}\\
	&+d\zeta_{\overline{K}}^{[k]}\overline{dz_K}^{[k]}(1-2\nu_{k+1})(\zeta_{\overline{k+1}} \otimes dz_{k+1})(d\zeta_{\overline{k+1}} \otimes dz_{\overline{k+1}})d\zeta_{L^-}^{[n-r-k-2]}dz_{L^-}^{[r-k]}\\
	&+d\zeta_{\overline{K}}^{[k]}\overline{dz_K}^{[k]}(\zeta_{\overline{k+1}} \otimes dz_{\overline{k+1}})(d\zeta_{k+1} \otimes dz_{k+1})2\zeta_{\overline{k+1}}d\zeta_{\overline{k+1}}\zeta_{L^-}d\zeta_{L^-}^{[n-r-k-3]}dz_{L^-}^{[r-k]}\\
	&= \zeta_{\overline{k+1}}^2(n-r)\Omega_7 - 2\zeta_{\overline{k+1}}^2 \Omega_1 + (1-2\nu_{k+1})\Xi_{5,7}-2\zeta_{\overline{k+1}}^2 \Omega_9
\end{align*}
\begin{align*}
	&Y_{k+1,k+1}\bullet  \Omega_8 =  \zeta_{\overline{k+1}}^2(k+1+n-r-k-1)\Omega_8 \\
	&+2\zeta_{\overline{k+1}}d\zeta_{\overline{k+1}} \zeta_{\overline{K}}d\zeta_{\overline{K}}^{[k-1]}\overline{dz_K}^{[k]}(\zeta_{\overline{k+1}} \otimes dz_{\overline{k+1}})(dz_{k+1} \otimes dz_{k+1})d\zeta_{L^-}^{[n-r-k-1]}dz_{L^-}^{[r-k-1]}\\
	&-d\zeta_{\overline{K}}^{[k]}\overline{dz_K}^{[k]}(\zeta_{\overline{k+1}} \otimes dz_{k+1})(dz_{\overline{k+1}} \otimes dz_{\overline{k+1}})d\zeta_{L^-}^{[n-r-k-1]}dz_{L^-}^{[r-k-1]}\\
	&+d\zeta_{\overline{K}}^{[k]}\overline{dz_K}^{[k]}(\zeta_{\overline{k+1}} \otimes dz_{\overline{k+1}})(dz_{k+1} \otimes dz_{k+1})2\zeta_{\overline{k+1}}d\zeta_{\overline{k+1}}\zeta_{L^-}d\zeta_{L^-}^{[n-r-k-2]}dz_{L^-}^{[r-k-1]}\\
	&=\zeta_{\overline{k+1}}^2(n-r)\Omega_8 -2\zeta_{\overline{k+1}}^2 \Omega_3 + \Omega_{k+1, 1} - 2\zeta_{\overline{k+1}}^2 \Omega_{11}
\end{align*}
\begin{align*}
	&Y_{k+1,k+1}\bullet  \Omega_9 =  \zeta_{\overline{k+1}}^2(k+4+n-r-k-2)\Omega_9
\end{align*}
\begin{align*}
	&Y_{k+1,k+1}\bullet  \Omega_{10} =  \zeta_{\overline{k+1}}^2(k+1+n-r-k-1)\Omega_{10} \\
	&+2\zeta_{\overline{k+1}}d\zeta_{\overline{k+1}} \zeta_{\overline{K}}d\zeta_{\overline{K}}^{[k-1]}\overline{dz_K}^{[k]}(d\zeta_{k+1} \otimes dz_{k+1})(dz_{\overline{k+1}} \otimes dz_{\overline{k+1}})\zeta_{L^-}d\zeta_{L^-}^{[n-r-k-2]}dz_{L^-}^{[r-k-1]}\\
	&+d\zeta_{\overline{K}}^{[k]}\overline{dz_K}^{[k]}(1-2\nu_{k+1})(dz_{\overline{k+1}}\otimes dz_{k+1})(d\zeta_{\overline{k+1}} \otimes dz_{\overline{k+1}})\zeta_{L^-}d\zeta_{L^-}^{[n-r-k-2]}dz_{L^-}^{[r-k-1]}\\
	&=\zeta_{\overline{k+1}}^2(n-r)\Omega_{10}-2\zeta_{\overline{k+1}}d\zeta_{k+1} \Xi_{2,10}+(1-2\nu_{k+1})\Omega_{k+1, 3}
\end{align*}
\begin{align*}
	&Y_{k+1,k+1}\bullet  \Omega_{11} =  \zeta_{\overline{k+1}}^2(k+3+n-r-k-1)\Omega_{11} \\
	&+d\zeta_{\overline{K}}^{[k]}\overline{dz_K}^{[k]}(dz_{\overline{k+1}} \otimes dz_{k+1})(d\zeta_{\overline{k+1}} \otimes dz_{\overline{k+1}})\zeta_{L^-}d\zeta_{L^-}^{[n-r-k-2]}dz_{L^-}^{[r-k-1]}\\
	&=\zeta_{\overline{k+1}}^2(n-r+2)\Omega_{11} + \Omega_{k+1, 3}
\end{align*}
\begin{align*}
	&Y_{k+1,k+1}\bullet  \Omega_{12} =  \zeta_{\overline{k+1}}^2(k+n-r-k)\Omega_{12} \\
	&+2\zeta_{\overline{k+1}}d\zeta_{\overline{k+1}} \zeta_{\overline{K}}d\zeta_{\overline{K}}^{[k-1]}\overline{dz_K}^{[k]}(dz_{k+1} \otimes dz_{k+1})(dz_{\overline{k+1}} \otimes dz_{\overline{k+1}})\zeta_{L^-}d\zeta_{L^-}^{[n-r-k-1]}dz_{L^-}^{[r-k-2]}\\
	&=\zeta_{\overline{k+1}}^2(n-r)\Omega_{12} + 2\zeta_{\overline{k+1}}d\zeta_{\overline{k+1}}\Xi_{4,12}
\end{align*}
\end{extracalc}

Summing up the contributions of each $\Omega_j$, $j=1, \dots, 12$, and canceling terms, the claim finally follows:
\begin{align*}
	Y_{k+1,k+1}\bullet  (\omega_{r,k}\otimes \Theta_1) %&= %(n-r)\zeta_{\overline{k+1}}^2\omega_{r,k}\otimes \Theta_1 + 2\zeta_{\overline{k+1}}^2\Omega_1 + 2(1-\nu_{k+1})\Omega_{k+1, 2} -2\zeta_{\overline{k+1}}d\zeta_{k+1} \Xi_{2,10}\\
%	&+2\zeta_{\overline{k+1}}^2\Omega_3 + 2 \Omega_{k+1, 2}- 2 \zeta_{\overline{k+1}}d\zeta_{\overline{k+1}} \Xi_{4,12}+2\zeta_{\overline{k+1}}^2\Omega_5 - 2 \Xi_{5,7}+ 2\nu_{k+1}\Omega_{k+1, 2} \\
%	&+ 2 \Omega_{k+1, 1} + 2\nu_{k+1}\Omega_{k+1, 3}- 2\zeta_{\overline{k+1}}^2\Omega_1 + 2(1-\nu_{k+1})\Xi_{5,7} - 2\zeta_{\overline{k+1}}^2\Omega_9- 2 \zeta_{\overline{k+1}}^2\Omega_3\\
%	&+ 2 \Omega_{k+1, 1} - 2 \zeta_{\overline{k+1}}^2\Omega_{11}+2\zeta_{\overline{k+1}}^2\Omega_9+ 2\zeta_{\overline{k+1}}d\zeta_{k+1}\Xi_{2,10} + 2(1-\nu_{k+1})\Omega_{k+1, 3}\\
%	&+2\zeta_{\overline{k+1}}^2\Omega_{11} + 2\Omega_{k+1, 3}+ 2\zeta_{\overline{k+1}}d\zeta_{\overline{k+1}} \Xi_{4,12}\\
	%
	&=(n-r)\zeta_{\overline{k+1}}^2\omega_{r,k}\otimes \Theta_1 + 2 (\Omega_{k+1, 1}+ \Omega_{k+1, 2} + \Omega_{k+1, 3}) \\
	&+2\zeta_{\overline{k+1}}^2\Omega_5  -2\nu_{k+1}\Xi_{5,7} \\
	&=(n-r)\zeta_{\overline{k+1}}^2\omega_{r,k}\otimes \Theta_1 + 2 \omega_{r,k+1} \otimes \Theta_1,
\end{align*}
where we used that $\nu_{k+1} \Xi_{5,7} = \zeta_{\overline{k+1}}^2\Omega_5$ in the last step.
\end{proof}

\begin{corollary}\label{cor:YkkOnOmegarkm}
	For $1\leq k<\min\{r,n-r\}$, $m\in\NN$, $m\ge 2$, 
	\begin{align} 
		\LieDer_{\widetilde{Y}_{k+1,k+1}} \omega_{r,k,m} = (n-r+m-2)\zeta_{\overline{k+1}}^{2}\omega_{r,k,m}  + 2 \omega_{r,k+1,m}
	\end{align}
\end{corollary}
\begin{proof}
	A short application of the product rule, \autoref{lem:lieDerYabzZeta} and \autoref{lem:actYkp1OmegaRk} yields
	\begin{align*}
		Y_{k+1,k+1}\bullet  \left(\omega_{r,k,m} \otimes \Theta_1\right) &=Y_{k+1,k+1}\bullet  \left( \zeta_{\overline{1}}^{m-2}\omega_{r,k} \otimes \Theta_1\right)\\
		&=(m-2)\zeta_{\overline{1}}^{m-3} (\zeta_{\overline{k+1}}^2\zeta_{\overline{1}}) \omega_{r,k} \otimes \Theta_1 \\
		&+ \zeta_{\overline{1}}^{m-2}((n-r)\zeta_{\overline{k+1}}^2\omega_{r,k} \otimes \Theta_1 + 2 \omega_{r,k+1} \otimes \Theta_1)\\
		&=(n-r+m-2)\zeta_{\overline{k+1}}^{2}\omega_{r,k,m} \otimes \Theta_1 + 2 \omega_{r,k+1,m} \otimes \Theta_1,
	\end{align*}
	which, by \eqref{eq:LieDerYabTensorTheta} shows the claim.
\end{proof}

\subsubsection{Pairing}
In this section, we determine the pairing from \autoref{sec:pairing} between $D\LieDer_{\widetilde{Y}_{k+1,k+1}}\omega_{r,k,m}$ and $D\omega_{n-r,k+1,m}$, which due to \autoref{cor:YkkOnOmegarkm} and \autoref{lem:pairKWForms} below reduces to a calculation of the product of $\omega_{r,k,m}$ and $D\omega_{r,k+1,m}$. As in \cite{Kotrbaty2022}, we will reduce this calculation to the corresponding product of double forms. However, if we take the product of the two double forms, we would obtain the factor $\Theta_1 \wedge \Theta_1$ in the second component, which vanishes. This can be avoided by replacing one occurrence of $\Theta_1$ by $\Theta_2$. This acts exactly as a renaming of the variables in the second component, so, formally, all calculations stay the same. To reflect the change in the second component notationally, we introduce for $I \subset \Iind$
\begin{align*}
	w_I &= \sum_{i \in I} z_i \otimes d\zeta_i, & \eta_I &= \sum_{i \in I} \zeta_{i} \otimes d\zeta_i.
\end{align*}
Then, clearly, by Eq.~\eqref{eq:sigmarkSep} and Eq.~\eqref{eq:taurkSep},
\begin{align}\label{eq:sigmarkSepDual}
	\sigma_{r,k} \otimes \Theta_2 = \overline{d\eta_K}^{[k]}\eta_{\overline{K}} d\eta_L^{[n-r-k]} dw_L^{[r-k]}dw_{\overline{K}}^{[k-1]} + \overline{d\eta_K}^{[k]}\eta_L d\eta_L^{[n-r-k]}dw_L^{[r-k-1]}dw_{\overline{K}}^{[k]},
\end{align}
and
\begin{align}\label{eq:taurkSepDual}
	\tau_{r,k} \otimes \Theta_2 =& (-1)^{k-1}\overline{d\eta_K}^{[k]}\eta_{\overline{K}} d\eta_L^{[n-r-k]}dw_L^{[r-k]}dw_{\overline{K}}^{[k-1]} \\
	&+ \overline{\eta_K}\overline{d\eta_K}^{[k-1]}d\eta_L^{[n-r-k+1]}dw_L^{[r-k-1]}dw_{\overline{K}}^{[k]}. \nonumber
\end{align}

We further need the following reductions from \cite{Kotrbaty2022}.
\begin{proposition}[\cite{Kotrbaty2022}*{Prop.s~5.2, 6.2 and 6.4}]\label{prop:fillFullZetaDZeta}
	Suppose that $M \subset \Iind$, $|M| = i$. Then
	\begin{align}
		\zeta_M d\zeta_M^{[i-1]} \gamma_M = (-1)^{i-1} \nu_M d\zeta_M^{[i]},\label{eq:redZetaGamma}\\
		\zeta_M dz_M^{[i-1]} \alpha_M = (-1)^{i-1} \nu_M dz_M^{[i]}\label{eq:redZetaZAlpha},
	\end{align}
	and likewise for $\eta_M$, $d\eta_M$, and $dw_M$.
	Moreover, we have for $1 \leq \alpha \leq n-1$ and $1 \leq k \leq \min\{\alpha, n-\alpha\}$
	\begin{align}\label{eq:redFullKTerms}
		\overline{d\zeta_{\overline{K}}}^{[k]} dz_K^{[k]}\overline{d\eta_{K}}^{[k]}dw_{\overline{K}}^{[k]} = \Theta_K \otimes \Theta_K,
	\end{align}
	\begin{align}\label{eq:redFullLTerms}
		\overline{d\zeta_L}^{[n-\alpha-k]}\overline{dz_L}^{[\alpha-k]}d\eta_L^{[\alpha-k]}dw_L^{[n-\alpha-k]} = (-1)^{n+l+\alpha}\binom{n-2k}{\alpha-k} \Theta_L \otimes \Theta_L,
	\end{align}
	and for $1 \leq \alpha \leq n-1$ and $1 \leq k < \min\{\alpha, n-\alpha\}$
	\begin{align}\label{eq:redAddLAlpGamTerms1}
		d\zeta_L^{[n-\alpha-k]}dz_L^{[\alpha-k]}\alpha_L +d\zeta_L^{[n-\alpha-k-1]}dz_L^{[\alpha-k+1]}\gamma_L=0,
	\end{align}
	\begin{align}\label{eq:redAddLAlpGamTerms2}
		\zeta_L d\zeta_L^{[n-\alpha-k]}dz_L^{[\alpha-k-1]}\alpha_L + \zeta_L d\zeta_L^{[n-\alpha-k-1]}dz_L^{[\alpha-k]}\gamma_L = (-1)^{n-1}d\zeta_L^{[n-\alpha-k]}dz_L^{[\alpha-k]}\nu_L.
	\end{align}
\end{proposition}
\begin{proof}
	Eq.~\eqref{eq:redFullKTerms} and Eq.~\eqref{eq:redFullLTerms} are the content of \cite{Kotrbaty2022}*{Prop.~6.4}; Eq.~\eqref{eq:redAddLAlpGamTerms1} and Eq.~\eqref{eq:redAddLAlpGamTerms2} are the content of \cite{Kotrbaty2022}*{Prop.~5.2}. 
	
	For the remaining relations, we will only show Eq.~\eqref{eq:redZetaGamma}, as Eq.~\eqref{eq:redZetaZAlpha} follows with a similar calculation. For $j \in M$,
	\begin{align*}
		\zeta_M d\zeta_M^{[i-1]} \zeta_{\overline{j}} d\zeta_j &= (\zeta_j \otimes dz_j)d\zeta_M^{[i-1]} \zeta_{\overline{j}} d\zeta_j\\
		&=(-1)^{i-1}d\zeta_M^{[i-1]} (d\zeta_j \otimes dz_j) \nu_j,
	\end{align*}
	and the claim follows by summing over $j \in M$.

\end{proof}

We are now in position to calculate the following pairing of differential forms, extending \cite{Kotrbaty2022}*{Lem.~6.7}. Let us note that we denote $|\zeta_{i}|^2 = \zeta_{i} \overline{\zeta_{i}} = \frac{1}{2}\|(\xi_{2i-1},\xi_{2i})\|^2$, where the last term denotes the standard norm on $\CC\cong \RR^2$.
\begin{lemma}\label{lem:pairOmkmOmkm1m}
	Let $r,k,m \in \NN$ such that $1\leq r \leq n-1$, $1\leq k \leq \min\{r,n-r\}$, $m \geq 2$. If also $k+1 \leq \min\{r,n-r\}$, then
	\begin{align}\label{eq:pairOmkmOmkm1m}
		\overline{\omega_{r,k,m}} \wedge D\omega_{n-r,k+1,m} =& 
		(r+m-2)|\zeta_{\overline{1}}|^{2(m-2)} \nu \zeta_{\overline{k+1}}^2 (-1)^{k+r}\\
		& \times (n+m-k-1)\binom{n-2k-2}{r-k-1}\vol_{S\RR^n}.\nonumber
	\end{align}
\end{lemma}
\begin{proof}
	The proof is a lengthy calculation using double forms. As before, we first split the terms of $\overline{\omega_{r,k}}\otimes \Theta_1$ from Eq.~\eqref{eq:OmSepL} according to the partition $K \cup \overline{K}$, $\{k+1, \overline{k+1}\}$ and $L^-$ of the index set $\Iind$. This yields
	\begin{align*}
		\overline{\omega_{r,k}} \otimes \Theta_1
		&= \overline{\zeta_{\overline{K}}}\,\overline{d\zeta_{\overline{K}}}^{[k-1]}dz_K^{[k]} (d\zeta_{\overline{k+1}} \otimes dz_{k+1})(d\zeta_{k+1} \otimes dz_{\overline{k+1}}) \overline{d\zeta_{L^-}}^{[n-r-k-2]} \overline{dz_{L^-}}^{[r-k]} \displaybreak[1]\\
		&+\overline{\zeta_{\overline{K}}}\,\overline{d\zeta_{\overline{K}}}^{[k-1]}dz_K^{[k]} (d\zeta_{\overline{k+1}} \otimes dz_{k+1})(dz_{k+1} \otimes dz_{\overline{k+1}}) \overline{d\zeta_{L^-}}^{[n-r-k-1]} \overline{dz_{L^-}}^{[r-k-1]} \displaybreak[1]\\
		&+\overline{\zeta_{\overline{K}}}\,\overline{d\zeta_{\overline{K}}}^{[k-1]}dz_K^{[k]} (dz_{\overline{k+1}} \otimes dz_{k+1})(d\zeta_{k+1} \otimes dz_{\overline{k+1}}) \overline{d\zeta_{L^-}}^{[n-r-k-1]} \overline{dz_{L^-}}^{[r-k-1]} \displaybreak[1]\\
		&+\overline{\zeta_{\overline{K}}}\,\overline{d\zeta_{\overline{K}}}^{[k-1]}dz_K^{[k]} (dz_{\overline{k+1}} \otimes dz_{k+1})(dz_{k+1} \otimes dz_{\overline{k+1}}) \overline{d\zeta_{L^-}}^{[n-r-k]} \overline{dz_{L^-}}^{[r-k-2]} \displaybreak[1]\\
		\displaybreak[1]\\
		&+ \overline{d\zeta_{\overline{K}}}^{[k]}dz_K^{[k]}(\zeta_{\overline{k+1}} \otimes dz_{k+1})(d\zeta_{k+1} \otimes dz_{\overline{k+1}})\overline{d\zeta_{L^-}}^{[n-r-k-2]}\overline{dz_{L^-}}^{[r-k]}\displaybreak[1]\\
		&+ \overline{d\zeta_{\overline{K}}}^{[k]}dz_K^{[k]}(\zeta_{\overline{k+1}} \otimes dz_{k+1})(dz_{k+1} \otimes dz_{\overline{k+1}})\overline{d\zeta_{L^-}}^{[n-r-k-1]}\overline{dz_{L^-}}^{[r-k-1]}\displaybreak[1]\\
		&+ \overline{d\zeta_{\overline{K}}}^{[k]}dz_K^{[k]}(\zeta_{k+1} \otimes dz_{\overline{k+1}})(d\zeta_{\overline{k+1}} \otimes dz_{k+1})\overline{d\zeta_{L^-}}^{[n-r-k-2]}\overline{dz_{L^-}}^{[r-k]}\displaybreak[1]\\
		&+ \overline{d\zeta_{\overline{K}}}^{[k]}dz_K^{[k]}(\zeta_{k+1} \otimes dz_{\overline{k+1}})(dz_{\overline{k+1}} \otimes dz_{k+1})\overline{d\zeta_{L^-}}^{[n-r-k-1]}\overline{dz_{L^-}}^{[r-k-1]}\displaybreak[1]\\
		\displaybreak[1]\\
		&+ \overline{d\zeta_{\overline{K}}^{[k]}}dz_K^{[k]}(d\zeta_{\overline{k+1}} \otimes dz_{k+1})(d\zeta_{k+1} \otimes dz_{\overline{k+1}})\overline{\zeta_{L^-}}\,\overline{d\zeta_{L^-}}^{[n-r-k-3]}\overline{dz_{L^-}}^{[r-k]}\displaybreak[1]\\
		&+ \overline{d\zeta_{\overline{K}}^{[k]}}dz_K^{[k]}(d\zeta_{\overline{k+1}} \otimes dz_{k+1})(dz_{k+1} \otimes dz_{\overline{k+1}})\overline{\zeta_{L^-}}\,\overline{d\zeta_{L^-}}^{[n-r-k-2]}\overline{dz_{L^-}}^{[r-k-1]}\displaybreak[1]\\
		&+ \overline{d\zeta_{\overline{K}}^{[k]}}dz_K^{[k]}(dz_{\overline{k+1}} \otimes dz_{k+1})(d\zeta_{k+1} \otimes dz_{\overline{k+1}})\overline{\zeta_{L^-}}\,\overline{d\zeta_{L^-}}^{[n-r-k-2]}\overline{dz_{L^-}}^{[r-k-1]}\displaybreak[1]\\
		&+ \overline{d\zeta_{\overline{K}}^{[k]}}dz_K^{[k]}(dz_{\overline{k+1}} \otimes dz_{k+1})(dz_{k+1} \otimes dz_{\overline{k+1}})\overline{\zeta_{L^-}}\,\overline{d\zeta_{L^-}}^{[n-r-k-1]}\overline{dz_{L^-}}^{[r-k-2]}\displaybreak[1]\\
		&=: \Omega_1 + \dots + \Omega_{12},
	\end{align*}
	where we label the twelve appearing terms $\Omega_1, \dots, \Omega_{12}$. Note that this is just the conjugate of the formula in the proof of \autoref{lem:actYkp1OmegaRk} and that we abuse notation by reusing the label $\Omega_i$ for the conjugated terms. We next do the same splitting for the forms $\sigma_{n-r,k+1} \otimes \Theta_2$ and $\tau_{n-r,k+1} \otimes \Theta_2$ from Eq.s~\eqref{eq:sigmarkSepDual} and \eqref{eq:taurkSepDual}:
	\begin{align*}%\label{eq:sigmaTaurkSepDual}
		\sigma_{n-r,k+1} \otimes \Theta_2 &= \overline{d\eta_{K^+}}^{[k+1]}\eta_{\overline{K^+}}dw_{\overline{K^+}}^{[k]} d\eta_{L^-}^{[r-k-1]} dw_{L^-}^{[n-r-k-1]}\\
		& + (-1)^{k+1}\overline{d\eta_{K^+}}^{[k+1]}dw_{\overline{K^+}}^{[k+1]}\eta_{L^-} d\eta_{L^-}^{[r-k-1]}dw_{L^-}^{[n-r-k-2]}\\
		&=(-1)^{k}\overline{d\eta_{K}}^{[k]}dw_{\overline{K}}^{[k]}(d\zeta_{\overline{k+1}} \otimes d\zeta_{k+1})(\zeta_{\overline{k+1}}\otimes d\zeta_{\overline{k+1}}) d\eta_{L^-}^{[r-k-1]} dw_{L^-}^{[n-r-k-1]}\\
		&-\overline{d\eta_{K}}^{[k]}\eta_{\overline{K}}dw_{\overline{K}}^{[k-1]}(d\zeta_{\overline{k+1}} \otimes d\zeta_{k+1})(dz_{\overline{k+1}}\otimes d\zeta_{\overline{k+1}}) d\eta_{L^-}^{[r-k-1]} dw_{L^-}^{[n-r-k-1]}\\
		&+ (-1)^{k+1}\overline{d\eta_{K}}^{[k]}dw_{\overline{K}}^{[k]}(d\zeta_{\overline{k+1}}\otimes d\zeta_{k+1})(dz_{\overline{k+1}} \otimes d\zeta_{\overline{k+1}})\eta_{L^-} d\eta_{L^-}^{[r-k-1]}dw_{L^-}^{[n-r-k-2]}\\
		&=:\Xi_1 + \Xi_2 + \Xi_3,
	\end{align*}
	\begin{align*}
		\tau_{n-r,k+1} \otimes \Theta_2 &= (-1)^{k}\overline{d\eta_{K^+}}^{[k+1]}\eta_{\overline{K^+}}dw_{\overline{K^+}}^{[k]} d\eta_{L^-}^{[r-k-1]} dw_{L^-}^{[n-r-k-1]} \\
		&+ \overline{\eta_{K^+}}\overline{d\eta_{K^+}}^{[k]}dw_{\overline{K^+}}^{[k+1]}d\eta_{L^-}^{[r-k]}dw_{L^-}^{[n-r-k-2]}\\
		&=\overline{d\eta_{K}}^{[k]}dw_{\overline{K}}^{[k]}(d\zeta_{\overline{k+1}} \otimes d\zeta_{k+1})(\zeta_{\overline{k+1}}\otimes d\zeta_{\overline{k+1}}) d\eta_{L^-}^{[r-k-1]} dw_{L^-}^{[n-r-k-1]}\\
		&+(-1)^{k-1}\overline{d\eta_{K}}^{[k]}\eta_{\overline{K}}dw_{\overline{K}}^{[k-1]}(d\zeta_{\overline{k+1}} \otimes d\zeta_{k+1})(dz_{\overline{k+1}}\otimes d\zeta_{\overline{k+1}}) d\eta_{L^-}^{[r-k-1]} dw_{L^-}^{[n-r-k-1]}\\
		&+\overline{d\eta_{K}}^{[k]}dw_{\overline{K}}^{[k]}(\zeta_{\overline{k+1}} \otimes d\zeta_{k+1})(dz_{\overline{k+1}}\otimes d\zeta_{\overline{k+1}})d\eta_{L^-}^{[r-k]}dw_{L^-}^{[n-r-k-2]}\\
		&+\overline{\eta_{K}}\overline{d\eta_{K}}^{[k-1]}dw_{\overline{K}}^{[k]}(d\zeta_{\overline{k+1}} \otimes d\zeta_{k+1})(dz_{\overline{k+1}}\otimes d\zeta_{\overline{k+1}})d\eta_{L^-}^{[r-k]}dw_{L^-}^{[n-r-k-2]}\\
		&=: (-1)^{k}\Xi_1 + (-1)^{k}\Xi_2 + \Upsilon_3 + \Upsilon_4.
	\end{align*}
	Here, we label the terms of $\sigma_{n-r,k+1} \otimes \Theta_2$ by $\Xi_1$, $\Xi_1$ and $\Xi_3$, whereas the first two terms of $\tau_{n-r,k+1} \otimes \Theta_2$ are given by $(-1)^k\Xi_1$ and $(-1)^k\Xi_2$ and we label the last two terms by $\Upsilon_3$ and $\Upsilon_4$.
	
	In order to calculate the wedge product in Eq.~\eqref{eq:pairOmkmOmkm1m}, note that, by Eq.~\eqref{eq:DOmrkm} in  \autoref{prop:OmSigTauSepL}, we need to calculate all products of the form $\Omega_i \wedge \Xi_j \wedge \alpha$ and $\Omega_i \wedge \Upsilon_j \wedge \alpha$. Since we are considering the restriction of these forms to $S\RR^n$, we are only interested in these products up to multiples of $\gamma$, i.e. we will exploit that the restriction of two forms on $\RR^n\times\RR^n$ to $S\RR^n$ coincide if and only if their products with $\gamma$ coincide on $\RR^n\times\RR ^n$, compare \cite{Kotrbaty2022}*{Lem.~4.1}. We will therefore multiply all of these terms with $\gamma$, which will simplify the calculation.
	
	As it turns out, most of the products vanish. \autoref{tab:pair0termsFirst} and \autoref{tab:pairWedgAlphGam} list the products which vanish due to the following reasons:
	
	First, note that $\Xi_1, \Xi_2, \Xi_3$, and $\Upsilon_4$ contain the factor $d\zeta_{\overline{k+1}} \otimes 1$. Hence their product vanishes with other such terms, marked by $0^\dagger$ in \autoref{tab:pair0termsFirst}.
	
	Second, as $\Xi_2, \Xi_3, \Upsilon_3$, and $\Upsilon_4$ contain the factor $dz_{\overline{k+1}} \otimes 1$, their product vanishes with other such terms, marked by $0^\ddagger$ (if not already marked by $0^\dagger$).
	
	Third, we count the appearances of $d\zeta_j$ resp.\ $dz_j$, $j \in L^-$, in the first component (that is, e.g., the factors $d\zeta_j \otimes 1$). If there are more than $|L^-| = n-2k-2$ terms of one kind, there must be a repetition, so the wedge product vanishes. We mark such situations by $0^{\zeta,L^-}$ if they appear for some $d\zeta_j$, and by $0^{z,L^-}$ if they appear for some $dz_j$, in case the specific term is not already zero by the other reasons above (that is, marked by $0^\dagger$ or $0^\ddagger$).\\
	
	\begin{table}
	\begin{tabular}{|c||c|c|c|c|c|}
		\hline
		$\wedge$& $\Xi_1$ & $\Xi_2$ & $\Xi_3$ & $\Upsilon_3$ & $\Upsilon_4$ \\
		\hline\hline
		$\Omega_1$ & $0^\dagger$ & $0^\dagger$ & $0^\dagger$ & &$0^\dagger$\\
		\hline
		$\Omega_2$& $0^\dagger$ & $0^\dagger$ & $0^\dagger$ & $0^{\zeta,L^-}$ &$0^\dagger$\\
		\hline
		$\Omega_3$&  & $0^\ddagger$ & $0^\ddagger$ & $0^\ddagger$ &$0^\ddagger$\\
		\hline
		$\Omega_4$& $0^{\zeta,{L^-}}$ & $0^\ddagger$ & $0^\ddagger$ &$0^\ddagger$ &$0^\ddagger$ \\
		\hline
		$\Omega_5$& $0^{z,L^-}$ & $0^{z,L^-}$ &  &  &\\
		\hline
		$\Omega_6$&  &  &  & $0^{\zeta,L^-}$ &$0^{\zeta,L^-}$\\
		\hline
		$\Omega_7$& $0^\dagger$ & $0^\dagger$ & $0^\dagger$ &  &$0^\dagger$\\
		\hline
		$\Omega_8$&  & $0^\ddagger$ & $0^\ddagger$ & $0^\ddagger$ &$0^\ddagger$ \\
		\hline
		$\Omega_9$& $0^\dagger$ & $0^\dagger$ & $0^\dagger$ &  &$0^\dagger$\\
		\hline
		$\Omega_{10}$& $0^\dagger$ & $0^\dagger$ & $0^\dagger$ &  &$0^\dagger$\\
		\hline
		$\Omega_{11}$&  & $0^\ddagger$ & $0^\ddagger$ & $0^\ddagger$ &$0^\ddagger$\\
		\hline
		$\Omega_{12}$&  & $0^\ddagger$ & $0^\ddagger$ & $0^\ddagger$ &$0^\ddagger$\\
		\hline
	\end{tabular}
	\centering
	\caption{trivial terms}
	\label{tab:pair0termsFirst}
	\end{table}
	
	By \autoref{tab:pair0termsFirst}, $14$ products remain, and we need to multiply the corresponding terms with $\alpha \wedge \gamma$. We split $\alpha$ and $\gamma$ according to the blocks $K$, $\overline{K}$, $\{k+1\}$, $\{\overline{k+1}\}$, and $L^-$, using that
	\begin{align*}
		\alpha = \alpha_K + \alpha_{\overline{K}} + \alpha_{k+1} + \alpha_{\overline{k+1}} + \alpha_{L^-},
	\end{align*} 
	and similarly for $\gamma$.
	
	As before, many of the wedge products are zero for trivial reasons, that is, because a $1$-form appears twice. We have listed the remaining terms in \autoref{tab:pairWedgAlphGam} and marked the zero terms as indicated below. Here, we denote by a single symbol those terms which vanish due to the product of the $1$-form with $\Omega_j$, and by a double symbol those which vanish due to the product with $\Xi_j$ resp. $\Upsilon_j$. The following products vanish trivially:
	\begin{itemize}
		\item $\Omega_j \wedge \alpha_K$ for all $j$, hence the column is omitted in \autoref{tab:pairWedgAlphGam};
		\item $\Xi_j \wedge \alpha_{\overline{K}}$, $j =1, 3$, marked by $0^*$;
		\item $\Upsilon_j \wedge \alpha_{\overline{K}}$, for all $j$, marked by $0^{**}$;
		\item $\Omega_j\wedge \gamma_K$ for $5 \leq j \leq 12$, marked by $0^x$;
		\item $\Xi_j \wedge \gamma_{\overline{K}}$, for all $j$, and $\Upsilon_3 \wedge \gamma_{\overline{K}}$, marked by $0^{xx}$;
		\item $\Omega_j \wedge \alpha_{k+1}$, $j=6, 10, 12$, marked by $0^\circ$;
		\item $\Omega_j \wedge \alpha_{\overline{k+1}}$, $j=3, 8, 11, 12$ marked by $0^\dagger$;
		\item $\Xi_j \wedge \alpha_{\overline{k+1}}$, $j=2, 3$, and $\Upsilon_i \wedge \alpha_{\overline{k+1}}$, for all $i$, marked by $0^\ddagger$;
		\item $\Omega_j \wedge \gamma_{k+1}$, $j=1, 3, 5, 9, 11$, marked by $0^\times$;
		\item $\Omega_j \wedge \gamma_{\overline{k+1}}$, $j=1, 7, 9, 10$, marked by $0^\star$;
		\item $\Xi_j \wedge \gamma_{\overline{k+1}}$, for all $j$, and $\Upsilon_4 \wedge \gamma_{\overline{k+1}}$, marked by $0^{\star\star}$.
	\end{itemize}
	Next, we count the $d\zeta$- resp. $dz$-terms in $L^-$ in the first component. As $|L^-| = n-2k-2$, all products with more than $n-2k-2$ such terms vanish:
	\begin{itemize}
		\item $\Xi_j \wedge \gamma_{L^-}$ , $j=1,2,3$, have $r-k$ terms in $d\zeta$. Hence, their product with $\Omega_i$, $i=3, 6, 8, 12$, vanishes, marked by $0^{\zeta}$.
		\item $\Upsilon_j \wedge \gamma_{L^-}$ , $j=3, 4$, have $r-k+1$ terms in $d\zeta$. Hence, their product with $\Omega_i$, $i\neq 9$, vanishes, marked by $0^{\zeta'}$.
		\item $\Xi_j \wedge \alpha_{L^-}$ , $j=1,2$, have $n-r-k$ terms in $dz$. Hence, their product with $\Omega_i$, $i\neq 4, 12$, vanishes, marked by $0^{z}$.
		\item $\Xi_3 \wedge \alpha_{L^-}$ and $\Upsilon_j \wedge \alpha_{L^-}$ , $j=3, 4$, have $n-r-k-1$ terms in $dz$. Hence, their product with $\Omega_i$, $i=1, 5, 7, 9$, vanishes, marked by $0^{z'}$.
	\end{itemize}
	
	\begin{table}[h]
		\begin{tabular}{|c||c|c|c|c||c|c|c|c|c|}
			\hline
			$\wedge$ & $\alpha_{\overline{K}}$ & $\alpha_{k+1}$ & $\alpha_{\overline{k+1}}$ & $\alpha_{L^-}$& $\gamma_K$ & $\gamma_{\overline{K}}$ & $\gamma_{k+1}$ & $\gamma_{\overline{k+1}}$ & $\gamma_{L^-}$ \\
			\hline\hline
			$\Omega_3 \wedge \Xi_1$ & $0^*$ & & $0^\dagger$ & $0^{z}$ && $0^{xx}$ & $0^\times$ & $0^{\star\star}$ & $0^{\zeta}$\\
			\hline
			$\Omega_6 \wedge \Xi_1$ & $0^*$ & $0^\circ$  && $0^{z}$& $0^x$ & $0^{xx}$ && $0^{\star\star}$ & $0^{\zeta}$\\
			\hline
			$\Omega_8 \wedge \Xi_1$ & $0^*$ && $0^\dagger$ & $0^{z}$& $0^x$ & $0^{xx}$ && $0^{\star\star}$ & $0^{\zeta}$\\
			\hline
			$\Omega_{11} \wedge \Xi_1$ & $0^*$ && $0^\dagger$ & $0^{z}$& $0^x$ & $0^{xx}$ & $0^\times$ & $0^{\star\star}$ & \\
			\hline
			$\Omega_{12} \wedge \Xi_1$ & $0^*$ & $0^\circ$  & $0^\dagger$ && $0^x$ & $0^{xx}$ && $0^{\star\star}$ & $0^{\zeta}$\\
			\hline
			$\Omega_6 \wedge \Xi_2$& & $0^\circ$  & $0^\ddagger$ & $0^{z}$& $0^x$ & $0^{xx}$ && $0^{\star\star}$ & $0^{\zeta}$\\
			\hline
			$\Omega_5 \wedge \Xi_3$ & $0^*$ && $0^\ddagger$ & $0^{z'}$& $0^x$ & $0^{xx}$ & $0^\times$ & $0^{\star\star}$ &\\
			\hline
			$\Omega_6 \wedge \Xi_3$ & $0^*$ & $0^\circ$  & $0^\ddagger$ && $0^x$ & $0^{xx}$ && $0^{\star\star}$ & $0^{\zeta}$\\
			\hline
			$\Omega_1 \wedge\Upsilon_3$& $0^{**}$ && $0^\ddagger$ & $0^{z'}$&& $0^{xx}$ & $0^\times$ & $0^\star$ & $0^{\zeta'}$\\
			\hline
			$\Omega_5 \wedge \Upsilon_3$ & $0^{**}$ && $0^\ddagger$ & $0^{z'}$& $0^x$ & $0^{xx}$ & $0^\times$ && $0^{\zeta'}$\\
			\hline
			$\Omega_7 \wedge \Upsilon_3$ & $0^{**}$ && $0^\ddagger$ & $0^{z'}$& $0^x$ & $0^{xx}$ && $0^\star$ & $0^{\zeta'}$\\
			\hline
			$\Omega_9 \wedge \Upsilon_3$ & $0^{**}$ && $0^\ddagger$ & $0^{z'}$& $0^x$ & $0^{xx}$ & $0^\times$ & $0^\star$ & \\
			\hline
			$\Omega_{10} \wedge \Upsilon_3$ & $0^{**}$ & $0^\circ$  & $0^\ddagger$ && $0^x$ & $0^{xx}$ && $0^\star$ & $0^{\zeta'}$\\
			\hline
			$\Omega_5 \wedge \Upsilon_4$ & $0^{**}$ && $0^\ddagger$ & $0^{z'}$& $0^x$ && $0^\times$ & $0^{\star\star}$ & $0^{\zeta'}$\\
			\hline
		\end{tabular}
		\centering
		\caption{Wedge products with $\alpha$ and $\gamma$}
		\label{tab:pairWedgAlphGam}
	\end{table}
Let us point out that in every row of the table there are exactly two non-zero entries, that is, every row yields only one term in the total product.

To reduce the necessary calculations further, note that the following relations hold (using the first part of \autoref{prop:fillFullZetaDZeta}):
\begin{align}
	\nu_{k+1} \Omega_1 \wedge \gamma_K &= \nu_K \Omega_7 \wedge \gamma_{k+1}\label{eq:pairRelOm1Om5}\\
	\nu_{k+1}\Omega_3 \wedge \gamma_K &= \nu_{K} \Omega_8 \wedge \gamma_{k+1} \label{eq:pairRelOm3Om8}\\
	\nu_{k+1} \Xi_2 \wedge \alpha_{\overline{K}} &= \nu_{K}\Xi_1 \wedge \alpha_{\overline{k+1}} \label{eq:pairRelXi1Xi2}\\
	\nu_{k+1}\Upsilon_4 \wedge\gamma_{\overline{K}} &= \nu_{K} \Upsilon_3 \wedge\gamma_{\overline{k+1}} \label{eq:pairRelUps4Ups3}
\end{align}
Using \eqref{eq:pairRelOm1Om5} to \eqref{eq:pairRelUps4Ups3}, we can replace terms with $\alpha_{\overline{K}}$, $\gamma_K$ or $\gamma_{\overline{K}}$ by terms with $\alpha_{k+1}$, $\alpha_{\overline{k+1}}$, $\gamma_{k+1}$ and $\gamma_{\overline{k+1}}$.

As the proofs of the relations \eqref{eq:pairRelOm1Om5} to \eqref{eq:pairRelUps4Ups3} are very similar, we will only prove Eq.~\eqref{eq:pairRelOm1Om5} here and omit the other proofs. For $\Omega_1 \wedge \gamma_K$, Eq.~\eqref{eq:redZetaGamma} in \autoref{prop:fillFullZetaDZeta} implies
\begin{align*}
	\Omega_1 \wedge \gamma_K &=(-1)^{n+k}\overline{\zeta_{\overline{K}}\,d\zeta_{\overline{K}}^{[k-1]}\gamma_{\overline{K}}} dz_K^{[k]} (d\zeta_{\overline{k+1}} \otimes dz_{k+1})(d\zeta_{k+1} \otimes dz_{\overline{k+1}}) \overline{d\zeta_{L^-}}^{[n-r-k-2]} \overline{dz_{L^-}}^{[r-k]}\\
	&=(-1)^{n+1}\nu_K\overline{d\zeta_{\overline{K}}}^{[k]} dz_K^{[k]} (d\zeta_{\overline{k+1}} \otimes dz_{k+1})(d\zeta_{k+1} \otimes dz_{\overline{k+1}}) \overline{d\zeta_{L^-}}^{[n-r-k-2]} \overline{dz_{L^-}}^{[r-k]}.
\end{align*}
However, a short calculation for $\Omega_7 \wedge \gamma_{k+1}$ shows that
\begin{align*}
	\Omega_7 \wedge \gamma_{k+1} &=(-1)^{n+1}\overline{d\zeta_{\overline{K}}}^{[k]}dz_K^{[k]}\zeta_{\overline{k+1}}d\zeta_{k+1}(\zeta_{k+1} \otimes dz_{\overline{k+1}})(d\zeta_{\overline{k+1}} \otimes dz_{k+1})\overline{d\zeta_{L^-}}^{[n-r-k-2]}\overline{dz_{L^-}}^{[r-k]}\\
	&=(-1)^{n+1}\nu_{k+1}\overline{d\zeta_{\overline{K}}}^{[k]}dz_K^{[k]}(d\zeta_{k+1} \otimes dz_{\overline{k+1}})(d\zeta_{\overline{k+1}} \otimes dz_{k+1})\overline{d\zeta_{L^-}}^{[n-r-k-2]}\overline{dz_{L^-}}^{[r-k]},
\end{align*}
which implies Eq.~\eqref{eq:pairRelOm1Om5}.

\begin{extracalc}
	\begin{align*}
		\Omega_3 \wedge \gamma_K &=(-1)^{n+k} \overline{\zeta_{\overline{K}}\,d\zeta_{\overline{K}}^{[k-1]}\gamma_{\overline{K}}}dz_K^{[k]} (dz_{\overline{k+1}} \otimes dz_{k+1})(d\zeta_{k+1} \otimes dz_{\overline{k+1}}) \overline{d\zeta_{L^-}}^{[n-r-k-1]} \overline{dz_{L^-}}^{[r-k-1]}\\
		&=(-1)^{n+1}\nu_K\overline{d\zeta_{\overline{K}}}^{[k]}dz_K^{[k]} (dz_{\overline{k+1}} \otimes dz_{k+1})(d\zeta_{k+1} \otimes dz_{\overline{k+1}}) \overline{d\zeta_{L^-}}^{[n-r-k-1]} \overline{dz_{L^-}}^{[r-k-1]}
	\end{align*}
	\begin{align*}
		\Omega_8 \wedge \gamma_{k+1} &=(-1)^{n+1} \overline{d\zeta_{\overline{K}}}^{[k]}dz_K^{[k]}\zeta_{\overline{k+1}}d\zeta_{k+1}(\zeta_{k+1} \otimes dz_{\overline{k+1}})(dz_{\overline{k+1}} \otimes dz_{k+1})\overline{d\zeta_{L^-}}^{[n-r-k-1]}\overline{dz_{L^-}}^{[r-k-1]}\\
		&=(-1)^{n+1}\nu_{k+1} \overline{d\zeta_{\overline{K}}}^{[k]}dz_K^{[k]}(d\zeta_{k+1} \otimes dz_{\overline{k+1}})(dz_{\overline{k+1}} \otimes dz_{k+1})\overline{d\zeta_{L^-}}^{[n-r-k-1]}\overline{dz_{L^-}}^{[r-k-1]}
	\end{align*}
	Hence, Eq.~\eqref{eq:pairRelOm3Om8} follows.
	
	\begin{align*}
		\Xi_1 \wedge \alpha_{\overline{k+1}} &=(-1)^{n}(-1)^{k}\overline{d\eta_{K}}^{[k]}dw_{\overline{K}}^{[k]}(d\zeta_{\overline{k+1}} \otimes d\zeta_{k+1})\zeta_{k+1}dz_{\overline{k+1}}(\zeta_{\overline{k+1}}\otimes d\zeta_{\overline{k+1}}) d\eta_{L^-}^{[r-k-1]} dw_{L^-}^{[n-r-k-1]}\\
		&=(-1)^{n+k}\nu_{k+1}\overline{d\eta_{K}}^{[k]}dw_{\overline{K}}^{[k]}(d\zeta_{\overline{k+1}} \otimes d\zeta_{k+1})(dz_{\overline{k+1}}\otimes d\zeta_{\overline{k+1}}) d\eta_{L^-}^{[r-k-1]} dw_{L^-}^{[n-r-k-1]}
	\end{align*}
	\begin{align*}
		\Xi_2 \wedge \alpha_{\overline{K}} &= (-1)^{n}(-1)\overline{d\eta_{K}}^{[k]}\eta_{\overline{K}}dw_{\overline{K}}^{[k-1]}\alpha_{\overline{K}}(d\zeta_{\overline{k+1}} \otimes d\zeta_{k+1})(dz_{\overline{k+1}}\otimes d\zeta_{\overline{k+1}}) d\eta_{L^-}^{[r-k-1]} dw_{L^-}^{[n-r-k-1]}\\
		&=(-1)^{n+k}\nu_{\overline{K}}\overline{d\eta_{K}}^{[k]}dw_{\overline{K}}^{[k]}(d\zeta_{\overline{k+1}} \otimes d\zeta_{k+1})(dz_{\overline{k+1}}\otimes d\zeta_{\overline{k+1}}) d\eta_{L^-}^{[r-k-1]} dw_{L^-}^{[n-r-k-1]}
	\end{align*}
	Hence, Eq.~\eqref{eq:pairRelXi1Xi2} follows
	\begin{align*}
		\Upsilon_4 \wedge \gamma_{\overline{K}} &= (-1)^{n+k}
		\overline{\eta_{K}d\eta_{K}^{[k-1]}\gamma_K}dw_{\overline{K}}^{[k]}(d\zeta_{\overline{k+1}} \otimes d\zeta_{k+1})(dz_{\overline{k+1}}\otimes d\zeta_{\overline{k+1}})d\eta_{L^-}^{[r-k]}dw_{L^-}^{[n-r-k-2]}\\
		&=(-1)^{n+1}\nu_{K}
		\overline{d\eta_{K}}^{[k]}dw_{\overline{K}}^{[k]}(d\zeta_{\overline{k+1}} \otimes d\zeta_{k+1})(dz_{\overline{k+1}}\otimes d\zeta_{\overline{k+1}})d\eta_{L^-}^{[r-k]}dw_{L^-}^{[n-r-k-2]}
	\end{align*}
	\begin{align*}
		\Upsilon_3 \wedge \gamma_{\overline{k+1}} &= (-1)^{n+1}	\overline{d\eta_{K}}^{[k]}dw_{\overline{K}}^{[k]}\zeta_{k+1}d\zeta_{\overline{k+1}}(\zeta_{\overline{k+1}} \otimes d\zeta_{k+1})(dz_{\overline{k+1}}\otimes d\zeta_{\overline{k+1}})d\eta_{L^-}^{[r-k]}dw_{L^-}^{[n-r-k-2]}\\
		&= (-1)^{n+1}\nu_{k+1}	\overline{d\eta_{K}}^{[k]}dw_{\overline{K}}^{[k]}(d\zeta_{\overline{k+1}} \otimes d\zeta_{k+1})(dz_{\overline{k+1}}\otimes d\zeta_{\overline{k+1}})d\eta_{L^-}^{[r-k]}dw_{L^-}^{[n-r-k-2]}
	\end{align*}
	Hence, Eq.~\eqref{eq:pairRelUps4Ups3} follows.
\end{extracalc}

In addition, the following equalities can be seen directly from the definitions:
\begin{align}
	\Omega_5 \wedge \gamma_{\overline{k+1}} &= \Omega_7 \wedge \gamma_{k+1}\\
	\Omega_6 \wedge \alpha_{\overline{k+1}} &= \Omega_8 \wedge \alpha_{k+1}
\end{align}
We therefore obtain the following relations between the rows of \autoref{tab:pairWedgAlphGam}:
\begin{align*}
	\nu_{k+1}\Omega_3 \wedge \Xi_1 \wedge \alpha_{k+1}\gamma_{K} =& \nu_{K} \Omega_8 \wedge \Xi_1 \wedge \alpha_{k+1} \gamma_{k+1} &= \nu_{K} \Omega_6 \wedge \Xi_1 \wedge \alpha_{\overline{k+1}} \gamma_{k+1}\\
	&\nu_{k+1}\Omega_6 \wedge \Xi_2 \wedge \alpha_{\overline{K}} \gamma_{k+1} &= \nu_{K} \Omega_6 \wedge \Xi_1 \wedge \alpha_{\overline{k+1}} \gamma_{k+1}\\
	\nu_{k+1}\Omega_1 \wedge \Upsilon_3 \wedge \alpha_{k+1} \gamma_{K} =& \nu_{K} \Omega_7 \wedge \Upsilon_3 \wedge \alpha_{k+1} \gamma_{k+1} &= \nu_{K}\Omega_5 \wedge \Upsilon_3 \wedge \alpha_{k+1} \gamma_{\overline{k+1}}\\
	&\nu_{k+1}\Omega_5 \wedge \Upsilon_4 \wedge \alpha_{k+1} \gamma_{\overline{K}} &= \nu_{K}\Omega_5 \wedge \Upsilon_3 \wedge \alpha_{k+1} \gamma_{\overline{k+1}} 
\end{align*}
We will calculate the terms on the right hand side of these equations. Note that these have the property that the $\alpha$- and $\gamma$-terms fill up the first component of the double form in the $\{k+1,\overline{k+1}\}$-part of the index set. Interchanging the order and using Eq.~\eqref{eq:redFullKTerms} and Eq.~\eqref{eq:redFullLTerms} then immediately yields
\begin{align*}
	&\Omega_6 \wedge \Xi_1 \wedge \alpha_{\overline{k+1}} \wedge \gamma_{k+1}\\
	=& \nu_{k+1} (-1)^{n+1} (\Omega_6 \wedge dz_{\overline{k+1}})\wedge (\Xi_1 \wedge d\zeta_{k+1}) \\
	=&\nu_{k+1} \zeta_{\overline{k+1}}^2 (-1)^{n+1}(-1)^{n+1}(-1)^{k}(-1)^{n}\overline{d\zeta_{\overline{K}}}^{[k]}dz_K^{[k]} \overline{d\eta_{K}}^{[k]}dw_{\overline{K}}^{[k]}\\
	&\wedge(dz_{\overline{k+1}} \otimes dz_{k+1})(dz_{k+1} \otimes dz_{\overline{k+1}})(d\zeta_{\overline{k+1}} \otimes d\zeta_{k+1})(d\zeta_{k+1}\otimes d\zeta_{\overline{k+1}})\\
	&\wedge\overline{d\zeta_{L^-}}^{[n-r-k-1]}\overline{dz_{L^-}}^{[r-k-1]}d\eta_{L^-}^{[r-k-1]} dw_{L^-}^{[n-r-k-1]} \\
	=&\nu_{k+1} \zeta_{\overline{k+1}}^2 (-1)^{n+k} (\Theta_{K^+} \otimes \Theta_{K^+}) \left((-1)^{n+l+r}\binom{n-2k-2}{r-k-1}\Theta_{L^-} \otimes \Theta_{L^-}\right) \\
	=&\nu_{k+1} \zeta_{\overline{k+1}}^2 (-1)^{k+l+r}\binom{n-2k-2}{r-k-1}\Theta\otimes \Theta.
\end{align*}
Similarly, we obtain
\begin{align*}
	\Omega_5 \wedge \Upsilon_3 \wedge \alpha_{k+1} \wedge \gamma_{\overline{k+1}} = \nu_{k+1}\zeta_{\overline{k+1}}^2 (-1)^{l+r}\binom{n-2k-2}{r-k} \Theta \otimes \Theta.
\end{align*}
The remaining rows of \autoref{tab:pairWedgAlphGam} (containing $\alpha_{L^-}$- and $\gamma_{L^-}$-terms) will be paired as follows:
\begin{align*}
	\Omega_{11} \wedge \Xi_1 \wedge \alpha_{k+1} \wedge \gamma_{L^-} &+ \Omega_{12} \wedge \Xi_1 \wedge \alpha_{L^-} \wedge \gamma_{k+1}\\
	\Omega_5 \wedge \Xi_3 \wedge \alpha_{k+1} \wedge \gamma_{L^-} &+ \Omega_6 \wedge \Xi_3 \wedge \alpha_{L^-} \wedge \gamma_{k+1}\\
	\Omega_9 \wedge \Upsilon_3 \wedge \alpha_{k+1} \wedge \gamma_{L^-} &+
	\Omega_{10} \wedge \Upsilon_3 \wedge \alpha_{L^-} \wedge \gamma_{k+1}
\end{align*}
By a short calculation using Eq.s~\eqref{eq:redAddLAlpGamTerms1} and \eqref{eq:redAddLAlpGamTerms2}, we obtain
\begin{align*}
	\Omega_{11} \wedge \Xi_1 \wedge \alpha_{k+1} \wedge \gamma_{L^-} &+ \Omega_{12} \wedge \Xi_1 \wedge \alpha_{L^-} \wedge \gamma_{k+1} &=&\zeta_{\overline{k+1}}^2\nu_{L^-}(-1)^{k+l+r}\binom{n-2k-2}{r-k-1}\Theta \otimes \Theta,\\
	\Omega_5 \wedge \Xi_3 \wedge \alpha_{k+1} \wedge \gamma_{L^-} &+ \Omega_6 \wedge \Xi_3 \wedge \alpha_{L^-} \wedge \gamma_{k+1} &=& 0,\\
	\Omega_9 \wedge \Upsilon_3 \wedge \alpha_{k+1} \wedge \gamma_{L^-} &+
	\Omega_{10} \wedge \Upsilon_3 \wedge \alpha_{L^-} \wedge \gamma_{k+1} &=&\zeta_{\overline{k+1}}^2\nu_{L^-}(-1)^{l+r}\binom{n-2k-2}{r-k}\Theta \otimes \Theta.
\end{align*}
\begin{extracalc}
	By \cite{Kotrbaty2022}*{Prop.~6.4(37)} for $\alpha = r+1$
	\begin{align*}
		\Omega_5 \wedge \Upsilon_3 \wedge \alpha_{k+1} \wedge \gamma_{\overline{k+1}} =& \nu_{k+1} (-1)^{n+1}(\Omega_5 \wedge dz_{k+1}) \wedge (\Upsilon_3 \wedge \gamma_{\overline{k+1}})\\
		=&\nu_{k+1}\zeta_{\overline{k+1}}^2 (-1)^{n+1}(-1)^{n+1}(-1)^{n+1}\\
		&\overline{d\zeta_{\overline{K}}}^{[k]}dz_K^{[k]}\overline{d\eta_{K}}^{[k]}dw_{\overline{K}}^{[k]}\\
		\wedge&(dz_{k+1} \otimes dz_{k+1})(d\zeta_{k+1} \otimes dz_{\overline{k+1}})(d\zeta_{\overline{k+1}} \otimes d\zeta_{k+1})(dz_{\overline{k+1}}\otimes d\zeta_{\overline{k+1}})\\
		\wedge& \overline{d\zeta_{L^-}}^{[n-r-k-2]}\overline{dz_{L^-}}^{[r-k]}
		d\eta_{L^-}^{[r-k]}dw_{L^-}^{[n-r-k-2]}\\
		=&\nu_{k+1}\zeta_{\overline{k+1}}^2 (-1)^{n+1} (\Theta_{K^+} \otimes \Theta_{K^+})\left((-1)^{n+l+r+1}\binom{n-2k-2}{r-k} \Theta_{L^-} \otimes \Theta_{L^-}\right)\\
		=&\nu_{k+1}\zeta_{\overline{k+1}}^2 (-1)^{l+r}\binom{n-2k-2}{r-k} \Theta \otimes \Theta
	\end{align*}
	By \cite{Kotrbaty2022}*{Prop.~5.2(19)} and \cite{Kotrbaty2022}*{Prop.~6.4(37)}
	\begin{align*}
		&\Omega_{11} \wedge \Xi_1 \wedge \alpha_{k+1} \wedge \gamma_{L^-} + \Omega_{12} \wedge \Xi_1 \wedge \alpha_{L^-} \wedge \gamma_{k+1}\\
		&=\zeta_{\overline{k+1}}\overline{d\zeta_{\overline{K}}^{[k]}}dz_K^{[k]}(dz_{k+1}dz_{\overline{k+1}} \otimes dz_{k+1})(d\zeta_{k+1} \otimes dz_{\overline{k+1}})\\
		&\left( (-1)^{n+1}\overline{\zeta_{L^-}}\,\overline{d\zeta_{L^-}}^{[n-r-k-2]}\overline{dz_{L^-}}^{[r-k-1]}\gamma_{L^-}
		+(-1)^{n+1}\overline{\zeta_{L^-}}\,\overline{d\zeta_{L^-}}^{[n-r-k-1]}\overline{dz_{L^-}}^{[r-k-2]}\alpha_{L^-}\right) \wedge \Xi_1\\
		&=(-1)^{n+1}\zeta_{\overline{k+1}}\overline{d\zeta_{\overline{K}}^{[k]}}dz_K^{[k]}(dz_{k+1}dz_{\overline{k+1}} \otimes dz_{k+1})(d\zeta_{k+1} \otimes dz_{\overline{k+1}})\\
		& (-1)^{n+1}\nu_{L^-}\overline{d\zeta_{L^-}}^{[n-r-k-1]}\overline{dz_{L^-}}^{[r-k-1]}
		 \wedge \Xi_1\\
		 &=\zeta_{\overline{k+1}}^2\nu_{L^-}(-1)^{k}(-1)^n\\
		 &\overline{d\zeta_{\overline{K}}^{[k]}}dz_K^{[k]}\overline{d\eta_{K}}^{[k]}dw_{\overline{K}}^{[k]}\\
		 \wedge & (dz_{k+1} \otimes dz_{k+1})(d\zeta_{k+1} \otimes dz_{\overline{k+1}})(d\zeta_{\overline{k+1}} \otimes d\zeta_{k+1})(dz_{\overline{k+1}}\otimes d\zeta_{\overline{k+1}})\\
		 \wedge &\overline{d\zeta_{L^-}}^{[n-r-k-1]}\overline{dz_{L^-}}^{[r-k-1]} d\eta_{L^-}^{[r-k-1]} dw_{L^-}^{[n-r-k-1]}\\
		 &=\zeta_{\overline{k+1}}^2\nu_{L^-}(-1)^{n+k} (\Theta_{K^+} \otimes \Theta_{K^+})(-1)^{n+l+r}\binom{n-2k-2}{r-k-1}\Theta_{L^-} \otimes \Theta_{L^-} \\
		 &=\zeta_{\overline{k+1}}^2\nu_{L^-}(-1)^{k+l+r}\binom{n-2k-2}{r-k-1}\Theta \otimes \Theta
	\end{align*}
	By \cite{Kotrbaty2022}*{Prop.~5.2(18)}
	\begin{align*}
		&\Omega_5 \wedge \Xi_3 \wedge \alpha_{k+1} \wedge \gamma_{L^-} + \Omega_6 \wedge \Xi_3 \wedge \alpha_{L^-} \wedge \gamma_{k+1}\\
		&= \zeta_{\overline{k+1}}\\
		&\overline{d\zeta_{\overline{K}}}^{[k]}dz_K^{[k]}(dz_{k+1} \otimes dz_{k+1})(d\zeta_{k+1} \otimes dz_{\overline{k+1}})\\
		&\left( (-1)^{n+1}\overline{d\zeta_{L^-}}^{[n-r-k-2]}\overline{dz_{L^-}}^{[r-k]}\wedge \gamma_{L^-} + (-1)^{n+1}\overline{d\zeta_{L^-}}^{[n-r-k-1]}\overline{dz_{L^-}}^{[r-k-1]}\wedge \alpha_{L^-}\right)\wedge \Xi_3\\
		&=0 
	\end{align*}
	By \cite{Kotrbaty2022}*{Prop.~5.2(19)} for $r+1$ and \cite{Kotrbaty2022}*{Prop.~6.4(37)} for $r+1$
	\begin{align*}
		&\Omega_9 \wedge \Upsilon_3 \wedge \alpha_{k+1} \wedge \gamma_{L^-} +
		\Omega_{10} \wedge \Upsilon_3 \wedge \alpha_{L^-} \wedge \gamma_{k+1}\\
		&=\zeta_{\overline{k+1}}\overline{d\zeta_{\overline{K}}}^{[k]}dz_K^{[k]}(d\zeta_{\overline{k+1}} \otimes dz_{k+1})(d\zeta_{k+1}dz_{k+1} \otimes dz_{\overline{k+1}})\\
		&\left( (-1)^{n+1}\overline{\zeta_{L^-}}\,\overline{d\zeta_{L^-}}^{[n-r-k-3]}\overline{dz_{L^-}}^{[r-k]} \gamma_{L^-} + (-1)^{n+1}\overline{\zeta_{L^-}}\,\overline{d\zeta_{L^-}}^{[n-r-k-2]}\overline{dz_{L^-}}^{[r-k-1]}\wedge \alpha_{L^-}\right) \wedge \Upsilon_3\\
		&=\zeta_{\overline{k+1}}(-1)^{n+1}\overline{d\zeta_{\overline{K}}}^{[k]}dz_K^{[k]}(d\zeta_{\overline{k+1}} \otimes dz_{k+1})(d\zeta_{k+1}dz_{k+1} \otimes dz_{\overline{k+1}})\\
		&(-1)^{n-1}\nu_{L^-}\overline{d\zeta_{L^-}}^{[n-r-k-2]}\overline{dz_{L^-}}^{[r-k]}  \wedge \Upsilon_3\\
		&=\zeta_{\overline{k+1}}^2\nu_{L^-}(-1)^{n+1}\\
		&\overline{d\zeta_{\overline{K}}}^{[k]}dz_K^{[k]}\overline{d\eta_{K}}^{[k]}dw_{\overline{K}}^{[k]}\\
		\wedge &(d\zeta_{\overline{k+1}} \otimes dz_{k+1})(dz_{k+1} \otimes dz_{\overline{k+1}})(d\zeta_{k+1} \otimes d\zeta_{k+1})(dz_{\overline{k+1}}\otimes d\zeta_{\overline{k+1}})\\
		\wedge&\overline{d\zeta_{L^-}}^{[n-r-k-2]}\overline{dz_{L^-}}^{[r-k]}d\eta_{L^-}^{[r-k]}dw_{L^-}^{[n-r-k-2]}\\
		&= \zeta_{\overline{k+1}}^2\nu_{L^-}(-1)^{n+1} (\Theta_{K^+} \otimes \Theta_{K^+})(-1)^{n+l+r+1}\binom{n-2k-2}{r-k}\Theta_{L^-} \otimes \Theta_{L^-}\\
		&=\zeta_{\overline{k+1}}^2\nu_{L^-}(-1)^{l+r}\binom{n-2k-2}{r-k}\Theta \otimes \Theta
	\end{align*}
\end{extracalc}
In order to finish the proof, it remains to collect the terms. Since
\begin{align*}
	(\Omega_1 + \dots + \Omega_{10}) \wedge (\Xi_1 + \Xi_2)\wedge \alpha \wedge \gamma &= \left(2\nu_K + 2\nu_{k+1}\right) \zeta_{\overline{k+1}}^2 (-1)^{k+l+r}\binom{n-2k-2}{r-k-1}\Theta\otimes \Theta\\
	&=\nu_{K^+ \cup \overline{K^+}} \zeta_{\overline{k+1}}^2 (-1)^{k+l+r}\binom{n-2k-2}{r-k-1}\Theta\otimes \Theta
\end{align*}
and
\begin{align*}
	(\Omega_{11} + \Omega_{12})\wedge (\Xi_1 + \Xi_2)\wedge \alpha \wedge \gamma &=\zeta_{\overline{k+1}}^2\nu_{L^-}(-1)^{k+l+r}\binom{n-2k-2}{r-k-1}\Theta \otimes \Theta,
\end{align*}
we obtain
\begin{align*}
	(\overline{\omega_{r,k}} \otimes \Theta_1)\wedge (\Xi_1 + \Xi_2)\wedge \alpha \wedge \gamma = \nu\zeta_{\overline{k+1}}^2(-1)^{k+l+r}\binom{n-2k-2}{r-k-1}\Theta \otimes \Theta.
\end{align*}
Moreover,
\begin{align*}
	(\Omega_1 + \dots + \Omega_{12}) \wedge \Xi_3 \wedge \alpha \wedge \gamma = 0.
\end{align*}
Next, since
\begin{align*}
	(\Omega_1 + \dots + \Omega_8) \wedge (\Upsilon_3 + \Upsilon_4) \wedge \alpha \wedge \gamma &= (2\nu_K + 2\nu_{k+1})\zeta_{\overline{k+1}}^2 (-1)^{l+r}\binom{n-2k-2}{r-k} \Theta \otimes \Theta\\
	&=\nu_{K^+ \cup \overline{K^+}}\zeta_{\overline{k+1}}^2 (-1)^{l+r}\binom{n-2k-2}{r-k} \Theta \otimes \Theta,
\end{align*}
and
\begin{align*}
	(\Omega_9 + \dots + \Omega_{12}) \wedge (\Upsilon_3 + \Upsilon_4) \wedge \alpha \wedge \gamma &=\nu_{L^-}\zeta_{\overline{k+1}}^2(-1)^{l+r}\binom{n-2k-2}{r-k}\Theta \otimes \Theta,
\end{align*}
we have
\begin{align*}
	(\overline{\omega_{r,k}} \otimes \Theta_1)\wedge (\Upsilon_3 + \Upsilon_4)\wedge \alpha \wedge \gamma =\nu\zeta_{\overline{k+1}}^2(-1)^{l+r}\binom{n-2k-2}{r-k}\Theta \otimes \Theta.
\end{align*}
Therefore we obtain for the double forms
\begin{align*}
	(\overline{\omega_{r,k}} \otimes \Theta_1)\wedge (\sigma_{n-r,k+1} \otimes \Theta_2)\wedge \alpha \wedge \gamma &= \nu\zeta_{\overline{k+1}}^2(-1)^{k+l+r}\binom{n-2k-2}{r-k-1}\Theta \otimes \Theta,\\
	(\overline{\omega_{r,k}} \otimes \Theta_1)\wedge (\tau_{n-r,k+1} \otimes \Theta_2)\wedge \alpha \wedge \gamma &= \nu\zeta_{\overline{k+1}}^2(-1)^{l+r}\left(\binom{n-2k-2}{r-k-1} + \binom{n-2k-2}{r-k}\right)\Theta \otimes \Theta\\
	&=\nu\zeta_{\overline{k+1}}^2(-1)^{l+r}\binom{n-2k-1}{r-k}\Theta \otimes \Theta,
\end{align*}
which implies
\begin{align}
	\overline{\omega_{r,k}}\wedge \sigma_{n-r,k+1}\wedge \alpha \wedge \gamma &= \nu\zeta_{\overline{k+1}}^2(-1)^{k+l+r}\binom{n-2k-2}{r-k-1}\Theta,\label{eq:pairPrfSigm}\\
	\overline{\omega_{r,k}}\wedge \tau_{n-r,k+1}\wedge \alpha \wedge \gamma &= \nu\zeta_{\overline{k+1}}^2(-1)^{l+r}\binom{n-2k-1}{r-k}\Theta. \label{eq:pairPrfTau}
\end{align}
Next, since by Eq.~\eqref{eq:DOmrkm}
\begin{align*}
	D\omega_{n-r,k+1,m} = c_{n-r,m}\zeta_{\overline{1}}^{m-2}\left((m+k)\sigma_{n-r,k+1} + (-1)^k(r-k)\tau_{n-r,k+1} \right) \alpha,
\end{align*}
where $c_{n-r,m} = (-1)^{n+1}(r+m-2)$, and $\overline{\omega_{r,k,m}} = \zeta_1^{m-2}\overline{\omega_{r,k}}$, we obtain from Eq.~\eqref{eq:pairPrfSigm} and Eq.~\eqref{eq:pairPrfTau}
\begin{align*}
	\overline{\omega_{r,k,m}} \wedge D\omega_{n-r,k+1,m}  \wedge \gamma =& 
	c_{n-r,m}|\zeta_{\overline{1}}|^{2(m-2)} \nu \zeta_{\overline{k+1}}^2 (-1)^{k+l+r}\\
	& \left( (m+k)\binom{n-2k-2}{r-k-1} + (r-k) \binom{n-2k-1}{r-k} \right) \Theta\\
	=&c_{n-r,m}|\zeta_{\overline{1}}|^{2(m-2)} \nu \zeta_{\overline{k+1}}^2 (-1)^{k+l+r}\\
	& \left( (m+k)\binom{n-2k-2}{r-k-1} + (n-2k-1) \binom{n-2k-2}{r-k-1} \right) \Theta\\
	=&c_{n-r,m}|\zeta_{\overline{1}}|^{2(m-2)} \nu \zeta_{\overline{k+1}}^2 (-1)^{k+l+r}(n+m-k-1)\binom{n-2k-2}{r-k-1} \Theta.
\end{align*}
Using $\Theta = (-1)^{n+l+1} \vol_{S\RR^n}\gamma$ and \cite{Kotrbaty2022}*{Lem.~4.1}, the claim follows.
\end{proof}

Next, we are going to calculate the integral in the pairing from \autoref{sec:pairing}. Recall that for $a,b \in \RR$,  
\begin{align}\label{eq:intCosaSinb}
	\int_0^{\frac{\pi}{2}} \cos(t)^a \sin(t)^b dt = \frac{\Gamma(\frac{a+1}{2})\Gamma(\frac{b+1}{2})}{2\Gamma(\frac{a+b}{2}+1)} = \frac{s_{a+b+1}}{s_as_b},
\end{align}
where $s_k = \frac{2\pi^{\frac{k+1}{2}}}{\Gamma(\frac{k+1}{2})}$ is the volume of the $k$-dimensional unit sphere. Note that
%\begin{align*}
%	s_k = (k+1)v_{k+1} = \frac{(k+1)\pi^{\frac{k+1}{2}}}{\Gamma(\frac{k+1}{2}+1)} = \frac{(k+1)2\pi^{\frac{k+1}{2}}}{\Gamma(\frac{k+1}{2})(k+1)} = \frac{2\pi^{\frac{k+1}{2}}}{\Gamma(\frac{k+1}{2})}
%\end{align*}
\begin{align}
	\label{eq:recursion_sn}
	(n-1)s_n = 2\pi s_{n-2} = s_1s_{n-2}.
\end{align}
Using the substitution $u = (u_1 \cos(\phi), u_2 \sin(\phi)) \in \S^{n-1}$, where $u_1 \in \S^{i-1}(E)$, $u_2 \in \S^{n-i-1}(E^\perp)$ and  $0 \leq \phi \leq \frac{\pi}{2}$, and $E \subset \RR^n$ is an $i$-dimensional subspace (see, e.g., \cite{Grinberg1999}*{Sec.~6}), several times, as well as Eq.~\eqref{eq:intCosaSinb}, one can further show the following lemma.
\begin{lemma}\label{lem:spherIntZetas}
	Let $a,b,c \in \RR$, then
	\begin{align*}
		\int_{\S^{n-1}} |\zeta_1|^{2a} |\zeta_{k+1}|^{2b} \nu_{K^+}^c du= 2^{-(a+b+c)}\frac{s_{n + 2(a+b+c)-1}s_{2(k+a+b)+1}}{s_{2(k+a+b+c)+1}} \frac{s_1^2}{s_{2a+1}s_{2b+1}}.
	\end{align*}
\end{lemma}
\begin{extracalc}
	\begin{proof}
		We will use the following substitution for an $i$-dimensional subspace $E \subset \RR^n$ (see, e.g., \cite{Grinberg1999}):
		\begin{align*}
			u &= (u_1 \cos(\phi), u_2 \sin(\phi)), \quad u_1 \in \S^{i-1}(E), u_2 \in \S^{n-i-1}(E^\perp), 0 \leq \phi \leq \frac{\pi}{2}\\
			du &= \cos(\phi)^{i-1}\sin(\phi)^{n-i-1} d\phi du_1 du_2.
		\end{align*}
		
		Using this, we obtain for $E = \mathrm{span}\{\xi_1, \dots, \xi_{2(k+1)}\}$
		\begin{align*}
			&\int_{\S^{n-1}} |\zeta_1|^{2a} |\zeta_{k+1}|^{2b} \nu_{K^+}^c du \\
			&= \int_{\S^{n-2k-3}}\int_{\S^{2k+1}}\int_{0}^{\frac{\pi}{2}} (|\zeta_1|^{2a} |\zeta_{k+1}|^{2b} \nu_{K^+}^c)(u_1 \cos(\phi), u_2 \sin(\phi))\cos(\phi)^{2k+1}\sin(\phi)^{n-2k-3} d\phi du_1 du_2\\
			&= 2^{-c} s_{n-2k-3}\left(\int_{\S^{2k+1}}|\zeta_1|^{2a} |\zeta_{k+1}|^{2b}du_1 \right) \left(\int_{0}^{\frac{\pi}{2}} \cos(\phi)^{2(k+a+b+c)+1}\sin(\phi)^{n-2k-3} d\phi\right)\\
			&= 2^{-c} s_{n-2k-3}\left(\int_{\S^{2k+1}}|\zeta_1|^{2a} |\zeta_{k+1}|^{2b}du_1 \right) \frac{s_{n + 2(a+b+c)-1}}{s_{2(k+a+b+c)+1}s_{n-2k-3}}\\
		\end{align*}
		
		Using again the substitution for $E = \mathrm{span}\{\xi_1, \xi_2\}$,
		\begin{align*}
			&\int_{\S^{2k+1}}|\zeta_1|^{2a} |\zeta_{k+1}|^{2b}du \\
			&= \int_{\S^{2k-1}}\int_{\S^1} \int_{0}^{\frac{\pi}{2}} (|\zeta_1|^{2a} |\zeta_{k+1}|^{2b})(u_1 \cos(\phi), u_2 \sin(\phi)) \cos(\phi) \sin(\phi)^{2k-1} d\phi du_1 du_2\\
			&= \left(\int_{\S^1}|\zeta_1|^{2a} du_1 \right) \left(\int_{\S^{2k-1}} |\zeta_{k+1}|^{2b} du_2\right)\left(\int_{0}^{\frac{\pi}{2}}  \cos(\phi)^{2a+1} \sin(\phi)^{2(k+b)-1} d\phi\right)\\
			&= \left(\int_{\S^1}|\zeta_1|^{2a} du_1 \right) \left(\int_{\S^{2k-1}} |\zeta_{k+1}|^{2b} du_2\right)\frac{s_{2(k+a+b)+1}}{s_{2a+1}s_{2(k+b)-1}}.
		\end{align*}
		
		Finally,
		\begin{align*}
			\int_{\S^1}|\zeta_1|^{2a} du = \frac{s_1}{2^a},
		\end{align*}
		and for $2k-1 \geq 3$, using again the substitution
		\begin{align*}
			\int_{\S^{2k-1}} |\zeta_{k+1}|^{2b} du &=  \int_{\S^{2k-3}}\int_{\S^1}\int_0^{\frac{\pi}{2}}|\zeta_{k+1}|^{2b}(u_1 \cos(\phi), u_2 \sin(\phi)) \cos(\phi) \sin(\phi)^{2k-3} d\phi du_1 du_2\\
			&=\frac{s_1}{2^b} s_{2k-3} \int_0^{\frac{\pi}{2}}\cos(\phi)^{2b+1} \sin(\phi)^{2k-3} d\phi\\
			&=\frac{s_1}{2^b} s_{2k-3} \frac{s_{2(b+k)-1}}{s_{2b+1}s_{2k-3}}.
		\end{align*}
		Hence, in total,
		\begin{align*}
			&\int_{\S^{n-1}} |\zeta_1|^{2a} |\zeta_{k+1}|^{2b} \nu_{K^+}^c du \\
			&= \frac{s_{n-2k-3}}{2^c}\frac{s_1}{2^a}\frac{s_1}{2^b} s_{2k-3} \frac{s_{2(b+k)-1}}{s_{2b+1}s_{2k-3}}\frac{s_{2(k+a+b)+1}}{s_{2a+1}s_{2(k+b)-1}} \frac{s_{n + 2(a+b+c)-1}}{s_{2(k+a+b+c)+1}s_{n-2k-3}}\\
			&= 2^{-(a+b+c)} \frac{s_{n + 2(a+b+c)-1}s_{2(k+a+b)+1}}{s_{2(k+a+b+c)+1}} \frac{s_1^2}{s_{2a+1}s_{2b+1}}.
		\end{align*}
	\end{proof}
	\phantom{a}
\end{extracalc}

Moreover, we need the following relation.
\begin{lemma}[\cite{Kotrbaty2022}*{Lem.~6.7}]\label{lem:pairKWForms}
	Let $m, r, k \in \NN$ such that $1\leq r \leq n-1$, $1\leq k \leq \min\{r, n-r\}$ and $m \geq 2$. Then 
	\begin{align}
		\overline{\omega_{r,k,m}} \wedge D\omega_{n-r,k,m} =& (-1)^{r+k} |\zeta_1|^{2(m-2)}(m+r-2)\binom{n-2k}{r-k}\\
		&\times \left((m+r)\nu_K^2 + (m+k-1)\frac{n-r-k}{n-2k}\nu_L\right) \vol_{S\RR^n}. \nonumber
	\end{align}
\end{lemma}

We now combine \autoref{lem:pairOmkmOmkm1m}, \autoref{lem:spherIntZetas} and \autoref{lem:pairKWForms}.
\begin{corollary}\label{cor:pairKp1NotZero}
	Let $m\ge 2$, $r \leq n-1$, and $1\leq k<\min\{r,n-r\}$. If $\vol$ is the euclidean volume form on $\RR^n$, then
	\begin{align}
		\label{eq:ValuePairing}
		\begin{split}
			&\int_{\{0\} \times \S^{n-1}} \pair{\overline{\LieDer_{\widetilde{Y}_{k+1,k+1}} \omega_{r,k,m}} \wedge D\omega_{n-r,k+1,m}}{\vol}\\
			&= (m+r-2)(-1)^{k+r}2^{-m+2}\frac{s_{n + 2m+1}}{s_1s_{2m-1}}\binom{n-2k-2}{r-k-1} C(n,m,r,k)
		\end{split}
	\end{align}
	where 
	\begin{align*}
		C(n,m,r,k)=&-(m+n-r-2)(m+k-1)(m+n-k-1)
	\end{align*}
	is strictly negative.
\end{corollary}
\begin{proof}
	By \autoref{cor:YkkOnOmegarkm},
	\begin{align*}
		\LieDer_{\widetilde{Y}_{k+1,k+1}} \omega_{r,k,m} = (n-r+m-2)\zeta_{\overline{k+1}}^2 \omega_{r,k,m} + 2\omega_{r,k+1,m},
	\end{align*}
	Hence, by \autoref{lem:pairKWForms},
	\begin{align*}
		&\overline{\LieDer_{\widetilde{Y}_{k+1,k+1}} \omega_{r,k,m}} \wedge D\omega_{n-r,k+1,m} \\
		&=(n-r+m-2)\zeta_{k+1}^{2} \overline{\omega_{r,k,m}} \wedge D\omega_{n-r,k+1,m} + 2 \overline{\omega_{r,k+1,m}}\wedge D\omega_{n-r,k+1,m}\\
	&=(m+r-2)(-1)^{k+r}|\zeta_{1}|^{2(m-2)}\binom{n-2k-2}{r-k-1}\vol_{S\RR^n} \\
	&\left((n-r+m-2)(n+m-k-1)\nu \nu_{k+1}^2-2\nu_{K^+}^2(m+r) - 2\nu_{L^-}(m+k)\frac{n-k-r-1}{n-2k-2}\right),
	\end{align*}
	where $\nu=1$ on $S\RR^n$. Moreover, note that $\pair{\vol_{S\RR^n}}{\vol} = \vol_{\S^{n-1}}$ is the standard volume form on $\S^{n-1}$. We thus need to calculate the following integrals:\\
	 First, by \autoref{lem:spherIntZetas} with $a=m-2$, $b=2$ and $c=0$,
	\begin{align*}
		\int_{\S^{n-1}}|\zeta_{1}|^{2(m-2)}|\zeta_{k+1}|^4 du=2^{-m} \frac{s_{n + 2m-1}s_1^2}{s_{2m-3}s_{5}}=2^{-m+3} \frac{s_{n + 2m-1}}{s_{2m-3}s_{1}},
	\end{align*}
	where we used that $8s_5=2s_3s_1=s_1^3$, due to Eq.~\eqref{eq:recursion_sn}. Second, again by \autoref{lem:spherIntZetas} with $a=m-2$, $b=0$ and $c=2$,
	\begin{align*}
		\int_{\S^{n-1}}|\zeta_{1}|^{2(m-2)}\nu_{K^+}^2 du
		&=2^{-m}\frac{s_{n + 2m-1}s_{2(k+m)-3}}{s_{2(k+m)+1}} \frac{s_1}{s_{2m-3}}\\	
		&=2^{-m+2}\frac{s_{n + 2m-1}}{s_{2m-3}s_1}(k+m)(k+m-1),
	\end{align*}
	where we used that, by Eq.~\eqref{eq:recursion_sn},
	\begin{align*}
		s_{2(k+m)-3}s_1=(2(k+m)-2)s_{2(k+m)-1}=2(k+m-1)(2(k+m))\frac{s_{2(k+m)+1}}{s_1}.
	\end{align*}
	Finally, by \autoref{lem:spherIntZetas} with $a=m-2$, $b=0$ and $c=0$ resp. $c=1$, and Eq.~\eqref{eq:recursion_sn},
	\begin{align*}
		\int_{\S^{n-1}} |\zeta_{1}|^{2(m-1)}\nu_{L^-} du&=\int_{\S^{n-1}}|\zeta_{1}|^{2(m-2)}(1-2\nu_{K^+}) du\\
		&= \frac{2^{-(m-2)}}{s_{2m-3}s_1}\left(s_{n + 2m-5}s_1^2 - \frac{s_{n + 2m-3}s_{2(k+m)-3}}{s_{2(k+m)-1}} s_1^2\right)\\
		&=2^{-m+2}\frac{s_{n + 2m-1}}{s_{2m-3}s_1}(n+2m-2)(n-2k-2).
	\end{align*}
%	where we used that $s_{2(k+m)+1}2(k+m)=s_{2(k+m)-1}s_1$.
	Collecting all terms, we therefore obtain
	\begin{align*}
		&\int_{\{0\} \times \S^{n-1}} \pair{\overline{\LieDer_{\widetilde{Y}_{k+1,k+1}} \omega_{r,k,m}} \wedge D\omega_{n-r,k+1,m}}{\vol}\\
		=&(m+r-2)(-1)^{k+r}2^{-m+3}\frac{s_{n + 2m-1}}{s_1s_{2m-3}}\binom{n-2k-2}{r-k-1}\\
		&\times \left((n-r+m-2)(n+m-k-1)-(m+r)(k+m)(k+m-1)\right.\\
		&\left.- (n+2m-2)(m+k)(n-k-r-1)\right).
	\end{align*}
	A direct computation shows that
	\begin{align*}
		&(n-r+m-2)(n+m-k-1)-(m+r)(k+m)(k+m-1)\\
		&- (n+2m-2)(m+k)(n-k-r-1)\\
		=&-(m+n-r-2)(m+k-1)(m+n-k-1)=C(n,m,r,k).
	\end{align*}
	This completes the proof of Eq. \eqref{eq:ValuePairing}. In order to see that $C(n,m,r,k)$ is strictly negative, note that each factor in $C(n,m,r,k)$ is strictly positive since $1\leq k<\min\{r,n-r\}$, $1\leq r\leq n-1$, and $m \geq 2$.
\begin{extracalc}
	\begin{align*}
		&\overline{\LieDer_{\widetilde{Y}_{k+1,k+1}} \omega_{r,k,m}} \wedge D\omega_{n-r,k+1,m} \\
		&=(n-r+m-2)\zeta_{k+1}^{2} \overline{\omega_{r,k,m}} \wedge D\omega_{n-r,k+1,m} \\
		&+ 2 \overline{\omega_{r,k+1,m}}\wedge D\omega_{n-r,k+1,m}\\
		&=(r+m-2)(n-r+m-2)(-1)^{k+r}(n+m-k-1)\binom{n-2k-2}{r-k-1} \\
		& |\zeta_{\overline{1}}|^{2(m-2)} \nu \nu_{k+1}^2\vol_{S\RR^n}\\
		&+ 2(-1)^{k+r+1}|\zeta_{1}|^{2(m-2)}(m+r-2)\\
		&\left(\nu_{K^+}^2(m+r)\binom{n-2k-2}{r-k-1} + \nu_{L^-}(m+k)\binom{n-2k-3}{r-k-1}\right)\vol_{S\RR^n}\\
		&=(m+r-2)(-1)^{k+r}|\zeta_{1}|^{2(m-2)}\vol_{S\RR^n} \\
		&\left((n-r+m-2)(n+m-k-1)\nu \nu_{k+1}^2\binom{n-2k-2}{r-k-1}\right.\\
		&\left.-2\nu_{K^+}^2(m+r)\binom{n-2k-2}{r-k-1} - 2\nu_{L^-}(m+k)\binom{n-2k-3}{r-k-1}\right)\\
		&=(m+r-2)(-1)^{k+r}|\zeta_{1}|^{2(m-2)}\vol_{S\RR^n} \\
		&\left((n-r+m-2)(n+m-k-1)\nu \nu_{k+1}^2\binom{n-2k-2}{r-k-1}\right.\\
		&\left.-2\nu_{K^+}^2(m+r)\binom{n-2k-2}{r-k-1} - 2\nu_{L^-}(m+k)\frac{n-2k-2-(r-k-1)}{n-2k-2}\binom{n-2k-2}{r-k-1}\right)\\
		&=(m+r-2)(-1)^{k+r}|\zeta_{1}|^{2(m-2)}\binom{n-2k-2}{r-k-1}\vol_{S\RR^n} \\
		&\left((n-r+m-2)(n+m-k-1)\nu \nu_{k+1}^2-2\nu_{K^+}^2(m+r) - 2\nu_{L^-}(m+k)\frac{n-k-r-1}{n-2k-2}\right)\\
	\end{align*}	
	
	By \autoref{lem:spherIntZetas} with $a=m-2$, $b=2$ and $c=0$,
	\begin{align*}
		\int_{\{0\}\times \S^{n-1}} |\zeta_{1}|^{2(m-2)}\nu \nu_{k+1}^2 \pair{\vol_{S\RR^n}}{\vol} &= \int_{\S^{n-1}}|\zeta_{1}|^{2(m-2)}|\zeta_{k+1}|^4 du\\
		&=2^{-m}\frac{s_{n + 2m-1}s_{2(k+m)+1}}{s_{2(k+m)+1}} \frac{s_1^2}{s_{2(m-2)+1}s_{4+1}}\\
		&=2^{-m} \frac{s_{n + 2m-1}s_1^2}{s_{2m-3}s_{5}}\\
		&=2^{-m} \frac{s_{n + 2m-1}2s_3}{s_{2m-3}s_{5}}\\
		&=2^{-m} \frac{s_{n + 2m-1}8}{s_{2m-3}s_{1}} = 2^{-m+3} \frac{s_{n + 2m-1}}{s_{2m-3}s_{1}}
	\end{align*}
	By \autoref{lem:spherIntZetas} with $a=m-2$, $b=0$ and $c=2$,
	\begin{align*}
		\int_{\{0\}\times \S^{n-1}} |\zeta_{1}|^{2(m-2)}\nu_{K^+}^2 \pair{\vol_{S\RR^n}}{\vol} &= \int_{\S^{n-1}}|\zeta_{1}|^{2(m-2)}\nu_{K^+}^2 du\\
		&=2^{-m}\frac{s_{n + 2m-1}s_{2(k+m-2)+1}}{s_{2(k+m)+1}} \frac{s_1^2}{s_{2(m-2)+1}s_{1}}\\
		&=2^{-m}\frac{s_{n + 2m-1}s_{2(k+m)-3}}{s_{2(k+m)+1}} \frac{s_1}{s_{2m-3}}\\
		&=2^{-m}\frac{s_{n + 2m-1}s_{2(k+m)-1}(2(k+m)-2)}{s_{2(k+m)+1}s_{2m-3}}\\	
		&=2^{-m}\frac{s_{n + 2m-1}s_{2(k+m)+1}2(k+m)}{s_{2(k+m)+1}s_{2m-3}s_1}(2(k+m)-2)\\	
		&=2^{-m+2}\frac{s_{n + 2m-1}}{s_{2m-3}s_1}(k+m)((k+m)-1)%\\
%		&=2^{-m+2}\frac{s_{n + 2m-1}}{s_{2m-1}(2m-2)}(k+m)((k+m)-1)\\
%		&=2^{-m+1}\frac{s_{n + 2m-1}}{s_{2m-1}}\frac{(k+m)((k+m)-1)}{m-1}\\
	\end{align*}
	By \autoref{lem:spherIntZetas} with $a=m-2$, $b=0$ and $c=0$ resp. $c=1$,
	\begin{align*}
		&\int_{\{0\}\times \S^{n-1}} |\zeta_{1}|^{2(m-2)}\nu_{L^-} \pair{\vol_{S\RR^n}}{\vol}\\
		 &= \int_{\S^{n-1}}|\zeta_{1}|^{2(m-2)}(1-2\nu_{K^+}) du\\
		&=2^{-(m-2)}\frac{s_{n + 2(m-2)-1}s_{2(k+m-2)+1}}{s_{2(k+m-2)+1}} \frac{s_1^2}{s_{2(m-2)+1}s_{1}}\\
		&-2^{-(m-2)}\frac{s_{n + 2(m-1)-1}s_{2(k+m-2)+1}}{s_{2(k+m-1)+1}} \frac{s_1^2}{s_{2(m-2)+1}s_{1}}\\
		&= \frac{2^{-(m-2)}}{s_{2m-3}s_1}\left(s_{n + 2m-5}s_1^2 - \frac{s_{n + 2m-3}s_{2(k+m)-3}}{s_{2(k+m)-1}} s_1^2\right)\\
		&= \frac{2^{-(m-2)}}{s_{2m-3}s_1}\left(s_{n + 2m-1}(n+2m-4)(n+2m-2) - \frac{s_{n + 2m-1}(n+2m-2)s_{2(k+m)-1}(2(k+m)-2)}{s_{2(k+m)-1}} \right)\\
		&=\frac{2^{-(m-2)}}{s_{2m-3}s_1}\left(s_{n + 2m-1}(n+2m-4)(n+2m-2) - s_{n + 2m-1}(n+2m-2)(2(k+m)-2) \right)\\
		&=2^{-(m-2)}\frac{s_{n + 2m-1}}{s_{2m-3}s_1}(n+2m-2) \left((n+2m-4)- (2(k+m)-2) \right)\\
		&=2^{-(m-2)}\frac{s_{n + 2m-1}}{s_{2m-3}s_1}(n+2m-2)(n-2k-2)
	\end{align*}
	
%	\vspace*{2cm}
	
	\begin{align*}
		&\int_{\{0\} \times \S^{n-1}} \pair{\overline{\LieDer_{\widetilde{Y}_{k+1,k+1}} \omega_{r,k,m}} \wedge D\omega_{n-r,k+1,m}}{\vol} = (m+r-2)(-1)^{k+r}2^{-m+3}\frac{s_{n + 2m-1}}{s_1s_{2m-3}}\binom{n-2k-2}{r-k-1}\\
		&\times \left((n-r+m-2)(n+m-k-1)-(m+r)(k+m)(k+m-1)\right.\\
		&\left.- (n+2m-2) (n-2k-2)(m+k)\frac{n-k-r-1}{n-2k-2}\right)	\\
		&=(m+r-2)(-1)^{k+r}2^{-m+3}\frac{s_{n + 2m-1}}{s_1s_{2m-3}}\binom{n-2k-2}{r-k-1}\\
		&\times \left((n-r+m-2)(n+m-k-1)-(m+r)(k+m)(k+m-1)\right.\\
		&\left.- (n+2m-2)(m+k)(n-k-r-1)\right)
	\end{align*}
		We have
	\begin{align*}
		&(n-r+m-2)(n+m-k-1)\\
		&-(m+r)(m+k)(k+m-1)\\
		&-(n+2m-2)(m+k)(n-k-r-1)\\
		=&\textcolor{red}{(n-r+m-2)(n-r-k-1)}+(n-r+m-2)(m+r)\\
		&\textcolor{blue}{-(m+r)(m+k)(k+m-1)}\\
		&\textcolor{red}{-(n+m-r-2)(m+k)(n-k-r-1)}	\textcolor{blue}{-(m+r)(m+k)(n-k-r-1)}\\
		=&\textcolor{red}{-(n+m-r-2)(m+k-1)(n-k-r-1)}	+(n-r+m-2)(m+r)\\
		&\textcolor{blue}{-(m+r)(m+k)(k+m-1+n-k-r-1)}\\
		=&-(n+m-r-2)(m+k-1)(n-k-r-1)	+\textcolor{green}{(n-r+m-2)(m+r)}\\
		&\textcolor{green}{-(m+r)(m+k)(m+n-r-2)}\\
		=&-(n+m-r-2)(m+k-1)(n-k-r-1)\\
		&\textcolor{green}{-(m+r)(m+k-1)(m+n-r-2)}\\
		=&-(m+n-r-2)(m+k-1)(m+n-k-1)
	\end{align*}
	Since $k<\min (r,n-r)$, $m\ge 2$, this is strictly negative.
\end{extracalc}
\end{proof}

\subsection{Calculations for $\lambda_{-l,m}$}
Suppose that $n=2l$ is even. In this section we are considering the forms
\begin{align*}
	\omega_{l,-l} \otimes \Theta_1 = \zeta_{\overline{M}} d\zeta_{\overline{M}}^{[l-1]}\overline{dz_M}^{[l]}, \qquad \omega_{l,-l, m} = \zeta_{\overline{1}}^{m-2}\omega_{l,-l}\in \Omega^{l,l-1}(S\RR^{n})^{\mathrm{\mathrm{tr}}},
\end{align*}
where $M = \{1, \dots, l-1, \overline{l}\}$.

%Write $N = \{1, \dots, l-1\}$

\begin{lemma}\label{lem:RonOmegarkm}
	Suppose that $n=2l$ is even and denote by $R \in \OO(n)$ the reflection along $e_n^\perp$. Then
	\begin{align*}
		G_R^\ast \omega_{l,-l,m} = -\omega_{l,l,m}, \quad \text{ and } \quad 
		G_R^\ast \omega_{r,k,m} = -\omega_{r,k,m}
	\end{align*}
	for all $k<r$, where $G_R(x,v) = \left(Rx, \frac{R^{-T}v}{\|R^{-T}v\|}\right)$, $(x,v) \in S\RR^n$.
\end{lemma}
\begin{proof}
	Note that $G_R^\ast \zeta_{l} = \zeta_{\overline{l}}$ and $G_R^\ast \zeta_j = \zeta_j$, for  $j\in \Iind\setminus\{l,\bar{l}\}$, and similarly for $z_l$ and $z_j$. Hence, $G_R^\ast (\omega_{r,k,m} \otimes \Theta_1) = \omega_{r,k,m} \otimes \Theta_1$ for all $k < r$, and $G_R^\ast (\omega_{l,-l,m} \otimes \Theta_1) = \omega_{l,l,m} \otimes \Theta_1$. Moreover, since $G_R^\ast$ interchanges the positions of $l$ and $\overline{l}$, we have $G_R^\ast \Theta_1 = -\Theta_1$. As $G_R^\ast (\tau \otimes \Theta_1) = (G_R^\ast \tau) \otimes (G_R^\ast \Theta_1)$ for any form $\tau$, this yields the claim.
	\begin{extracalc}
	$K = \{1, \dots, l\}$.
	\begin{align*}
		\omega_{l,-l} \otimes \Theta_1 = \zeta_{\overline{K}} d\zeta_{\overline{K}}^{[l-1]}\overline{dz_K}^{[l]}
	\end{align*}
	\begin{align*}
		\omega_{l,-l} \otimes \Theta_1 = \zeta_{\overline{M}} d\zeta_{\overline{M}}^{[l-1]}\overline{dz_M}^{[l]}
	\end{align*}
\end{extracalc}
\end{proof}
\begin{remark}
	Let us note that while the differential forms $\omega_{r,k,m}$ change sign under $G_R^\ast$, the double forms $\omega_{r,k,m} \otimes \Theta_1$ and also the corresponding valuations (and, in particular the pairing in \cite{Kotrbaty2022}) do not. This is due to the fact that $G_R$ reverses the orientation of the normal cycle of a convex body (see also the sign in \eqref{eq:actGLOnVal}), which is compensated by the sign change in the differential form.
	
	Let us further point out that there is a factor of $(-1)$ missing in \cite{Kotrbaty2022}*{Cor.~5.6}.
\end{remark}

\begin{lemma}\label{lem:adYll}
	If $n=2l$, then we have $\Ad_R(Y_{l,l}) = \overline{Y_{l,l}}$ for the adjoint representation $\Ad_g(X) = g X g^{-1}$ of $\GL(n,\RR)$ on $\gl(n)_\CC$.
\end{lemma}
\begin{proof}
	A short matrix multiplication shows that $R Y_{l,l} R^{-1} = \overline{Y_{l,l}}$.
\end{proof}

\begin{proposition}\label{prop:pairYllOmlml}
	Suppose that $n=2l$ is even, $m\ge 2$. Then
	\begin{align}
		(\LieDer_{\widetilde{Y}_{l,l}} D\omega_{l,-l,m}, D\omega_{l,l-1,m}) =- (-1)^{n}\overline{(\LieDer_{\widetilde{Y}_{l,l}}D\omega_{l,l-1,m},D\omega_{l,l,m})}.
	\end{align}
\end{proposition}
\begin{proof}
	By \autoref{lem:RonOmegarkm}, and since $D$ intertwines pullbacks by contactomorphisms, 
	\begin{align*}
		-\LieDer_{\widetilde{Y}_{l,l}} D\omega_{l,-l,m} = \LieDer_{\widetilde{Y}_{l,l}} G_{R^{-1}}^\ast D\omega_{l,l,m} = G_{R^{-1}}^\ast G_{R}^\ast\LieDer_{\widetilde{Y}_{l,l}} G_{R^{-1}}^\ast D\omega_{l,l,m} = G_{R^{-1}}^\ast \LieDer_{\widetilde{\Ad_R(Y_{l,l})}} D\omega_{l,l,m}.
	\end{align*}
	Thus \autoref{lem:adYll} shows that
	\begin{align*}
		(\LieDer_{\widetilde{Y}_{l,l}} D\omega_{l,-l,m}, D\omega_{l,l-1,m}) = -(G_{R^{-1}}^\ast \LieDer_{\widetilde{\overline{Y_{l,l}}}} D\omega_{l,l,m}, D\omega_{l,l-1,m}),
	\end{align*}
	so by \autoref{prop:pairingFormsSLinv} and \autoref{lem:RonOmegarkm} and since $R^{-1}=R$,
	\begin{align*}
		(\LieDer_{\widetilde{Y}_{l,l}} D\omega_{l,-l,m}, D\omega_{l,l-1,m}) = -(\LieDer_{\widetilde{\overline{Y_{l,l}}}} D\omega_{l,l,m}, G_{R}^\ast D\omega_{l,l-1,m}) = (\LieDer_{\widetilde{\overline{Y_{l,l}}}} D\omega_{l,l,m},  D\omega_{l,l-1,m}).
	\end{align*}
	It therefore remains to apply \autoref{cor:pairFormsSLnInvLiealg} and \autoref{lem:pairFormWellDef} to obtain the claim,
	\begin{align*}
		(\LieDer_{\widetilde{Y}_{l,l}} D\omega_{l,-l,m}, D\omega_{l,l-1,m}) =& -(D\omega_{l,l,m},  \LieDer_{\widetilde{Y}_{l,l}}D\omega_{l,l-1,m}) \\
		=& -(-1)^{n}\overline{( \LieDer_{\widetilde{Y}_{l,l}}D\omega_{l,l-1,m}, D\omega_{l,l,m})}.
	\end{align*}
\end{proof}

\begin{extracalc}

\begin{align}
	\omega_{l,-l} \otimes \Theta_1 =& d\zeta_{\overline{N}}^{[l-1]}\overline{dz_N}^{[l-1]}(\zeta_l \otimes dz_{l})(dz_l \otimes dz_{\overline{l}})\\
	&+\zeta_{\overline{N}}d\zeta_{\overline{N}}^{[l-2]}\overline{dz_N}^{[l-1]}(d\zeta_l \otimes dz_l)(dz_l \otimes dz_{\overline{l}})\\
	=:& \Omega_1 + \Omega_2
\end{align}

\begin{lemma}\label{lem:actYll_simpleDf}
	\begin{align}
		\LieDer_{\widetilde{Y}_{l,l}} \zeta_{\overline{N}} &= \zeta_{\overline{l}}^2 \zeta_{\overline{N}},
		& \LieDer_{\widetilde{Y}_{l,l}} d\zeta_{\overline{N}} &=\zeta_{\overline{l}}^2 d\zeta_{\overline{N}} + 2 \zeta_{\overline{l}}d\zeta_{\overline{l}} \zeta_{\overline{N}},\\
%		\LieDer_{\widetilde{Y}_{l,l}} (\zeta_l \otimes dz_l) &= ,
%		& \LieDer_{\widetilde{Y}_{l,l}} (d\zeta_l \otimes dz_l) &= \\
%		\LieDer_{\widetilde{Y}_{l,l}} (dz_l \otimes dz_{\overline{l}}) &= 
%		& 
		\LieDer_{\widetilde{Y}_{l,l}} \overline{dz_N}^{[l-1]} &= 0.
	\end{align}
\end{lemma}
\begin{proof}
	By \autoref{lem:lieDerYabzZeta}
\end{proof}

\begin{align*}
	\LieDer_{\widetilde{Y}_{l,l}} (\zeta_l \otimes dz_l)(dz_l \otimes dz_{\overline{l}}) =& \left((-\zeta_{\overline{l}} + \zeta_{\overline{l}}^2 \zeta_l) \otimes dz_l + \zeta_l \otimes dz_{\overline{l}}\right)(dz_l \otimes dz_{\overline{l}})\\
	&+(\zeta_l \otimes dz_l)(dz_{\overline{l}} \otimes dz_{\overline{l}})\\
	=&\left((-\zeta_{\overline{l}} + \zeta_{\overline{l}}^2 \zeta_l) \otimes dz_l\right) (dz_l \otimes dz_{\overline{l}})
	+(\zeta_l \otimes dz_l)(dz_{\overline{l}} \otimes dz_{\overline{l}})\\
	=&\zeta_{\overline{l}}^2(\zeta_l \otimes dz_l)(dz_l \otimes dz_{\overline{l}})\\
	&-(\zeta_{\overline{l}} \otimes dz_l) (dz_l \otimes dz_{\overline{l}})
	+(\zeta_l \otimes dz_l)(dz_{\overline{l}} \otimes dz_{\overline{l}})
\end{align*}
\begin{align*}
	\LieDer_{\widetilde{Y}_{l,l}} (d\zeta_l \otimes dz_l)(dz_l \otimes dz_{\overline{l}}) =&\zeta_{\overline{l}}^2(d\zeta_l \otimes dz_l)(dz_l \otimes dz_{\overline{l}})\\
	&+(2\nu_{l} -1)(d\zeta_{\overline{l}} \otimes dz_l) (dz_l \otimes dz_{\overline{l}})
	+(d\zeta_l \otimes dz_l)(dz_{\overline{l}} \otimes dz_{\overline{l}})
\end{align*}

\begin{align*}
	\LieDer_{\widetilde{Y}_{l,l}} d\zeta_{\overline{N}}^{[l-1]} =& d\zeta_{\overline{N}}^{[l-2]} (\zeta_{\overline{l}}^2 d\zeta_{\overline{N}} + 2 \zeta_{\overline{l}}d\zeta_{\overline{l}} \zeta_{\overline{N}})\\
	=& \zeta_{\overline{l}}^2(l-1) d\zeta_{\overline{N}}^{[l-1]} + 2\zeta_{\overline{l}}d\zeta_{\overline{l}} \zeta_{\overline{N}}d\zeta_{\overline{N}}^{[l-2]}
\end{align*}
\begin{align*}
	\LieDer_{\widetilde{Y}_{l,l}} \zeta_{\overline{N}}d\zeta_{\overline{N}}^{[l-2]} =& \zeta_{\overline{l}}^2 \zeta_{\overline{N}}d\zeta_{\overline{N}}^{[l-2]} + \zeta_{\overline{N}}d\zeta_{\overline{N}}^{[l-3]}\left(\zeta_{\overline{l}}^2 d\zeta_{\overline{N}} + 2 \zeta_{\overline{l}}d\zeta_{\overline{l}} \zeta_{\overline{N}} \right)\\
	=&(1+l-2)\zeta_{\overline{l}}^2 \zeta_{\overline{N}}d\zeta_{\overline{N}}^{[l-2]}
\end{align*}

\begin{table}[h]
	\begin{center}
		\begin{tabular}{|c|c|c|}
			\hline
			$\tau$ & $\lambda$ & $\LieDer_{\widetilde{Y}_{l,l}}\tau -\lambda\zeta_{\overline{l}}^2 \tau$\\% & yields terms\\
			\hline\hline
			$\zeta_{\overline{N}}d\zeta_{\overline{N}}^{[l-2]}$  & $l-1$ & -- \\%& --\\
			\hline
			$d\zeta_{\overline{N}}^{[l-1]}$ & $l-1$ & $2\zeta_{\overline{l}}d\zeta_{\overline{l}} \zeta_{\overline{N}}d\zeta_{\overline{N}}^{[l-2]}$ \\%& n.a.\\
			\hline\hline
			$(\zeta_l \otimes dz_l)(dz_l \otimes dz_{\overline{l}})$ & $1$ & $-(\zeta_{\overline{l}} \otimes dz_l) (dz_l \otimes dz_{\overline{l}})
			+(\zeta_l \otimes dz_l)(dz_{\overline{l}} \otimes dz_{\overline{l}})$ \\%& --\\
			\hline
			$(d\zeta_l \otimes dz_l)(dz_l \otimes dz_{\overline{l}})$ & $1$ & $(2\nu_{l} -1)(d\zeta_{\overline{l}} \otimes dz_l) (dz_l \otimes dz_{\overline{l}})
			+(d\zeta_l \otimes dz_l)(dz_{\overline{l}} \otimes dz_{\overline{l}})$ \\%& $\Omega_{k+1, 2}$\\
			\hline
		\end{tabular}
	\end{center}
\end{table}

\begin{align*}
	\LieDer_{\widetilde{Y}_{l,l}} \Omega_1 =& (l-1+1) \zeta_{\overline{l}}^2\Omega_1\\
	&+2\zeta_{\overline{l}}d\zeta_{\overline{l}} \zeta_{\overline{N}}d\zeta_{\overline{N}}^{[l-2]}\overline{dz_N}^{[l-1]}(\zeta_l \otimes dz_{l})(dz_l \otimes dz_{\overline{l}})\\
	&+d\zeta_{\overline{N}}^{[l-1]}\overline{dz_N}^{[l-1]}\left(-(\zeta_{\overline{l}} \otimes dz_l) (dz_l \otimes dz_{\overline{l}})
	+(\zeta_l \otimes dz_l)(dz_{\overline{l}} \otimes dz_{\overline{l}})\right)\\
	=&l\zeta_{\overline{l}}^2\Omega_1 \\
	&+2\nu_{l} \zeta_{\overline{N}}d\zeta_{\overline{N}}^{[l-2]}\overline{dz_N}^{[l-1]}(d\zeta_{\overline{l}} \otimes dz_{\overline{l}})(dz_l \otimes dz_{l})\\
	&+d\zeta_{\overline{N}}^{[l-1]}\overline{dz_N}^{[l-1]}\left((\zeta_{\overline{l}} \otimes dz_{\overline{l}}) (dz_l \otimes dz_l)
	+(\zeta_l \otimes dz_l)(dz_{\overline{l}} \otimes dz_{\overline{l}})\right)
\end{align*}
\begin{align*}
	\LieDer_{\widetilde{Y}_{l,l}} \Omega_2 =& (l-1+1) \zeta_{\overline{l}}^2\Omega_2\\
	&+ \zeta_{\overline{N}}d\zeta_{\overline{N}}^{[l-2]}\overline{dz_N}^{[l-1]}\left((2\nu_{l} -1)(d\zeta_{\overline{l}} \otimes dz_l) (dz_l \otimes dz_{\overline{l}})
	+(d\zeta_l \otimes dz_l)(dz_{\overline{l}} \otimes dz_{\overline{l}})\right)\\
	=& l \zeta_{\overline{l}}^2\Omega_2\\
	&+ \zeta_{\overline{N}}d\zeta_{\overline{N}}^{[l-2]}\overline{dz_N}^{[l-1]}
	\left((1-2\nu_{l})(d\zeta_{\overline{l}} \otimes dz_{\overline{l}}) (dz_l \otimes dz_l)
	+(d\zeta_l \otimes dz_l)(dz_{\overline{l}} \otimes dz_{\overline{l}})\right)
\end{align*}

in the current notation
\begin{align*}
	\omega_{l,l-1} \otimes \Theta_1 =& \zeta_{\overline{N}} d\zeta_{\overline{N}}^{[l-2]}\overline{dz_N}^{[l-1]} \left(d\zeta_l \otimes dz_l + d\zeta_{\overline{l}}\otimes dz_{\overline{l}}\right)\left(dz_l \otimes dz_l + dz_{\overline{l}}\otimes dz_{\overline{l}}\right)\\
	&+ d\zeta_{\overline{N}}^{[l-1]}\overline{dz_N}^{[l-1]}\left(\zeta_l \otimes dz_l + \zeta_{\overline{l}}\otimes dz_{\overline{l}}\right)\left(dz_l \otimes dz_l + dz_{\overline{l}}\otimes dz_{\overline{l}}\right)\\
	=&\zeta_{\overline{N}} d\zeta_{\overline{N}}^{[l-2]}\overline{dz_N}^{[l-1]} 
	\left((d\zeta_l \otimes dz_l)(dz_{\overline{l}}\otimes dz_{\overline{l}}) + (d\zeta_{\overline{l}}\otimes dz_{\overline{l}})(dz_l \otimes dz_l)\right)\\
	&+ d\zeta_{\overline{N}}^{[l-1]}\overline{dz_N}^{[l-1]}
	\left((\zeta_l \otimes dz_l)(dz_{\overline{l}}\otimes dz_{\overline{l}}) + (\zeta_{\overline{l}}\otimes dz_{\overline{l}})(dz_l \otimes dz_l)\right)
\end{align*}

in total
\begin{align*}
	\LieDer_{\widetilde{Y}_{l,l}} \omega_{l,-l} \otimes \Theta_1 =& l \zeta_{\overline{l}}^2 \omega_{l,-l} \otimes \Theta_1 + \omega_{l,l-1} \otimes \Theta_1
\end{align*}

\begin{align*}
	\LieDer_{\widetilde{Y}_{l,l}} \omega_{l,-l,m} =& (m-2)\zeta_{\overline{1}}^{m-3}\zeta_{\overline{l}}^2 \zeta_{\overline{1}} \omega_{l,-l} + \zeta_{\overline{1}}^{m-2}(l \zeta_{\overline{l}}^2 \omega_{l,-l} + \omega_{l,l-1})\\
	=&(m-2+l) \zeta_{\overline{l}}^2 \omega_{l,-l,m} + \omega_{l,l-1,m}
\end{align*}

\hrule

\begin{align*}
	\overline{\omega_{l,-l}} \otimes \Theta_2 =&
	\overline{d\eta_{\overline{N}}}^{[l-1]}dw_N^{[l-1]}(\zeta_{\overline{l}} \otimes d\zeta_{l})(dz_{\overline{l}} \otimes d\zeta_{\overline{l}})\\
	&+\overline{\eta_{\overline{N}}}\overline{d\eta_{\overline{N}}}^{[l-2]}dw_N^{[l-1]}(d\zeta_{\overline{l}} \otimes d\zeta_l)(dz_{\overline{l}} \otimes d\zeta_{\overline{l}})\\
	=:& \Omega_1 + \Omega_2
\end{align*}
\begin{align*}
	\sigma_{l,l-1} \otimes \Theta_1 =&\overline{d\zeta_{N}}^{[l-1]}\zeta_{\overline{N}} d\zeta_{\{l,\overline{l}\}}^{[1]} dz_{\{l,\overline{l}\}}^{[1]}dz_{\overline{N}}^{[l-2]}\\
	 &+ \overline{d\zeta_N}^{[l-1]}\zeta_{\{l,\overline{l}\}} d\zeta_{\{l,\overline{l}\}}^{[1]}dz_{\overline{N}}^{[l-1]}\\
	 =&\overline{d\zeta_{N}}^{[l-1]}\zeta_{\overline{N}}dz_{\overline{N}}^{[l-2]} d\zeta_{\{l,\overline{l}\}} dz_{\{l,\overline{l}\}}\\
	 &+ (-1)^{l-1}\overline{d\zeta_N}^{[l-1]}dz_{\overline{N}}^{[l-1]}\zeta_{\{l,\overline{l}\}} d\zeta_{\{l,\overline{l}\}}\\
	 =&\overline{d\zeta_{N}}^{[l-1]}\zeta_{\overline{N}}dz_{\overline{N}}^{[l-2]} \left((d\zeta_l \otimes dz_l)(dz_{\overline{l}}\otimes dz_{\overline{l}}) + (d\zeta_{\overline{l}}\otimes dz_{\overline{l}})(dz_l \otimes dz_l)\right)\\
	 &+ (-1)^{l-1}\overline{d\zeta_N}^{[l-1]}dz_{\overline{N}}^{[l-1]}\left((\zeta_l \otimes dz_l)(d\zeta_{\overline{l}}\otimes dz_{\overline{l}}) + (\zeta_{\overline{l}}\otimes dz_{\overline{l}})(d\zeta_l \otimes dz_l)
	 \right)\\
	 =:& \Xi_1^1 + \Xi_1^2 + \Xi_2^1 + \Xi_2^2
\end{align*}
The coefficient of $\tau$ is zero!!!!

\begin{table}
	\begin{tabular}{|c||c|c|c|c|}
		\hline
		$\wedge$& $\Xi_1^1$ & $\Xi_1^2$ & $\Xi_2^1$ & $\Xi_2^2$ \\
		\hline\hline
		$\Omega_1$ & $0^\ddagger$ &  &  & \\
		\hline
		$\Omega_2$ & $0^\ddagger$ & $0^\dagger$ & $0^\dagger$ & \\
		\hline
	\end{tabular}
	\centering
	\caption{$0$ terms}
	\label{tab:pairlm1_0termsFirst}
\end{table}

terms containing $d\zeta_{\overline{l}} \otimes (\cdot)$ twice: $0^\dagger$

terms containing $dz_{\overline{l}} \otimes (\cdot)$ twice: $0^\ddagger$

\begin{table}[h]
	\begin{tabular}{|c||c|c||c|c|c|}
		\hline
		$\wedge$ & $\alpha_{\overline{N}}$ & $\alpha_{l}$ & $\gamma_{N}$ & $\gamma_{l}$ & $\gamma_{\overline{l}}$ \\
		\hline\hline
		$\Omega_1 \wedge \Xi_1^2$ & & $0^{xx}$ & $0^\dagger$ & & $0^{\star\star}$ \\
		\hline
		$\Omega_1 \wedge \Xi_2^1$ & $0^{\ast\ast}$ & & $0^\dagger$ & & $0^{\star\star}$ \\
		\hline
		$\Omega_1 \wedge \Xi_2^2$ & $0^{\ast\ast}$ & & $0^\dagger$ & $0^{\times\times}$ & \\
		\hline
		$\Omega_2 \wedge \Xi_2^2$ & $0^{\ast\ast}$ & & & $0^{\times\times}$ & $0^{\star}$ \\
		\hline
	\end{tabular}
	\centering
	\caption{Wedge products with $\alpha$ and $\gamma$}
	\label{tab:pairWedglm1AlphGam}
\end{table}

\begin{itemize}
	\item $\Omega_j \wedge \alpha_{\overline{l}} = 0$ for all $j$, therefore the column is omitted
	\item $\Omega_j \wedge \alpha_N = 0$ for all $j$, therefore the column is omitted
	\item $\Xi_j^k \wedge \gamma_{\overline{N}} = 0$ for all $j,k$, therefore the column is omitted
	\item $\Xi_2^j \wedge \alpha_{\overline{N}} = 0$ for all $j$: $0^{\ast\ast}$
	\item $\Xi_1^2 \wedge \alpha_l = 0$: $0^{xx}$
	\item $\Omega_1 \wedge \gamma_N = 0$: $0^\dagger$
	\item $\Xi_2^2 \wedge \gamma_l = 0$: $0^{\times\times}$
	\item $\Omega_2 \wedge \gamma_{\overline{l}} = 0$: $0^{\star}$
	\item $\Xi_1^2 \wedge \gamma_{\overline{l}} = 0 = \Xi_2^1 \wedge \gamma_{\overline{l}}$: $0^{\star\star}$
\end{itemize}

hurra, only one term per line is non-zero!

By Eq.~\eqref{eq:redZetaZAlpha}
\begin{align*}
	\zeta_{\overline{N}}dz_{\overline{N}}^{[l-2]} \alpha_{\overline{N}} = (-1)^{l-2}\nu_{\overline{N}}dz_{\overline{N}}^{[l-1]}
\end{align*}
Therefore
\begin{align*}
	\Xi_1^2 \wedge \alpha_{\overline{N}} = (-1)^{l}\nu_{\overline{N}}\overline{d\zeta_{N}}^{[l-1]}dz_{\overline{N}}^{[l-1]}(d\zeta_{\overline{l}}\otimes dz_{\overline{l}})(dz_l \otimes dz_l)
\end{align*}
\begin{align*}
	\Xi_2^1 \wedge \alpha_l = (-1)^l\nu_l \overline{d\zeta_{N}}^{[l-1]}dz_{\overline{N}}^{[l-1]}(d\zeta_{\overline{l}}\otimes dz_{\overline{l}})(dz_l \otimes dz_l)
\end{align*}

By Eq.~\eqref{eq:redZetaGamma}
\begin{align*}
	\overline{\eta_{\overline{N}}}\overline{d\eta_{\overline{N}}}^{[l-2]} \gamma_N = (-1)^{l-2} \nu_N\overline{d\eta_{\overline{N}}}^{[l-1]}
\end{align*}

\begin{align*}
	\Omega_2 \wedge \gamma_N =& (-1)^{l-1}\overline{\eta_{\overline{N}}}\overline{d\eta_{\overline{N}}}^{[l-2]}\gamma_Ndw_N^{[l-1]}(d\zeta_{\overline{l}} \otimes d\zeta_l)(dz_{\overline{l}} \otimes d\zeta_{\overline{l}})\\
	=&(-1)^{l-1+l-2}\nu_N\overline{d\eta_{\overline{N}}}^{[l-1]}dw_N^{[l-1]}(d\zeta_{\overline{l}} \otimes d\zeta_l)(dz_{\overline{l}} \otimes d\zeta_{\overline{l}})
\end{align*}
\begin{align*}
	\Omega_1 \wedge \gamma_{\overline{l}} =&\overline{d\eta_{\overline{N}}}^{[l-1]}dw_N^{[l-1]}(\zeta_{\overline{l}} \otimes d\zeta_{l})(dz_{\overline{l}} \otimes d\zeta_{\overline{l}})\gamma_{\overline{l}}\\
	=&-\nu_l\overline{d\eta_{\overline{N}}}^{[l-1]}dw_N^{[l-1]}(d\zeta_{\overline{l}} \otimes d\zeta_{l})(dz_{\overline{l}} \otimes d\zeta_{\overline{l}})
\end{align*}
\begin{align*}
	\Omega_1 \wedge \gamma_l =&\overline{d\eta_{\overline{N}}}^{[l-1]}dw_N^{[l-1]}(\zeta_{\overline{l}} \otimes d\zeta_{l})(dz_{\overline{l}} \otimes d\zeta_{\overline{l}})\gamma_l \\
	=&-\zeta_{\overline{l}}^2\overline{d\eta_{\overline{N}}}^{[l-1]}dw_N^{[l-1]}(d\zeta_l \otimes d\zeta_{l})(dz_{\overline{l}} \otimes d\zeta_{\overline{l}})
\end{align*}

\begin{align*}
	\nu_l\Xi_1^2 \wedge \alpha_{\overline{N}} =& \nu_{\overline{N}}\Xi_2^1 \wedge \alpha_l\\
	\nu_l\Omega_2 \wedge \gamma_N =& \nu_N \Omega_1 \wedge \gamma_{\overline{l}}
\end{align*}

\begin{align*}
	\Omega_1 \wedge \Xi_2^1 \wedge \alpha_l \wedge \gamma_l =& 
	(-1)^{l+0}\nu_l (\Omega_1 \wedge \gamma_l) \overline{d\zeta_{N}}^{[l-1]}dz_{\overline{N}}^{[l-1]}(d\zeta_{\overline{l}}\otimes dz_{\overline{l}})(dz_l \otimes dz_l)\\
	=&(-1)^{l+1}\nu_l \zeta_{\overline{l}}^2\overline{d\eta_{\overline{N}}}^{[l-1]}dw_N^{[l-1]}(d\zeta_l \otimes d\zeta_{l})(dz_{\overline{l}} \otimes d\zeta_{\overline{l}})\\ &\overline{d\zeta_{N}}^{[l-1]}dz_{\overline{N}}^{[l-1]}(d\zeta_{\overline{l}}\otimes dz_{\overline{l}})(dz_l \otimes dz_l)\\
	=&(-1)^{l+1+1}\nu_l \zeta_{\overline{l}}^2 \Theta \otimes \Theta
\end{align*}

\begin{align*}
	\Omega_1 \wedge \Xi_1^2 \wedge \alpha_{\overline{N}} \wedge \gamma_l =& \frac{\nu_{\overline{N}}}{\nu_l} \Omega_1 \wedge \Xi_2^1 \wedge \alpha_{l} \wedge \gamma_l\\
	=&(-1)^{l}\nu_N \zeta_{\overline{l}}^2 \Theta \otimes \Theta
\end{align*}

\begin{align*}
	\Omega_1  \wedge \Xi_2^2 \wedge \alpha_l \wedge \gamma_{\overline{l}} =& (-1)^{0}(\Omega_1 \wedge \gamma_{\overline{l}})(\Xi_2^2 \wedge \alpha_l)\\
	=&(-1)^{1}\nu_l\overline{d\eta_{\overline{N}}}^{[l-1]}dw_N^{[l-1]}(d\zeta_{\overline{l}} \otimes d\zeta_{l})(dz_{\overline{l}} \otimes d\zeta_{\overline{l}})\\
	&(-1)^{l-1}\overline{d\zeta_N}^{[l-1]}dz_{\overline{N}}^{[l-1]}(\zeta_{\overline{l}}\otimes dz_{\overline{l}})(d\zeta_l \otimes dz_l)\alpha_l\\
	=&(-1)^{l+1}\nu_l\zeta_{\overline{l}}^2\overline{d\eta_{\overline{N}}}^{[l-1]}dw_N^{[l-1]}(d\zeta_{\overline{l}} \otimes d\zeta_{l})(dz_{\overline{l}} \otimes d\zeta_{\overline{l}})\\
	&\overline{d\zeta_N}^{[l-1]}dz_{\overline{N}}^{[l-1]}(dz_l\otimes dz_{\overline{l}})(d\zeta_l \otimes dz_l)\\
	=&(-1)^{l+1+1}\nu_l\zeta_{\overline{l}}^2 \Theta \otimes \Theta
\end{align*}

\begin{align*}
	\Omega_2 \wedge \Xi_2^2 \wedge \alpha_l \wedge \gamma_N =& \frac{\nu_N}{\nu_l}\Omega_1 \wedge \Xi_2^2 \wedge \alpha_l \wedge \gamma_{\overline{l}}\\
	=&(-1)^l\nu_N\zeta_{\overline{l}}^2 \Theta \otimes \Theta
\end{align*}

in total as $\Theta = (-1)^{n+l+1}\vol_{S\RR^n} \gamma$
\begin{align*}
	\overline{\omega_{l,-l}} \wedge \sigma_{l,l-1}\wedge \alpha \wedge \gamma =& (-1)^l\zeta_{\overline{l}}^2 (2\nu_N + 2\nu_l) \Theta\\
	=&(-1)^{n+1}\zeta_{\overline{l}}^2 \nu \vol_{S\RR^n} \gamma \otimes \Theta
\end{align*}
hence, as we can factor $\gamma$
\begin{align*}
	\overline{\omega_{l,-l,m}} \wedge D\omega_{l,l-1,m} =& (-1)^{n+1}|\zeta_{\overline{1}}|^{2(m-2)}\zeta_{\overline{l}}^2\nu(-1)^{n+1}(n+m-l-2)(m+l-2)\vol_{S\RR^n}\\
	=& (l+m-2)(l+m-2)|\zeta_{\overline{1}}|^{2(m-2)}\zeta_{\overline{l}}^2\nu\vol_{S\RR^n}
\end{align*}
this is Eq.~\eqref{eq:pairOmkmOmkm1m} for $r=l$ and $k=l-1$. We can therefore conclude as before that
\begin{align*}
	\int \pair{\overline{\LieDer_{\widetilde{Y}_{l,l}} \omega_{l,-l,m}} \wedge D \omega_{l,l-1,m}}{\vol} \neq 0
\end{align*}
\end{extracalc}

\subsection{Calculations for $\lambda=0$}\label{sec:weight0Form}

Throughout this section, we set $K=\{1\}$, $L=\Iind \setminus K\cup \overline{K}$. Complementing the notation in \autoref{sec:doubleforms}, we set $\lambda_{1,m}=(m,0,\dots,0)\in\ZZ^l$ for $m\in\NN_0$. In particular, $\lambda_{1,0}=0$.
\begin{theorem}
	\label{thm:HWVintrinsicVolumes}
	For any $m\in\NN_0$, $0\leq r\leq n-1$, the form
	\begin{align*}		\omega_{r,0,m}=\zeta_{\bar{1}}^m\omega_{r,0}\in\Omega^{n-1}(S\RR^n)
	\end{align*}
	is a highest weight vector with weight $\lambda_{1,m}$ in $\Omega^{r,n-r-1}(S\RR^n)^{\mathrm{\mathrm{tr}}}$, where
	\begin{align*}
		\omega_{r,0}\otimes\Theta_1=\zeta_{\Iind} d\zeta_{\Iind}^{[n-r-1]}dz_{\Iind}^{[r]}.
	\end{align*}
\end{theorem}
\begin{proof}
	Note that $\zeta_{\overline{1}}^m$ is a highest weight vector of weight $\lambda_{1,m}$. As a product of two highest weight vectors is again a highest weight vector with weight equal to the sum of their weights, we therefore need to show that $\omega_{r,0}$ is a highest weight vector of weight $(0, \dots, 0)$, that is, $\omega_{r,0}$ is invariant under the action of $\SO(n)$.
	
	Note that
	\begin{align*}
		\omega_{r,0}\otimes \Theta_1 = \zeta_{\Iind} d\zeta_{\Iind}^{[n-r-1]}dz_{\Iind}^{[r]} = (-1)^l \overline{\zeta_{\Iind}} \overline{d\zeta_{\Iind}}^{[n-r-1]}\overline{dz_{\Iind}}^{[r]},
	\end{align*}
	and a short calculation shows that
	\begin{align*}
		\overline{\zeta_{\Iind}} = \sum_{j=1}^n \xi_j \otimes dx_j, \quad
		\overline{d\zeta_{\Iind}} = \sum_{j=1}^n d\xi_j \otimes dx_j, \quad
		\overline{dz_{\Iind}} = \sum_{j=1}^n dx_j \otimes dx_j
	\end{align*}
	are $\SO(n)$-invariant forms. Hence, $\omega_{r,0}\otimes \Theta_1$ is a product of $\SO(n)$-invariant forms and therefore $\SO(n)$-invariant, which shows that $\omega_{r,0}$ is $\SO(n)$-invariant.
	\begin{extracalc}
		\begin{align*}
			\overline{\zeta_{\Iind}} =& \sum_{i =1}^{l} \left( \zeta_{i} \otimes dz_{\overline{i}} + \zeta_{\overline{i}} \otimes dz_i \right) + \mathbbm{1}_{n \text{ odd}} \xi_n \otimes dz_n\\
			=&\sum_{i =1}^{l} \frac{1}{2}\left( (\xi_{2i-1} + i \xi_{2i}) \otimes (dx_{2i-1} - i dx_{2i}) + (\xi_{2i-1} - i \xi_{2i}) \otimes (dx_{2i-1} + i dx_{2i}) \right) + \mathbbm{1}_{n \text{ odd}} \xi_n \otimes dz_n\\
			=&\sum_{i =1}^{l} \left( \xi_{2i-1}\otimes dx_{2i-1} + \xi_{2i}\otimes dx_{2i}\right) + \mathbbm{1}_{n \text{ odd}} \xi_n \otimes dz_n\\
			=& \sum_{j=1}^n \xi_j \otimes dx_j
		\end{align*}
		For $g=(g_{ab})_{a,b} \in \SO(n)$
		\begin{align*}
			g^\ast \overline{\zeta_{\Iind}} =& \sum_{j=1}^n \sum_{a,b} g_{ja} g_{jb} \xi_a \otimes dx_b \\
			=&\sum_{a,b} \left(\sum_{j=1}^n  g_{ja} g_{jb}\right) \xi_a \otimes dx_b\\
			=&\sum_{a,b} \delta_{ab} \xi_a \otimes dx_b = \overline{\zeta_{\Iind}}
		\end{align*}
	\end{extracalc}
\end{proof}

Next, we show that $\omega_{r,0,m}$, $m\ne 1$, defines a nontrivial valuation by showing that its Rumin differential $D\omega_{r,0,m}$ is non-zero. We establish the required relations with a similar reasoning as in \cite{Kotrbaty2022}*{Sec.~5}. Formally, the formulas below correspond to the case $k=0$ in \cite{Kotrbaty2022}*{Sec.~5}, however, this case is not covered by their calculations. The relations below thus complement the picture.
\medskip

For $0\leq r\leq n$, consider the form $\delta_{r,0}$ on $\RR^n\times\RR^n$ given by
\begin{align*}
	\delta_{r,0} \otimes \Theta_1 =d\zeta_{\Iind}^{[n-r]}dz_{\Iind}^{[r]}.
\end{align*}

\begin{lemma}
	For $0\leq r\leq n-1$, the following relations hold on $\RR^n\times\RR^n$:
	\begin{align}
		\label{eq:differenitalOmega_0}
		d\omega_{r,0}=&(n-r)\delta_{r,0},\\
		\label{eq:relationDelta_0Omega_0}
		\nu\delta_{r,0}=&\begin{cases}
			\alpha\omega_{r-1,0}+\gamma\omega_{r,0}, &r>0,\\
			\gamma\omega_{r,0}, & r=0.
		\end{cases}%\\
%		\label{eq:alphDeltPlusGammDeltZero}\alpha \delta_{r,0} =& - \gamma \delta_{r+1,0}.
	\end{align}
\end{lemma}
\begin{proof}
	Eq.~\eqref{eq:differenitalOmega_0} follows directly from the definition. For Eq.~\eqref{eq:relationDelta_0Omega_0}, we note that, using Eq.~\eqref{eq:binomDFCalc} multiple times and recalling that $\Iind_i = \Iind \setminus\{i\}$, we obtain for each $i\in\Iind$,
	\begin{align*}
		\zeta_{\bar i}dz_i\omega_{r-1,0}\otimes \Theta_1=&\zeta_{\bar i}dz_i\zeta_{\Iind} d\zeta_{\Iind}^{[n-r]}dz_{\Iind}^{[r-1]}\\
		=&\zeta_{\bar i}dz_i(\zeta_i\otimes dz_i)d\zeta_{\Iind_i}^{[n-r]}dz_{\Iind_i}^{[r-1]}+\zeta_{\bar i}dz_i\zeta_{\Iind_i} (d\zeta_i\otimes dz_i)d\zeta_{\Iind_i}^{[n-r-1]}dz_{\Iind_i}^{[r-1]}\\
		=&|\zeta_i|^2 (dz_i\otimes dz_i)d\zeta_{\Iind_i}^{[n-r]}dz_{\Iind_i}^{[r-1]}-\zeta_{\bar i}d\zeta_i\zeta_{\Iind_i} (dz_i\otimes dz_i)d\zeta_{\Iind_i}^{[n-r-1]}dz_{\Iind_i}^{[r-1]}\\
		=&\left(|\zeta_i|^2 d\zeta_{\Iind}^{[n-r]}dz_{\Iind}^{[r]}-|\zeta_i|^2 (d\zeta_i\otimes dz_i)d\zeta_{\Iind_i}^{[n-r-1]}dz_{\Iind}^{[r]}\right)\\
		&-\left(\gamma_{\{i\}} \zeta_{\Iind}d\zeta_{\Iind}^{[n-r]}dz_{\Iind}^{[r]}-\zeta_{\bar i}d\zeta_i (\zeta_{i}\otimes dz_i)d\zeta_{\Iind}^{[n-r-1]}dz_{\Iind}^{[r-1]}\right)\\
		=&|\zeta_i|^2 d\zeta_{\Iind}^{[n-r]}dz_{\Iind}^{[r]}-\gamma_{\{i\}} \zeta_{\Iind}d\zeta_{\Iind}^{[n-r]}dz_{\Iind}^{[r]}.
	\end{align*}
	Summing over $i\in\Iind$, the desired result follows.
\end{proof}

We have the following relation to the forms $\tau_{r,k}$ from Eq.~\eqref{eq:defTau}.
\begin{lemma}
	For $1\leq r\leq n-1$, the following relation holds on $\RR^n\times\RR^n$:
	\begin{align}
		\label{eq:relationTauDelta_0}
		\zeta_{\bar1}^2\delta_{r,0}=\alpha\tau_{r,1}+ \gamma\tau_{r+1,1}.
	\end{align}
\end{lemma}
\begin{proof}
	From Eq.~\eqref{eq:taurkSep}, we obtain using $K = \{1\}$, $L = \Iind \setminus\{1, \overline{1}\}$ and $\overline{d\zeta_{K}}\zeta_{\overline{K}} = \overline{\zeta_K}d\zeta_{\overline{K}}$,
	\begin{align*}
		\alpha\tau_{r,1}\otimes \Theta_1=&(\alpha_K+\alpha_{\overline{K}}+\alpha_L)\overline{\zeta_K}d\zeta_{\overline{K}}d\zeta_L^{[n-r-1]}dz_L^{[r-1]}\\
		&+(\alpha_K+\alpha_L)\overline{\zeta_K}dz_{\overline{K}}d\zeta_L^{[n-r]}dz_L^{[r-2]}.
	\end{align*}
	As $\overline{\zeta_{K}} = \zeta_{\overline{1}} \otimes dz_1$, and similarly for $d\zeta_{\overline{K}}$, $dz_{\overline{K}}$ and $\alpha$, clearly
	\begin{align*}
		\alpha_K\overline{\zeta_K}=\zeta_{\bar1}^2 dz_K, \quad
		\alpha_{\overline{K}}d\zeta_{\overline{K}}=-\gamma_{\overline{K}}dz_{\overline{K}}, \quad \text{ and } \quad
		\gamma_K\overline{\zeta_K}=\zeta_{\bar1}^2d\zeta_K.
	\end{align*}
	Thus, in combination with \eqref{eq:redAddLAlpGamTerms1}, we obtain
	\begin{align*}
		\alpha\tau_{r,1}\otimes \Theta_1=&\zeta_{\bar1}^2 dz_Kd\zeta_{\overline{K}}d\zeta_L^{[n-r-1]}dz_L^{[r-1]}-\gamma_{\overline{K}}\overline{\zeta_K}dz_{\overline{K}}d\zeta_L^{[n-r-1]}dz_L^{[r-1]}\\
		&-\gamma_L\overline{\zeta_K}d\zeta_{\overline{K}}d\zeta_L^{[n-r-2]}dz_L^{[r]} +\zeta_{\bar1}^2 dz_Kdz_{\overline{K}}d\zeta_L^{[n-r]}dz_L^{[r-2]}\\
		&-\gamma_L\overline{\zeta_K}dz_{\overline{K}}d\zeta_L^{[n-r-1]}dz_L^{[r-1]}.
	\end{align*}
	We similarly obtain
	\begin{align*}
		\gamma\tau_{r+1,1}\otimes\Theta_1=&(\gamma_K+\gamma_L)\overline{\zeta_{K}}d\zeta_{\overline{K}}d\zeta_L^{[n-r-2]}dz_L^{[r]}\\
		&+(\gamma_K+\gamma_{\overline{K}}+\gamma_L)\overline{\zeta_{K}}dz_{\overline{K}}d\zeta_L^{[n-r-1]}dz_L^{[r-1]}\\
		=&\zeta_{\bar1}^2d\zeta_Kd\zeta_{\overline{K}}d\zeta_L^{[n-r-2]}dz_L^{[r]} +\gamma_L\overline{\zeta_{K}}d\zeta_{\overline{K}}d\zeta_L^{[n-r-2]}dz_L^{[r]}\\
		&+\zeta_{\bar1}^2d\zeta_Kdz_{\overline{K}}d\zeta_L^{[n-r-1]}dz_L^{[r-1]}+\gamma_{\overline{K}}\overline{\zeta_{K}}dz_{\overline{K}}d\zeta_L^{[n-r-1]}dz_L^{[r-1]}\\
		&+\gamma_L\overline{\zeta_{K}}dz_{\overline{K}}d\zeta_L^{[n-r-1]}dz_L^{[r-1]}.
	\end{align*}
	Combining the two equations, we obtain the claim:
	\begin{align*}
		\alpha\tau_{r,1}\otimes \Theta_1+\gamma\tau_{r+1,1}\Theta_1
		=&\zeta_{\bar1}^2 \left(dz_Kd\zeta_{\overline{K}}d\zeta_L^{[n-r-1]}dz_L^{[r-1]}
		+dz_Kdz_{\overline{K}}d\zeta_L^{[n-r]}dz_L^{[r-2]} \right.\\
		&\left.d\zeta_Kd\zeta_{\overline{K}}d\zeta_L^{[n-r-2]}dz_L^{[r]}
		+d\zeta_Kdz_{\overline{K}}d\zeta_L^{[n-r-1]}dz_L^{[r-1]}\right)\\
		=&\zeta_{\bar1}^2d\zeta_{\Iind}^{[n-r]}dz_{\Iind}^{[r]}.
	\end{align*}
	\begin{extracalc}
		The last equality follows from
		\begin{align*}
			d\zeta_{\Iind}^{[n-r]} dz_{\Iind}^{[r]} =& (d\zeta_{\Iind_1}^{[n-r]} + d\zeta_{K} d\zeta_{\Iind_1}^{[n-r-1]})(dz_{\Iind_1}^{[r]} + dz_K dz_{\Iind_1}^{[r-1]})\\
			=&d\zeta_{\Iind_1}^{[n-r]}dz_K dz_{\Iind_1}^{[r-1]} + d\zeta_{K} d\zeta_{\Iind_1}^{[n-r-1]}dz_{\Iind_1}^{[r]}\\
			=&dz_K\left(d\zeta_{L}^{[n-r]} + d\zeta_{\overline{K}}d\zeta_{L}^{[n-r-1]}\right) \left(dz_{L}^{[r-1]} + dz_{\overline{K}}dz_{L}^{[r-2]}\right)\\
			&+ d\zeta_{K} \left(d\zeta_{L}^{[n-r-1]} + d\zeta_{\overline{K}}d\zeta_{L}^{[n-r-2]}\right)\left(dz_{L}^{[r]} + dz_{\overline{K}}dz_{L}^{[r-1]}\right)\\
			=&dz_K\left(d\zeta_{L}^{[n-r]}dz_{\overline{K}}dz_{L}^{[r-2]} + d\zeta_{\overline{K}}d\zeta_{L}^{[n-r-1]}dz_{L}^{[r-1]} \right)\\
			&+ d\zeta_{K} \left(d\zeta_{L}^{[n-r-1]}dz_{\overline{K}}dz_{L}^{[r-1]} + d\zeta_{\overline{K}}d\zeta_{L}^{[n-r-2]}dz_{L}^{[r]} \right).
		\end{align*}
	\end{extracalc}
\end{proof}

\begin{lemma}
	The following relation holds on $\RR^n\times\RR^n$:
	\begin{align}
		\label{eq:OmegaNull_dzeta}
		\zeta_{\bar1}d\zeta_{\bar 1}\omega_{r,0}=-\nu\delta_{r,1}+\alpha\sigma_{r,1}+\gamma\sigma_{r+1,1}.
	\end{align}
\end{lemma}
\begin{proof}
	\begin{extracalc}
		\begin{align*}
			d\zeta_{\overline{1}} \zeta_{\Iind} d\zeta_{\Iind}^{[n-r-1]}dz_{\Iind}^{[r]} =&d\zeta_{\overline{1}} \left(\zeta_{K} d\zeta_{\Iind_1}^{[n-r-1]}dz_{\Iind_1}^{[r]} + \zeta_{\Iind_1}d\zeta_K d\zeta_{\Iind_1}^{[n-r-2]}dz_{\Iind_1}^{[r]} + \zeta_{\Iind_1} d\zeta_{\Iind_1}^{[n-r-1]}dz_Kdz_{\Iind_1}^{[r-1]}\right)\\
			=&d\zeta_{\overline{1}} \left(\zeta_{K} d\zeta_{L}^{[n-r-1]}dz_{\Iind_1}^{[r]} + \zeta_{\Iind_1}d\zeta_K d\zeta_{L}^{[n-r-2]}dz_{\Iind_1}^{[r]} + \zeta_{\Iind_1} d\zeta_{L}^{[n-r-1]}dz_Kdz_{\Iind_1}^{[r-1]}\right)\\
			=&d\zeta_{\overline{1}} \left(\zeta_{K} d\zeta_{L}^{[n-r-1]}dz_{\overline{K}}dz_{L}^{[r-1]} \right.\\
			& + \zeta_{\overline{K}}d\zeta_K d\zeta_{L}^{[n-r-2]}dz_{L}^{[r]} + \zeta_{L}d\zeta_K d\zeta_{L}^{[n-r-2]}dz_{\overline{K}}dz_{L}^{[r-1]}\\
			& \left. + \zeta_{\overline{K}} d\zeta_{L}^{[n-r-1]}dz_Kdz_{L}^{[r-1]} + \zeta_{L} d\zeta_{L}^{[n-r-1]}dz_Kdz_{\overline{K}}dz_{L}^{[r-2]}\right)
		\end{align*}
	\end{extracalc}
	Writing $K=\{1\}$, $L = \Iind\setminus\{1,\overline{1}\}$, a short calculation shows
	\begin{align*}
		d\zeta_{\overline{1}}\omega_{r,0}\otimes \Theta_1
		=&d\zeta_{\overline{1}}\zeta_K dz_{\overline{K}}d\zeta_L^{[n-r-1]}dz_L^{[r-1]}\\
		&+d\zeta_{\overline{1}}\zeta_{\overline{K}}dz_Kd\zeta_L^{[n-r-1]}dz_L^{[r-1]}+d\zeta_{\overline{1}}\zeta_{\overline{K}}d\zeta_Kd\zeta_L^{[n-r-2]}dz_L^{[r]}\\
		&+d\zeta_{\overline{1}}d\zeta_{K}dz_{\overline{K}}\zeta_{L}d\zeta_L^{[n-r-2]}dz_L^{[r-1]}+d\zeta_{\overline{1}}dz_{K}dz_{\overline{K}}\zeta_{L}d\zeta_L^{[n-r-1]}dz_L^{[r-2]}\\
		=&-dz_{\overline{1}}(\zeta_1\otimes dz_1)(d\zeta_{\overline{1}}\otimes dz_{\overline{1}})d\zeta_L^{[n-r-1]}dz_L^{[r-1]}\\
		&-dz_1\zeta_{\overline{K}}\overline{d\zeta_K}d\zeta_L^{[n-r-1]}dz_L^{[r-1]}-d\zeta_{1}\zeta_{\overline{K}}\overline{d\zeta_K}d\zeta_L^{[n-r-2]}dz_L^{[r]}\\
		&-d\zeta_{1}\overline{d\zeta_{K}}dz_{\overline{K}}\zeta_{L}d\zeta_L^{[n-r-2]}dz_L^{[r-1]}-dz_1\overline{d\zeta_{K}}dz_{\overline{K}}\zeta_{L}d\zeta_L^{[n-r-1]}dz_L^{[r-2]}\\
		=&-dz_{\overline{1}}(\zeta_1\otimes dz_1)(d\zeta_{\overline{1}}\otimes dz_{\overline{1}})d\zeta_L^{[n-r-1]}dz_L^{[r-1]}\\
		&+dz_1\sigma_{r,1}\otimes\Theta_1+d\zeta_1\sigma_{r+1,1}\otimes\Theta_1,
	\end{align*}
	where we used Eq.~\eqref{eq:sigmarkSep} in the last step. On the other hand, Eq.\eqref{eq:deltarkSep} shows that
	\begin{align*}
		\delta_{r,1}\otimes\Theta_1=&-\overline{d\zeta_K}d\zeta_l^{[n-r-1]}dz_L^{[r-1]}dz_{\overline{K}}\\
		=&(dz_{\overline{1}}\otimes dz_1)(d\zeta_{\overline{1}}\otimes dz_{\overline{1}})d\zeta_L^{[n-r-1]}dz_L^{[r-1]},
	\end{align*}
	so
	\begin{align*}
		d\zeta_{\bar 1}\omega_{r,0}=-\zeta_1 \delta_{r,1}+dz_1\sigma_{r,1}+d\zeta_1\sigma_{r+1,1}.
	\end{align*}
	In particular, and since by \cite{Kotrbaty2022}*{Lem.~5.3(b)} $\delta_{r,1}\nu_{K} + (-1)^n \sigma_{r,1} \alpha_{\overline{K}} = 0$,
	\begin{align}
		\label{eq:intermediateStepRumin1}
		\begin{split}
			\zeta_{\overline{1}}d\zeta_{\overline{1}}\omega_{r,0}=&-\nu_K \delta_{r,1}+\alpha_K\sigma_{r,1}+\gamma_K\sigma_{r+1,1}\\
			=&-2\nu_K \delta_{r,1}-\sigma_{r,1}\alpha_{\overline{K}}(-1)^n+\alpha_K\sigma_{r,1}+\gamma_K\sigma_{r+1,1}.
		\end{split}
	\end{align}
	We may further apply Eq.s~\eqref{eq:redAddLAlpGamTerms1} and \eqref{eq:redAddLAlpGamTerms2}, as well as Eq.s~\eqref{eq:deltarkSep} and \eqref{eq:sigmarkSep}, to obtain
	\begin{align*}
		\gamma_L\sigma_{r+1,1}\otimes\Theta_1=&\gamma_L\overline{d\zeta_K}\zeta_{\overline{K}}d\zeta_L^{[n-r-2]}dz_L^{[r]}-\gamma_L\overline{d\zeta_{K}}dz_{\overline{K}}\zeta_{L}d\zeta_L^{[n-r-2]}dz_L^{[r-1]}\\
		=&-\alpha_L\overline{d\zeta_K}\zeta_{\overline{K}}d\zeta_L^{[n-r-1]}dz_L^{[r-1]}+\alpha_L\overline{d\zeta_{K}}dz_{\overline{K}}\zeta_{L}d\zeta_L^{[n-r-1]}dz_L^{[r-2]}\\
		&-\nu_L\overline{d\zeta_{K}}dz_{\overline{K}}d\zeta_L^{[n-r-1]}dz_L^{[r-1]}\\
		=&-\alpha_L\sigma_{r,1}\otimes \Theta_1+(-1)^{n-1}\nu_L\delta_{r,1}\otimes\Theta_1,
	\end{align*}
	so using that $\gamma_{\overline{K}}\sigma_{r+1,1}=0$, we may rewrite Eq.~\eqref{eq:intermediateStepRumin1} to obtain
	\begin{align*}
		\zeta_{\overline{1}}d\zeta_{\overline{1}}\omega_{r,0}=&-2\nu_K \delta_{r,1}-\sigma_{r,1}\alpha_{\overline{K}}(-1)^n+\alpha_K\sigma_{r,1}+(\gamma-\gamma_L)\sigma_{r+1,1}\\
		=&-(2\nu_K+\nu_L)\delta_{r,1}+(\alpha_K+\alpha_{\overline{K}}+\alpha_L)\sigma_{r,1}+\gamma\sigma_{r+1,1}\\
		=&-\nu\delta_{r,1}+\alpha\sigma_{r,1}+\gamma\sigma_{r+1,1}.
	\end{align*}
\end{proof}

Next, we need the following relations for $\theta_{r,k}$ (as defined in Eq.~\eqref{eq:defThetark}) from \cite{Kotrbaty2022}.
\begin{lemma}[\cite{Kotrbaty2022}]
	Let $1\leq k\le\min\{r,n-r\}$. Then the following relation holds in $\Omega^*(S\RR^n)$:
	\begin{align}
		\label{eq:theta_dalpha}
		\theta_{r,k}d\alpha=&\delta_{r,k}+(-1)^n\sigma_{r,k}\alpha
	\end{align}
	The following relations hold on $\RR^n\times\RR^n$:
	\begin{align}	
		\label{eq:dzeta_theta}
		d\zeta_{\bar1}\theta_{r,k}=&\zeta_{\bar1}\sigma_{r,k}\\
		\label{eq:dTheta}
		d\theta_{r,k}=&k\sigma_{r,k}+(-1)^{k+1}(n-r-k+1)\tau_{r,k}.
	\end{align}
\end{lemma}
\begin{proof}
	Eq.~\eqref{eq:theta_dalpha} is the content of \cite{Kotrbaty2022}*{Cor.~5.4}, Eq.~\eqref{eq:dzeta_theta} is \cite{Kotrbaty2022}*{Prop.~5.5(25)}, and Eq.~\eqref{eq:dTheta} is \cite{Kotrbaty2022}*{Prop.~5.5(26)}.
\end{proof}

We are now in position to determine the Rumin differential of $\omega_{r,0,m}=\zeta_{\bar 1}^m\omega_{r,0}$.
\begin{lemma}\label{lem:DintrinsicVols}
	For every $m\in\NN_0$,
	\begin{align*}
		D\omega_{r,0,m}=&d(\omega_{r,0,m}+m(-1)^n\zeta_{\bar 1}^{m-2}\theta_{r,1}\alpha)\\
		=&(m-1)(-1)^n\zeta_{\bar1}^{m-2}\left[m\sigma_{r,1}+(n-r)\tau_{r,1}\right]\alpha.
	\end{align*}
	In particular, for $m\ge 2$,
	\begin{align*}
		D\omega_{r,0,m}=-\frac{m-1}{n+m-r-2}D\omega_{r,1,m}.
	\end{align*}
\end{lemma}
\begin{proof}
	First, since $\omega_{r,0,m} = \zeta_{\overline{1}}^m \omega_{r,0}$, the Leibniz rule for differential forms implies
	\begin{align*}
		d(\omega_{r,0,m}+m(-1)^n\zeta_{\bar 1}^{m-2}\theta_{r,1}\alpha)
		=&m\zeta_{\bar 1}^{m-1}d\zeta_{\bar1}\omega_{r,0} + m(m-2)(-1)^n\zeta_{\bar1}^{m-3}d\zeta_{\bar1}\theta_{r,1}\alpha\\
		& +\zeta_{\bar1}^m d\omega_{r,0}+m(-1)^n\zeta_{\bar 1}^{m-2}d\theta_{r,1}\alpha+m\zeta_{\bar 1}^{m-2}\theta_{r,1}d\alpha.
	\end{align*}
	Concentrating on the single terms, Eqs.~\eqref{eq:OmegaNull_dzeta}, \eqref{eq:dzeta_theta}, \eqref{eq:differenitalOmega_0}, \eqref{eq:dTheta}, and \eqref{eq:theta_dalpha} yield
	\begin{align*}
		m\zeta_{\bar 1}^{m-1}d\zeta_{\bar1}\omega_{r,0} &= m\zeta_{\bar 1}^{m-2}\left(-\nu \delta_{r,1} + \alpha \sigma_{r,1} + \gamma \sigma_{r+1,1} \right),\\
		m(m-2)(-1)^n\zeta_{\bar1}^{m-3}d\zeta_{\bar1}\theta_{r,1}\alpha &=m(m-2)(-1)^n\zeta_{\bar1}^{m-2}\sigma_{r,1}\alpha, \\
		\zeta_{\bar1}^m d\omega_{r,0}&=(n-r)\zeta_{\bar1}^m \delta_{r,0},\\
		m(-1)^n\zeta_{\bar 1}^{m-2}d\theta_{r,1}\alpha &=m(-1)^n\zeta_{\bar 1}^{m-2}\left(\sigma_{r,1}+(n-r)\tau_{r,1}\right)\alpha,\\
		m\zeta_{\bar 1}^{m-2}\theta_{r,1}d\alpha &=m\zeta_{\bar 1}^{m-2}\left(\delta_{r,1}+(-1)^n\sigma_{r,1}\alpha\right).
	\end{align*}
	We therefore obtain by summing up and combining the terms,
	\begin{align*}
		&d(\omega_{r,0,m}+m(-1)^n\zeta_{\bar 1}^{m-2}\theta_{r,1}\alpha)\\
		=&m\zeta_{\bar 1}^{m-2} (-1)^n\left( -\sigma_{r,1} + (m-2)\sigma_{r,1} + \sigma_{r,1}+(n-r)\tau_{r,1} + \sigma_{r,1} \right)\alpha\\
		&+ \zeta_{\bar 1}^{m-2}\left(-m \nu \delta_{r,1} + m\gamma \sigma_{r+1,1} + (n-r)\zeta_{\bar1}^2 \delta_{r,0} + m\delta_{r,1} \right)\\
		=&m\zeta_{\bar 1}^{m-2} (-1)^n\left( (m-1)\sigma_{r,1}+(n-r)\tau_{r,1} \right)\alpha\\
		&+ \zeta_{\bar 1}^{m-2}\left((n-r)\zeta_{\bar1}^2 \delta_{r,0} \right),
	\end{align*}
	where we used that $\nu \equiv 1$ and $\gamma|_{S\RR^n} = 0$ on $S\RR^n$. Next, by Eq.~\eqref{eq:relationTauDelta_0},
	\begin{align*}
		\zeta_{\bar1}^2 \delta_{r,0} = \alpha\tau_{r,1}+ \gamma\tau_{r+1,1} = \alpha\tau_{r,1}, \quad \text{ on $S\RR^n$},
	\end{align*}
	which shows the claim:
	\begin{align*}
		d(\omega_{r,0,m}+m(-1)^n\zeta_{\bar 1}^{m-2}\theta_{r,1}\alpha)
		=(m-1)\zeta_{\bar 1}^{m-2} (-1)^n\left( m\sigma_{r,1}+(n-r)\tau_{r,1} \right)\alpha.
	\end{align*}
\end{proof}

\begin{remark}
	\label{remark:holomorphicDomega0}
	As in \autoref{remark:holomorphicDomegakm}, we may define the relevant forms for arbitrary $m\in\CC$ on the open set of $S\RR^n$ where $\zeta_1\notin (-\infty,0]$. In this case \autoref{lem:DintrinsicVols} holds for all $m\in\mathbb{C}$.
\end{remark}

\begin{corollary}
	\label{cor:MovingM_intrinsicVolume}
	For $1\leq r\leq n-1$ and $m\in\NN_0$,
	\begin{align*}
		\mathcal{L}_{\widetilde{Y_{11}}}D\omega_{r,0,m}=(m-1)\frac{n+m-r}{m+1}D\omega_{r,0,m+2}=-(m-1)D\omega_{r,1,m+2}.
	\end{align*}
\end{corollary}
\begin{proof}
	This can in principle be calculated as in \autoref{sec:Ykkaction}, however, we will present a shorter argument. As mentioned in \autoref{remark:holomorphicDomegakm}, \autoref{remark:actionY11omegak} and \autoref{remark:holomorphicDomega0}, we can define the relevant forms on the open subset of $S\RR^n$ given by $\zeta_1\notin (-\infty,0]$ for arbitrary $m\in\CC$ and then the formulas for the Rumin differential in Eq.~\eqref{eq:DOmrkm} and \autoref{lem:DintrinsicVols} hold for all $m\in\CC$. Combining both formulas, we obtain
	\begin{align*}
		D(\zeta_{\bar1}^m\omega_{r,0})=&(-1)^n(m-1)\zeta_{\bar1}^{m-2}\left[m\sigma_{r,1}+(n-r)\tau_{r,1}\right]\alpha\\
		=&-\frac{m-1}{n+m-r-2}D(\zeta_{\bar1}^{m-2}\omega_{r,1})
	\end{align*}
	for all $m\in \CC\setminus\{-(n-r-2)\}$ for $\zeta_{\bar 1}\notin (-\infty,0]$. \autoref{remark:actionY11omegak} thus implies that
	\begin{align*}
		\mathcal{L}_{\widetilde{Y_{11}}}D(\zeta_{\bar1}^m\omega_{r,0})=&-\frac{m-1}{n+m-r-2}\mathcal{L}_{\widetilde{Y_{11}}}D(\zeta_{\bar1}^{m-2}\omega_{r,1})\\
		=&-(m-1)D(\zeta_{\bar1}^{m}\omega_{r,1})\\
		=&(m-1)\frac{n+m-r}{m+1}D(\zeta_{\bar1}^{m+2}\omega_{r,0})
	\end{align*}
	for $m\in \CC\setminus\{-(n-r-2),-1\}$ and $\zeta_{\bar 1}\notin (-\infty,0]$. In particular, the equations hold for $m\in\NN_0$ for $\zeta_{\bar1}\notin (-\infty,0]$, and then for all $\zeta_{\bar1}$ by continuity.
\end{proof}

%%%%%%%%%%%%%%%%%%%%%%%%%%%%%%%%%%%%%%%

\section{A new proof of Alesker's Irreducibility theorem}\label{sec:prfIrredThm}
This section contains the proofs of \autoref{mthm:imDIrred} and \autoref{mcor:ValslIrred}. We first recall some further background on infinite dimensional representations. Then we use \autoref{mthm:smoothValsByDiffform} and results from \cite{Bernig2007} to reduce the problem to showing that certain spaces of differential forms are algebraically irreducible $(\sln(n),\SO(n))$-modules. For these spaces, we first survey some known results, including the decomposition into $\SO(n)$-types and corresponding highest weight vectors in \autoref{sec:prelimImD}, and then use the calculations from \autoref{sec:actLieAlgdf} to show that they are algebraically irreducible $(\sln(n),\SO(n))$-modules in \autoref{sec:IrredDiffForm}. 

\subsection{Some remarks on infinite dimensional representations}
\label{sec:infDimRep}

We call a representation $\pi$ of a Lie group $G$ on a locally convex vector space $E$ continuous if the map
\begin{align*}
	G\times E&\rightarrow E\\
	(g,v)&\mapsto \pi(g)v
\end{align*}
is continuous. Then $v\in E$ is called a smooth vector if $G\to E$, $g\mapsto \pi(g)v$, is a smooth map. We will denote the space of smooth vectors by $E^\infty$. This space is $G$-invariant and naturally equipped with the Garding topology, which is stronger than the topology induced from $E$ and has the property that any element in $E^\infty$ is a smooth vector (see \cite{Warner1972}*{Section~4.4.1}). If $E$ is complete, then $E^\infty$ is dense in $E$ (compare \cite{Warner1972}*{Prop.~4.4.1.1}). Moreover, if $E$ is a Banach space, then $E^\infty$ is a Fr\'echet space. $E^\infty$ can naturally be equipped with an action of the Lie algebra $\mathfrak{g}$ of $G$ by setting
\begin{align*}
	d\pi(X)v:=\frac{d}{dt}\Big|_0\pi(\exp(tX))v
\end{align*}
for $v\in E^\infty$ and $X\in \mathfrak{g}$. In general, there is no direct correspondence between the representation of $G$ on $E$ and the representation of $\mathfrak{g}$ on $E^\infty$ -- for example the closure in $E$ of a $\mathfrak{g}$-invariant subspace of $E^\infty$ is not necessarily $G$-invariant. Nevertheless, the action of $\mathfrak{g}$ on $E^\infty$ may be used to simplify the study of the representation of $G$ on $E$.\\
If $K\subset G$ is a compact subgroup, then a vector $v\in E$ is called $K$-finite if $\{\pi(g)v:g\in K\}$ spans a finite dimensional subspace of $E$. We denote by $E^{K-\mathrm{fin}}$ the subspace of $K$-finite vectors. If $E$ is complete, then 
\begin{align*}
	E_K:=E^{K-\mathrm{fin}}\cap E^\infty
\end{align*}
is dense in $E$ (see  \cite{Warner1972}*{Thm.~4.4.3.1}) and invariant with respect to the representation of $\mathfrak{g}$ on $E^\infty$. Moreover, by construction $(E^\infty)_K=E_K$, so $E_K$ is also dense in $E^\infty$ with respect to the Garding topology. Since it is also invariant with respect to $K$, it thus naturally carries the structure of a $(\mathfrak{g},K)$-module (see \cite{Wallach1988}*{Section~3.3} for the general definition).\\
Recall that a continuous representation of $G$ on a locally convex vector space is called topologically irreducible if every nontrivial $G$-invariant subspace is dense in $E$. Similarly, a $(\mathfrak{g},K)$-module is called algebraically irreducible if it does not contain any nontrivial subspaces that are invariant under the action of both $\mathfrak{g}$ and $K$. \\
The following simple result is well-known. We include the argument for the convenience of the reader.
\begin{lemma}
	\label{lemma:EquivIrredu}
	Let $G$ be Lie group with Lie algebra $\mathfrak{g}$, $K\subset G$ a compact subgroup, and $E$ a continuous representation of $G$ on a complete locally convex vector space. 
	If $E_K$ is an algebraically irreducible $(\mathfrak{g},K)$-module, then $E$ is topologically irreducible.
\end{lemma}
\begin{proof}
	If $F\subset E$ is a nontrivial and closed $G$-invariant subspace, then $F_K$ is a nontrivial $(\mathfrak{g},K)$-invariant subspace of $E_K$ and dense in $F$. Since $E_K$ is algebraically irreducible, $F_K=E_K$, and thus $F\subset E$ is dense, which shows the claim.
\end{proof}
\autoref{lemma:EquivIrredu} is usually applied to so called admissible representations. Let $\widehat{K}$ denote the set of equivalence classes of irreducible finite dimensional representations of $K$. Given an irreducible representation  $\delta\in\widehat{K}$, we let $E[\delta]\subset E^{K-\mathrm{fin}}$ denote the sum of all irreducible $K$-subrepresentations in $E$ equivalent to $\delta$, and call $\overline{E[\delta]}$ the $\delta$-isotypical component of $E$. Then $E_K=\bigoplus_{\delta\in \widehat{K}}E^\infty[\delta]$. We call $E$ a $K$-admissible representation if $\dim E[\delta]<\infty$ for every $\delta\in \widehat{K}$. In this case, $\overline{E[\delta]}=E[\delta]=E^\infty[\delta]$ (compare \cite{Warner1972}*{Cor.~4.4.3.3}).\\
Note that for a $K$-admissible representation, the structure of the $(\mathfrak{g},K)$-module $E_K$ may be understood by investigating how the different components $E[\delta]$ are related under the action of the Lie algebra $\mathfrak{g}$, which for each individual component is a finite dimensional problem. This approach has been used for a variety of groups, compare \cite{Howe1999} and the references therein.

\begin{remark}
	\label{remarks:representationtheory}
	\begin{enumerate}
		\item In general, $E_K$ is not necessarily algebraically irreducible even if $E$ is topologically irreducible. However, if  $G$ is a real reductive Lie group with associated maximal compact subgroup $K$, and $E$ is a $K$-admissible representation on a Banach space, these two notions are in fact equivalent, compare \cite{Warner1972}*{Thm. 4.5.5.4}. 
		\item This applies in particular to the representation of $G=\GL(n,\RR)$ on $\Val_r^\pm(\RR^n)$. In this case $K=\mathrm{O}(n)$, and the fact that $\Val^{\pm}_r(\RR^n)$ is admissible can be deduced from the existence and properties of the Goodey--Weil distributions, compare \cite{Alesker2001}*{Prop.~2.10}.
		\item If $G$ is a real reductive Lie group with associated maximal compact subgroup $K$, then an algebraically irreducible $(\mathfrak{g},K)$-module is always $K$-admissible, compare \cite{Wallach1988}*{Cor.~3.4.8}. In contrast, a topologically irreducible Banach representation of a real reductive group is not necessarily $K$-admissible, see \cite{Soergel1988}.
		\item The proof of \autoref{thm:AleskerIrredThm} given in \cite{Alesker2001} uses the Beilinson--Bernstein localization theorem to establish the algebraic irreducibility of the $(\gl(n),\mathrm{O}(n))$-module $\Val_r^{\pm}(\RR^n)^{\mathrm{O}(n)-\mathrm{fin}}$, which directly implies that   $\Val^\pm_r(\RR^n)$ is topologically irreducible (due to \autoref{lemma:EquivIrredu}). However, the same reasoning shows that the corresponding space of smooth valuations is topologically irreducible (with respect to the Garding topology). 
		\item More generally, if $E$ is a continuous representation of $G$ on a complete locally convex vector space, then $E$ is topologically irreducible if and only if $E^\infty$ is topologically irreducible with respect to the Garding topology, compare the discussion in \cite{Warner1972}*{Section~4.4.1}. In particular, the topological irreducibility of $\Val_r^\pm(\RR^n)$ is equivalent to the topological irreducibility of the corresponding space of smooth valuations (with respect to the Garding topology).
	\end{enumerate}
\end{remark}

\subsection{Preliminaries on $V_\bullet$ and $\mathcal{V}^{\infty,\mathrm{tr}}$}\label{sec:prelimImD}
Recall that we defined the space 
\begin{align*}
	V_r=\im (D:\Omega^{r,n-1-r}(S\RR^n)^{\mathrm{tr}}\rightarrow\Omega^{r,n-r}(S\RR^n)^{\mathrm{tr}})
\end{align*}
in \autoref{sec:pairing}. Then $V_r\subset (\im D)^{\mathrm{tr}}_{r,n-r}$. The reason to consider this space comes from the following result due to Bernig and Bröcker from \cite{Bernig2006}. Let $\mathcal{V}^{\infty,\mathrm{tr}}_r\subset \Val_r(\RR^n)$ denote the subspace of all $r$-homogeneous valuations in $\Val(\RR^n)$ that are representable by integration with respect to the normal cycle.
\begin{theorem}[\cite{Bernig2007}*{Thm.~3.3}]
	\label{thm:BernigBroeckerImageRumin}
	\begin{enumerate}
		\item For $1\leq r\leq n-1$, there is an injective map
		\begin{align*}
			\mathcal{V}^{\infty,\mathrm{tr}}_r\rightarrow (\im D)^{\mathrm{tr}}_{r,n-r}.
		\end{align*}
		\item For $2\leq r\leq n-1$, this map is also surjective, i.e., an isomorphism.
		\item For $r=1$, the above map induces an isomorphism
		\begin{align*}
			\mathcal{V}^{\infty,\mathrm{tr}}_1\cong \left\{f\wedge\alpha\wedge \vol_{\S^{n-1}}: f\in C^\infty(\S^{n-1}), \int_{\S^{n-1}}v f(v)d\mathcal{H}^{n-1}(v)=0\right\},
		\end{align*}
		where $\vol_{\S^{n-1}}$ denotes the standard volume form on $\S^{n-1}$.
		\item For every $\varphi\in \mathcal{V}^{\infty,\mathrm{tr}}_r$ there exists a differential form $\omega\in \Omega^{r,n-1-r}(S\RR^n)^{\mathrm{tr}}$ such that $\varphi(K)=\int_{\nc(K)}\omega$ for all $K\in\mathcal{K}(\RR^n)$.
	\end{enumerate}
\end{theorem}
\begin{remark}
	In \cite{Bernig2006} the result is stated for $\Val^\infty_r(\RR^n)$ instead of $\mathcal{V}^{\infty,\mathrm{tr}}_r$, however, the proofs only use that the given valuations are representable by integration with respect to the normal cycle. Although \autoref{mthm:smoothValsByDiffform} (which is proved in \autoref{sec:RepSmoothVal} without relying on the Irreducibilty \autoref{thm:AleskerIrredThm}) implies that this is equivalent, we prefer to state the result without using this equivalence in order to clearly distinguish these notions.
\end{remark}
The maps above are induced by associating to a valuation $\varphi=\int_{\nc(\cdot)}\omega$ the differential form $D\omega$, which is well defined due \autoref{thm:BernigBroeckerKernel}. In particular, \autoref{thm:BernigBroeckerImageRumin} shows that a translation invariant valuation in $\mathcal{V}^{\infty,\mathrm{tr}}$ admits an integral representation with a translation invariant form.
Let us reformulate it in the following way.
\begin{corollary}\label{cor:isomorphismV}
	There exists an isomorphism $S:\mathcal{V}^{\infty,\mathrm{tr}}_{r}\rightarrow V_r$ such that that for all $g\in\GL(n,\RR)$ and $\varphi\in\mathcal{V}^{\infty,\mathrm{tr}}_r$,
	\begin{align*}
		S(g\cdot \varphi)=\sign(\det g)\ G_{g^{-1}}^\ast S(\varphi).
	\end{align*}
	
\end{corollary}
\begin{proof}
	\autoref{thm:BernigBroeckerImageRumin} shows that the image of $\mathcal{V}_r^{\infty,\mathrm{tr}}$ under the map induced by $D$ is contained in $V_r$. By construction, this map is surjective, so \autoref{thm:BernigBroeckerImageRumin} shows that it is an isomorphism. In order to show the stated equivariance with respect to the action of $\GL(n,\RR)$ on both spaces, note that $D$ commutes with the pullback by $G_{g^{-1}}$ since this is a contactomorphism (compare the discussion in \autoref{sec:GLRnrepOnVal}). Thus the statement follows from Eq.~\eqref{eq:actGLOnVal}.
\end{proof}

\begin{remark}
	\autoref{thm:BernigBroeckerImageRumin} shows that $(\im D)^{\mathrm{tr}}_{r,n-r}=V_r$ for $2\leq r\leq n-1$. For $r=1$, $V_1\subset (\im D)^{\mathrm{tr}}_{1,n-1}$  is a proper subspace with $(\im D)^{\mathrm{tr}}_{1,n-1}/V_1\cong \CC^n$. 
\end{remark}
The decomposition of the spaces $\mathcal{V}^{\mathrm{tr}}_{r,n-r}\cong V_r$ into $\SO(n)$-types was obtained in \cite{Alesker2011}.
\begin{theorem}[\cite{Alesker2011}*{Thm.~1}]\label{thm:AleskerBernigSchuster}
	Let $n\ge 3$, $l =\lfloor\frac{n}{2}\rfloor$ and  $1 \leq r \leq l$. The nontrivial $\SO(n)$-types in $\mathcal{V}^{\infty,\mathrm{tr}}_r\cong V_r$ and $\mathcal{V}^{\infty,\mathrm{tr}}_{n-r}\cong V_{n-r}$ are the same and given by the following set of highest weights:
	%	
	%	The nontrivial $\SO(n)$-types in $\Val_r(\RR^n)$ and $\Val_{n-r}(\RR^n)$ are the same and given by the following set of highest weights:
	\begin{align}\label{eq:ABSweights}
		\begin{cases}
			\{\lambda_{k,m}: \, m \geq 2, 1 \leq k \leq r\} \cup \{0\} \cup \{\lambda_{-l,m}: \, m \geq 2\},& \text{ if } n=2l=2r,\\
			\{\lambda_{k,m}: \, m \geq 2, 1 \leq k \leq r\} \cup \{0\},& \text{ otherwise},
		\end{cases}
	\end{align}
	where 
	\begin{align*}
		\lambda_{k,m} = \begin{cases}
			(m, \underbrace{2, \dots, 2}_{k-1}, 0, \dots, 0) \in \ZZ^l & \text{ for } k=1, \dots, l,\\
			(m, 2 \dots, 2, -2) \in \ZZ^l & \text{ for } k=-l.
		\end{cases}
	\end{align*}
	Moreover, each appears with multiplicity one.
\end{theorem}

Recently, this result was refined in \cite{Kotrbaty2022} by explicitly determining the highest weight vectors of the nontrivial $\SO(n)$-types. For $\mathcal{V}_r^{\infty,\mathrm{tr}}$, the corresponding valuations are induced by the differential forms $\omega_{r,k,m}$ defined in \eqref{eq:defOmegaRkm} and $\omega_{l,-l, m}$ defined in \eqref{eq:defOmegalmlm}. We use the isomorphism in \autoref{cor:isomorphismV} to restate the result.
\begin{theorem}[\cite{Kotrbaty2022}*{Thm.~1.3}]\label{thm:KotrbatyWannererHighestWeights}
	For any $r,k,m \in \NN$ with $1\le r \leq n-1$, $1\le k \leq \min\{r,n-r\}$, and $m \geq 2$, as well as for $(r,k,m) = (l,-l,m)$ if $n=2l$ is even, the differential form 
	\begin{align*}
		D\omega_{r,k,m}
	\end{align*}%\frac{\sqrt{-1}^{\lfloor\frac{n}{2}\rfloor}\sqrt{2}^{m-2}}{s_{n+m-r-3}}
	is a nontrivial highest weight vector of weight $\lambda_{k,m}$ of the $\SO(n)$-representation $V_r$. 
\end{theorem}
Let us point out that the weight $\lambda = 0$ is not covered by \autoref{thm:KotrbatyWannererHighestWeights}. As the authors remark in \cite{Kotrbaty2022}, this weight corresponds to $\SO(n)$-invariant valuations, that is, by Hadwiger's characterization~\cite{Hadwiger1957}, to the well-studied family of intrinsic volumes. For the sake of completeness, we provided a description by differential forms of the intrinsic volumes in \autoref{thm:HWVintrinsicVolumes} and \autoref{lem:DintrinsicVols}.

Set 
\begin{align*}
	V_r^{\pm}:=\{\tau\in V_r:G_{-\id}^\ast \tau=\pm (-1)^n \tau\}.
\end{align*}
Under the isomorphism $V_r\cong \mathcal{V}_r^{\infty,\mathrm{tr}}$ from \autoref{cor:isomorphismV}, this subspace corresponds to even/odd valuations in $\mathcal{V}_r^{\infty,\mathrm{tr}}$, and we will consequently call differential forms belonging to these spaces even/odd. In particular, these are $\GL(n,\RR)$-invariant subspaces of $V_r$. Note further that the $\SO(n)$-type with highest weight $\lambda_{k,m}$ contains only even (resp.\ odd) forms if $m$ is even (resp.\ odd). Indeed, this follows from \autoref{thm:KotrbatyWannererHighestWeights} and the fact that $G_{-\id}^\ast \omega_{r,k,m} = (-1)^{n+m-2} \omega_{r,k,m}$. In particular, \autoref{thm:AleskerBernigSchuster} shows that $V_r^\pm$ is a multiplicity free representation of $\SO(n)$ and that the $\SO(n)$-types occurring in $V_r^+$ and $V_r^-$ are indexed by the subset of Eq.~\eqref{eq:ABSweights} of weights $\lambda$ such that $\lambda_1$ is even and odd respectively.

\subsection{Algebraic irreducibility of $(V_r^\pm)^{\SO(n)-\mathrm{fin}}$}
\label{sec:IrredDiffForm}
In this section we show that the $\SO(n)$-finite vectors of $V_r^\pm$, $1 \leq r \leq n-1$, form an algebraically irreducible $(\sln(n), \SO(n))$-module, where we assume $n\ge 3$ throughout the section. In the proof, we will use the computations from \autoref{sec:actLieAlgdf} to relate the $\SO(n)$-types using the action of $\sln(n)_\CC$ on the highest weight vectors from \autoref{thm:KotrbatyWannererHighestWeights}. Since \autoref{thm:AleskerBernigSchuster} provides the decomposition into $\SO(n)$-types of these spaces, this is sufficient to determine whether this space is algebraically irreducible as an $(\sln(n),\SO(n))$-module. The pairing considered in \autoref{sec:pairing} will be used in combination with the following result, which is a simple consequence of Schur's Lemma.
\begin{lemma}
	\label{lem:NonvanishingPairingImpliesSubset}
	Let $W$ be a direct sum of irreducible representations of a compact Lie group $K$, $U$ an irreducible representation of $K$, and $(\cdot,\cdot):W\times U\rightarrow\CC$ a $K$-invariant sesquilinear form. If $(\cdot,\cdot)$ is nontrivial, then $W$ contains an irreducible subrepresentation equivalent to $U$.
\end{lemma}

\begin{theorem}\label{thm:imDslnSonirred}
	For $1\leq r\leq n-1$, the space $\left(V_r^\pm\right)^{\SO(n)-\mathrm{fin}}$ is an algebraically irreducible $(\sln(n),\SO(n))$-module.
\end{theorem}
\begin{proof}
	Let $W\subset \left(V_r^\pm\right)^{\SO(n)-\mathrm{fin}}$ be an $(\sln(n),\SO(n))$-invariant subspace. Then $W=\bigoplus_{\lambda\in \Lambda(W)}W[\lambda]$ for a nonempty subset $\Lambda(W)\subset \Lambda^\pm_r$, where $\Lambda^+_r$ consists of all weights $\lambda = (\lambda_1, \dots, \lambda_l)$, $l = \left\lfloor \frac{n}{2}\right\rfloor$, contained in the index set from \eqref{eq:ABSweights} where $\lambda_1$ is even, and $\Lambda^-_r$ to those where $\lambda_1$ is odd.
%	for $l = \left\lfloor \frac{n}{2}\right\rfloor$
%	\begin{align*}
%		\Lambda^+_r:=&\begin{cases}
%			\left\{\lambda_{k,m}:1\leq k\leq r,m\ge 2~\text{even}\right\}\cup \{0\}\cup\{\lambda_{-l,m}:m\ge 2~\text{even}\}, & \text{if}~n=2l=2r,\\
%			\left\{\lambda_{k,m}:1\leq k\leq r,m\ge 2~\text{even}\right\}\cup \{0\},& \text{otherwise},
%		\end{cases}
%	\end{align*}
%	and
%	\begin{align*}
%		\Lambda^-_r:=&\begin{cases}
%			\left\{\lambda_{k,m}:1\leq k\leq r,m\ge 2~\text{odd}\right\}\cup\{\lambda_{-l,m}:m\ge 2~\text{odd}\}, & \text{if}~n=2l=2r,\\
%			\left\{\lambda_{k,m}:1\leq k\leq r,m\ge 2~\text{odd}\right\},& \text{otherwise},
%		\end{cases}
%	\end{align*}
	Since, by \autoref{thm:AleskerBernigSchuster}, $\left(V_r^\pm\right)^{\SO(n)-\mathrm{fin}}$ is a multiplicity free representation of $\SO(n)$ with highest weights indexed by $\Lambda_r^\pm$, the claim follows if we can show that $\Lambda(W)=\Lambda^\pm_r$.

	To this end, we will show the following statements, where $\kappa := \min\{r,n-r\}$:
	\begin{enumerate}
		\item\label{it:moveInmUp} If $\lambda_{k,m}\in \Lambda(W)$ for $1\leq k\leq \kappa$, $m\ge 2$, then $\lambda_{k,m+2}\in\Lambda(W)$.
		\item\label{it:moveInkUp} If $\lambda_{k,m}\in \Lambda(W)$ for $1\leq k<\kappa$, $m\ge 2$, then $\lambda_{k+1,m}\in\Lambda(W)$.
		\item\label{it:moveInmDown} If $\lambda_{k,m+2}\in \Lambda(W)$ for $1\leq k\leq \kappa$, $m\ge 2$, then $\lambda_{k,m}\in\Lambda(W)$.
		\item\label{it:moveInkDown} If $\lambda_{k+1,m}\in \Lambda(W)$ for $1\leq k< \kappa$, then $\lambda_{k,m}\in\Lambda(W)$. 
%		\item\label{it:movemltol} If $n=2l=2r$, then $\lambda_{l,m}\in\Lambda(W)$ if and only if $\lambda_{-l,m}\in\Lambda(W)$.
		\item\label{it:movemltol} If $n=2l=2r$ and $l \geq 2$, then $\lambda_{-l,m}\in\Lambda(W)$ if and only if $\lambda_{l-1,m}\in\Lambda(W)$.
	\end{enumerate}
	In the even case, that is, when $\Lambda(W) \subset \Lambda^+_r$, we will also show the following:
	\begin{enumerate}
		\setcounter{enumi}{5}
		\item\label{it:moveintrUp} If $0\in \Lambda(W)$, then $\lambda_{1,2}\in\Lambda(W)$.
		\item\label{it:moveintrDown} If $\lambda_{1,2}\in \Lambda(W)$, then $0\in\Lambda(W)$.
	\end{enumerate}
	Clearly, any nonempty subset of $\Lambda_r^\pm$ with these properties must coincide with the whole set. Let us show these claims.
	\begin{enumerate}
		\item[\eqref{it:moveInmUp}] If $\lambda_{k,m}\in \Lambda(W)$ for $1\leq k\leq \kappa$, $m\ge 2$, then $Dw_{r,k,m}\in W$, since this is the unique highest weight vector corresponding to this representation, compare \autoref{thm:KotrbatyWannererHighestWeights}. Since $W$ is in particular $\sln(n)_\CC$-invariant, and the Lie derivative and the Rumin differential commute, \autoref{cor:LieY11onOmegarkm} shows that
		\begin{align*}
			\LieDer_{\widetilde{Y}_{11}} D \omega_{r,k,m} = (n-r+m-2)D\omega_{r,k,m+2} \in W.
		\end{align*}
		As $n-r\ge 1$ and $m\ge 2$, this implies $D\omega_{r,k,m+2}\in W$ and thus the corresponding $\SO(n)$-type is contained in $W$. Hence, $\lambda_{k,m+2}\in\Lambda(W)$.
		
		\item[\eqref{it:moveInkUp}] By \autoref{lem:pairFormWellDef}, we have a well-defined paring $(\cdot,\cdot)$ between $V_r^\pm$ and $V_{n-r}^\pm$, which is $\SL(n,\RR)$-invariant by \autoref{prop:pairingFormsSLinv}. In particular, it is $\SO(n)$-invariant.

		If $\lambda_{k,m}\in \Lambda(W)$ for $1\leq k< \kappa$, $m\ge 2$, then $D\omega_{r,k,m}\in W$, so $\LieDer_{\widetilde{Y}_{k+1,k+1}}D\omega_{r,k,m}\in W$, and \autoref{cor:pairKp1NotZero} shows that
		\begin{align*}
			\left(\LieDer_{\widetilde{Y}_{k+1,k+1}} D\omega_{r,k,m}, D\omega_{n-r,k+1,m}\right) \neq 0.
		\end{align*}
		Thus, since $D\omega_{n-r,k+1,m}$ is the highest weight vector of an irreducible $\SO(n)$-representation with weight $\lambda_{k+1,m}$ (in $V_{n-r}^\pm$), \autoref{lem:NonvanishingPairingImpliesSubset} implies that the $\SO(n)$-type with weight $\lambda_{k+1,m}$ belongs to $W$, so $\lambda_{k+1,m}\in \Lambda(W)$.
		
		\item[\eqref{it:moveInmDown}] By assumption, $D\omega_{r,k,m+2} \in W$ and therefore also $\LieDer_{\widetilde{\overline{Y_{11}}}} D \omega_{r,k,m+2} \in W$. As, by \autoref{cor:LieY11onOmegarkm}, $\LieDer_{\widetilde{Y}_{11}}D \omega_{n-r,k,m} = (r+m-2)D\omega_{n-r,k,m+2}$, we can apply \autoref{cor:pairFormsSLnInvLiealg} to obtain
		\begin{align*}
			\left(\LieDer_{\widetilde{\overline{Y_{11}}}} D \omega_{r,k,m+2}, D\omega_{n-r,k,m}\right) =& -\left(D \omega_{r,k,m+2}, \LieDer_{\widetilde{Y}_{11}}  D\omega_{n-r,k,m}\right)\\
			=&-(r+m-2) \left(D \omega_{r,k,m+2}, D\omega_{n-r,k,m+2}\right) \neq 0.
		\end{align*}
		where we used \cite{Kotrbaty2022}*{Thm.~6.1} in the last step (see also \autoref{lem:pairKWForms}). Thus, the pairing is nontrivial and \autoref{lem:NonvanishingPairingImpliesSubset} implies as in the proof of \eqref{it:moveInkUp} that the $\SO(n)$-type with weight $\lambda_{k,m}$ belongs to $W$, so $\lambda_{k,m}\in \Lambda(W)$.

		\item[\eqref{it:moveInkDown}] We will use the same reasoning as in \eqref{it:moveInmDown}. By assumption, $D\omega_{r,k+1,m} \in W$ and, hence, also $\LieDer_{\widetilde{\overline{Y_{k+1,k+1}}}} D\omega_{r,k+1,m}\in W$. By the invariance of the pairing (\autoref{cor:pairFormsSLnInvLiealg}), \autoref{lem:pairFormWellDef}, and \autoref{cor:pairKp1NotZero}, we thus have
		\begin{align*}
			\left(\LieDer_{\widetilde{\overline{Y_{k+1,k+1}}}} D\omega_{r,k+1,m}, D\omega_{n-r,k,m}\right) =& - \left( D\omega_{r,k+1,m},\LieDer_{\widetilde{Y}_{k+1,k+1}} D\omega_{n-r,k,m}\right)\\
			=& -(-1)^{n}\overline{\left(\LieDer_{\widetilde{Y}_{k+1,k+1}} D\omega_{n-r,k,m},D\omega_{r,k+1,m}\right)} \neq 0.
		\end{align*}
		Thus, as before, \autoref{lem:NonvanishingPairingImpliesSubset} implies that the $\SO(n)$-type with weight $\lambda_{k,m}$ belongs to $W$, so $\lambda_{k,m}\in \Lambda(W)$.
		
%		\item[\eqref{it:movemltol}] If $R\in\mathrm{O}(n)$ denotes the reflection along the hyperplane $e_n^\perp$, then $R^*\omega_{l,l,m}=\omega_{l,-l,m}$. Since $D$ commutes with $R$, this implies $R^*D\omega_{l,l,m}=D\omega_{l,-l,m}$. As we assume that $W$ is $\mathrm{O}(n)$-invariant, this shows that $\lambda_{l,m}\in\Lambda(W)$ if and only if $\lambda_{-l,m}\in \Lambda(W)$.\todo{ersetzen durch sln}\todo{$\omega_{l,-l, m}$ Fall ersetzen}
		\item[\eqref{it:movemltol}] Note that, by \autoref{cor:pairFormsSLnInvLiealg} and \autoref{prop:pairYllOmlml},
		\begin{align*}
			-\left(D\omega_{l,-l,m}, \LieDer_{\widetilde{\overline{Y_{l,l}}}} D\omega_{l,l-1,m}\right) =& \left(\LieDer_{\widetilde{Y}_{l,l}} D\omega_{l,-l,m}, D\omega_{l,l-1,m}\right)\\
			 =& -(-1)^{n}\overline{\left(\LieDer_{\widetilde{Y}_{l,l}}D\omega_{l,l-1,m},D\omega_{l,l,m}\right)}.
		\end{align*}
		Since the right hand side of this equation is non-zero by \autoref{cor:pairKp1NotZero}, the same reasoning as in \eqref{it:moveInkUp} using \autoref{lem:NonvanishingPairingImpliesSubset} shows \eqref{it:movemltol}.
		
		\item[\eqref{it:moveintrUp}] This follows with the same reasoning as \eqref{it:moveInmUp} using \autoref{cor:MovingM_intrinsicVolume}.
		
		\item[\eqref{it:moveintrDown}] This follows with the same reasoning as \eqref{it:moveInmDown} using \autoref{cor:MovingM_intrinsicVolume}.
	\end{enumerate}
	This concludes the proof.
\end{proof}
%The following is now a simple consequence of \autoref{lemma:EquivIrredu} and \autoref{thm:imDslnSonirred}.
%\begin{corollary}
%	For $1\leq r\leq n-1$, $V_r$ is an irreducible representation of $\SL(n,\RR)$.
%\end{corollary}

\subsection{Alesker's Irreducibility theorem}\label{sec:prfAlirred}

\begin{proof}[Proof of \autoref{mthm:imDIrred}]
	Recall that $n\ge 3$ and fix $1 \leq r \leq n-1$. By \autoref{cor:isomorphismV}, we have an $\SL(n,\RR)$-equivariant isomorphism $\mathcal{V}^{\infty,\mathrm{tr}}_r\cong V_r$. In particular, $\mathcal{V}^{\infty,\mathrm{tr}}_r$ is an admissible representation of $\SO(n)$, and this map restricts to an isomorphism of $(\sln(n),\SO(n))$-modules \begin{align*}
		\left(\mathcal{V}^{\infty,\mathrm{tr}}_{r,\pm}\right)^{\SO(n)-\mathrm{fin}}\cong (V_r^\pm)^{\SO(n)-\mathrm{fin}}.
	\end{align*}
	Since the space on the right hand side is an algebraically irreducible $(\sln(n),\SO(n))$-module by \autoref{thm:imDslnSonirred}, so is the space on the left hand side.
\end{proof}

\begin{proof}[Proof of \autoref{mcor:ValslIrred}]
	By \autoref{mthm:smoothValsByDiffform}, $\Val^\infty(\RR^n)=\mathcal{V}^{\infty,\mathrm{tr}}$, so 
	\begin{align}
		\label{eq:lemValsmoothForRemark}
		\Val^\infty(\RR^n)\cap \Val_r^{\pm}(\RR^n)^{\SO(n)-\mathrm{fin}}=\left(\mathcal{V}_{r,\pm}^{\infty,\mathrm{tr}}\right)^{\SO(n)-\mathrm{fin}}.
	\end{align}
	Since this is an algebraically irreducible $(\sln(n),\SO(n))$-module by \autoref{mthm:imDIrred}, \autoref{lemma:EquivIrredu} shows that $\Val_r^\pm(\RR^n)$ is a topologically irreducible representation of $\SL(n,\RR)$.
\end{proof}

\begin{remark}
	Since $\Val(\RR^n)$ is an admissible representation of $\SO(n)$ (as $\Val^\infty(\RR^n)=\mathcal{V}^{\infty,\mathrm{tr}}$ is admissible), $\Val(\RR^n)^{\SO(n)-\mathrm{fin}}\subset \Val^\infty(\RR^n)$ as discussed in \autoref{sec:infDimRep} (compare \cite{Warner1972}*{Cor.~4.4.3.3}), so the intersection in Eq.~\eqref{eq:lemValsmoothForRemark} is not really necessary. Since the proof does not need this additional fact, we have omitted it from the argument.
\end{remark}

Let us add some comments on the case $n=2$, $k=1$. By a result due to McMullen \cite{McMullen1980}, $\Val_1(\RR^2)$ is isomorphic as a representation of $\SO(2)$ to the subspace of $C(\S^1)$ of all continuous functions orthogonal to linear functions (with respect to the standard $L^2$ inner product). The irreducible representations of $\SO(2)$ are one dimensional and indexed by their weight $m\in\mathbb{Z}$. In the case of $C(\S^1)$, the corresponding isotypical components are spanned by $\zeta_{\bar1}^m$ for $m\in\ZZ$. In particular, $\Val_1(\RR^2)$ is a multiplicity free representation of $\SO(2)$ and the nontrivial $\SO(2)$-types are given by the weights $m\in \ZZ\setminus\{\pm1\}$. Using the forms from \autoref{sec:weight0Form}, it is not difficult to see that a nontrivial weight vector is given by the differential form
\begin{align*}
	\tilde{\omega}_{m}:=\sqrt{-1}\zeta_{\bar 1}^m \omega_{1,0}, \quad m\in\mathbb{Z}\setminus\{\pm1\},
\end{align*}
where the additional factor $\sqrt{-1}$ guarantees that $\overline{\tilde{\omega}_m}=\tilde{\omega}_{-m}$, compare the definition of $\omega_{1,0}$ in \autoref{sec:weight0Form}. Using the same arguments as in \autoref{thm:imDslnSonirred}, it is easy to check that
\begin{align*}
	(V^+_1)^{\SO(2)-\mathrm{fin}}=\mathrm{span}\{D\tilde{\omega}_m: m\in \mathbb{Z}~\text{even}\} 
\end{align*}
is an algebraically irreducible $(\sln(2),\SO(2))$-module. In particular, as in the proof of \autoref{mcor:ValslIrred}, \autoref{lemma:EquivIrredu} can be used to shows that $\Val^+_1(\RR^2)$ is a topologically irreducible representation of $\SL(2,\RR)$.\\

For odd valuations, the situation is different: In this case, the same reasoning as in \autoref{thm:imDslnSonirred} can be used to show that
\begin{align*}
	\mathrm{span}\{D\tilde{\omega}_m: m\in \mathbb{Z}~\text{odd}, m\ge 3\}\subset (V^-_1)^{\SO(2)-\mathrm{fin}},\\
	\mathrm{span}\{D\tilde{\omega}_m: m\in \mathbb{Z}~\text{odd}, m\le -3\}\subset (V^-_1)^{\SO(2)-\mathrm{fin}}
\end{align*}
are two algebraically irreducible $(\sln(2),\SO(2))$-modules. In terms of McMullen's characterization of $\Val_1^-(\RR^2)$ from \cite{McMullen1980}, this corresponds to continuous odd functions on $\S^1$ with Fourier expansion containing either only positive or negative powers of $\zeta_1$. The corresponding spaces of valuations are closed and $\SL(2,\RR)$-invariant, so $\Val^-_1(\RR^2)$ is in particular not topologically irreducible with respect to $\SL(2,\RR)$. However, the orthogonal reflection along $\mathrm{span}(e_1)$ (which belongs to $\mathrm{O}(2)$) acts on $\tilde{\omega}_m$ by complex conjugation. In particular, since $\overline{\tilde{\omega}_m}=\tilde{\omega}_{-m}$ for $m\in\ZZ$, 
\begin{align*}
	(V^-_1)^{\mathrm{O}(2)-\mathrm{fin}}=(V^-_1)^{\SO(2)-\mathrm{fin}}=\mathrm{span}\{D\tilde{\omega}_m: m\in \mathbb{Z}~\text{odd}, m\ne \pm 1\}
\end{align*}
is an algebraically irreducible $(\gl(2),\mathrm{O}(2))$-module. As in the proof of \autoref{mcor:ValslIrred}, this implies that $\Val^-_1(\RR^2)$ is a topologically irreducible representation of $\GL(2,\RR)$. Thus \autoref{thm:AleskerIrredThm} also holds for $n=2$.
\begin{extracalc}
	Congratulations for making it through this article! I guess you understand now why we didn't try to sell this proof as "easier than Alesker's original proof".
\end{extracalc}

\subsection*{Acknowledgments}
The authors would like to thank Martin Rubey for his help in their attempt to use sagemath in the calculations.

\bibliographystyle{abbrv}
\bibliography{./booksLocPolyVal}

\end{document}